\documentclass[12pt]{amsart}
\usepackage{amsmath}
\usepackage{amsxtra}
\usepackage{amscd}
\usepackage{amsthm}
\usepackage{amsfonts}
\usepackage{amssymb}
\usepackage{eucal}

\usepackage{xcolor}

\input prepictex
\input pictex
\input postpictex
\usepackage{mathptm}

\usepackage{verbatim}

\newtheorem{theorem}{Theorem}[section]

\newtheorem{lemma}[theorem]{Lemma}

\theoremstyle{definition}
\newtheorem{definition}[theorem]{Definition}

\theoremstyle{remark}
\newtheorem{remark}[theorem]{Remark}

\numberwithin{equation}{section}

\newcommand{\R}{{\mathbb R}}
\newcommand{\C}{{\mathbb C}}
\newcommand{\Z}{{\mathbb Z}}
\newcommand{\N}{{\mathbb N}}

\newcommand{\fg}{{\mathfrak g}}
\newcommand{\fh}{{\mathfrak h}}
\newcommand{\fb}{{\mathfrak b}}
\newcommand{\fl}{{\mathfrak l}}
\newcommand{\fn}{{\mathfrak n}}
\newcommand{\fp}{{\mathfrak p}}
\newcommand{\fu}{{\mathfrak u}}
\newcommand{\fr}{{\mathfrak r}}
\newcommand{\fs}{\mathfrak s}

\newcommand{\id}{{\text{id}}}

\newcommand{\0}{{\bar 0}}
\newcommand{\1}{{\bar 1}}

\newcommand{\U}{{{\rm U}}}
\newcommand{\Uq}{{{\rm U}_q}}

\begin{document}
\title[Serre presentations of Lie superalgebras]
{Serre presentations of Lie superalgebras}

\author{R.B. Zhang}
\thanks{Supported by the Australian Research Council.}
\address{School of Mathematics and Statistics,
University of Sydney, Sydney, Australia}
\email{ruibin.zhang@sydney.edu.au}

\begin{abstract}
An analogue of Serre's theorem is established for finite dimensional simple Lie superalgebras, which describes presentations in terms of Chevalley generators and Serre type relations relative to all possible choices of Borel subalgebras. The proof of the theorem is conceptually transparent; it also provides an alternative approach to Serre's theorem for ordinary Lie algebras.
\end{abstract}

\date{January 12, 2011}

\maketitle

%
%
%
%

%
%

\section{Introduction}\label{introduction}
\subsubsection{}
A well known theorem of Serre gave presentations of finite
dimensional semi-simple Lie algebras in terms of Chevalley
generators and Serre relations.  It was generalised to Kac-Moody
algebras with symmetrisable Cartan matrices by Gabber and Kac
\cite{GK}. The theorem and its generalisation now provide the standard
method to present simple Lie algebras and Kac-Moody
algebras \cite{Kac2}, as well as the associated
quantised universal enveloping algebras \cite{Dr, Jim}.

A natural question is how to present simple contragredient Lie superalgebras
(i.e., Lie superalgebras with Cartan matrices) in a similar way.
Surprisingly
this was only seriously studied after quantised
universal enveloping superalgebras \cite{BGZ90}
had become popular in the early 90s because of
their applications in a variety of areas
such as low dimensional topology \cite{LGZ93, Z95},
statistical physics \cite{BGZ90} and noncommutative
geometry \cite{Man88, Z98, Z04}.

In the Lie superalgebra setting, unconventional higher order relations \cite{LS}
are required beside the usual Serre relations, and their origin is somewhat mysterious. Since a Serre type presentation is always given relative to a chosen Borel subalgebra, the issue is further complicated by the fact
\cite{Kac1, Kac2} that a simple contragredient Lie superalgebra admits
classes of Borel subalgebras, which are not Weyl group conjugate.

\subsubsection{}

At the present, investigation on Serre type presentations for
Lie superalgebras is still rather incomplete even in the finite dimensional
case. Presentations relative to many non-distinguished Borel subalgebras
of such Lie superalgebras have never been constructed (see Remark \ref{new}). The crucial question on whether the Serre type relations obtained so far are complete (i.e., whether they are all the defining relations needed for the Lie superalgebras under consideration) has not been answered satisfactorily.
Therefore, there is the need of a systematic treatment of Serre presentations for the finite dimensional simple contragredient Lie superalgebras, and this paper aims to provide such a treatment.

\subsubsection{}
It was Leites and Serganova \cite{LS} who first obtained the higher order Serre relations for $\mathfrak{sl}_{m|n}$ relative to the so-called
distinguished Borel subalgebra (for which the simple roots
are the easiest to describe).
The corresponding quantum relations
for $\Uq(\mathfrak{sl}_{m|n})$ were constructed in \cite{Sch93,
FLV}. Yamane \cite{Y1} wrote down higher order quantum Serre
relations for quantised universal enveloping superalgebras of finite
dimensional simple Lie superalgebras for the distinguished
and some (but not all) non-distinguished Borel subalgebras.
In the ensuing years, much further work was done to find Serre type
relations for Lie superalgebras by Leites and collaborators
\cite{GL, GLP, BGLL} and by Yamane \cite{Y2}.

References \cite{GL, GLP} and \cite{Y1, Y2} represent the current state of the problem of constructing Serre type presentations for the finite dimensional simple contragredient Lie superalgebras. [Reference \cite{Y2} is largely on affine superalgebras.] However, the papers \cite{Y1, Y2} left out presentations of exceptional simple Lie superalgebras relative to non-distinguished Borel subalgebras. Reference \cite{GL} in principle treated all the Dynkin diagrams which could potentially require higher order Serre relations, but the relations in \cite{GL} and \cite{Y1, Y2} look very different and it is not clear at all whether they are equivalent.

\subsubsection{}
The problem on whether the Serre type relations constructed were complete was only investigated by computer calculations. According to \cite[\S 1]{GL}, completeness of the relations of \cite{GL} was verified by computers for finite dimensional simple contragredient Lie superalgebras, but a conceptual proof is lacking. The problem is open for the Serre type relations given in \cite{Y1, Y2},  and so is also in the infinite dimensional case.

We comment that in the cases considered in \cite{Y1},
completeness of the relations can in principle be deduced
from the existence of a non-degenerate
invariant bilinear form between the quantised universal enveloping
superalgebras of the upper and low triangular Borel
subalgebras, by using Geer's result \cite{G} that quantised
universal enveloping superalgebras are trivial deformations.
However, it is a highly complicated matter to establish the
non-degeneracy of the bilinear form even in the case of ordinary
quantised universal enveloping algebras (see, e.g., \cite{L}).
Many of the representation theoretical results
required for proving the non-degeneracy are lacking for quantised universal enveloping superalgebras,
rendering the super case much more difficult.

\subsubsection{}
In this paper, we give a complete treatment of the Serre presentations
of finite dimensional simple contragredient Lie superalgebras, proving an analogue of Serre's theorem relative to all possible choices of Borel subalgebras. Comparing our results with those of \cite{Y1} (in the $q\to 1$ limit), we have many more higher order Serre relations which are necessary, especially in the case of exceptional Lie superalgebras relative to non-distinguished Borel subalgebras. Our method is also different from those in the literature. It in particular automatically shows the completeness of the relations which we construct.

\subsubsection{}

Let us now describe more precisely the results of this paper.
Given a realisation of the Cartan matrix $A=(a_{i j})$ of a simple
contragredient Lie superalgebra with the set of simple roots
$\Pi_\fb=\{\alpha_1, \dots, \alpha_r\}$, we introduce an auxiliary
Lie superalgebra $\tilde\fg$, which is generated by Chevalley
generators  $\{e_i, \ f_i, \ h_i \mid i=1, 2, \dots, r\}$ subject to
quadratic relations only (see Definition \ref{auxiliary},
where more informative notation is used). Let $\fr$
be the $\Z_2$-graded maximal ideal of $\tilde\fg$ that intersects
trivially the Cartan subalgebra spanned by all $h_i$. Then
$L:=\tilde\fg/\fr$ is the simple Lie superalgebra which we
started with in all cases except in type $A(n, n)$ where $L$ is
$\mathfrak{sl}_{n+1|n+1}$ (see Theorem \ref{LA}).

We introduce a $\Z_2$-graded ideal $\fs$ of the auxiliary Lie
superalgebra, which is generated by explicitly given generators. A
main result proved in Theorem \ref{generating} states that
$\fs=\fr$, or equivalently, $\fg:=\tilde\fg/\fs\cong L$.
From this result, we deduce a super analogue of Serre's theorem,
Theorem \ref{Serre-g}, which gives presentations of the finite
dimensional simple contragredient Lie superalgebras relative to all
possible choices of Borel subalgebras.

The completeness of the relations in Theorem \ref{Serre-g} is guaranteed by Theorem \ref{generating}.

\subsubsection{}
The proof of Theorem \ref{generating} makes use of a $\Z$-grading of
$\tilde\fg$, which descends to $L$ and $\fg$ to give
$\Z$-gradings to these Lie superalgebras. Write $L=\oplus_k L_k$
and $\fg=\oplus_k \fg_k$ with respect the $\Z$-gradings. Lemma
\ref{key} states that $L_0 \cong \fg_0$ as Lie superalgebras and
$L_k \cong \fg_k$ as $\fg_0$-modules for all $k\ne 0$. Then Theorem
\ref{generating} follows from this lemma.

The unconventional Serre relations can now be understood as arising
from two sources: the conditions for $\fg_{\pm 1}$ to be irreducible
$\fg_0$-modules; and the requirement that $[\fg_{\pm 1}, \fg_{\pm
1}]=L_{\pm 2}$ and similar requirements at other degrees.

Recall that Yamane \cite{Y2} used odd reflections \cite{Se} to find such
relations. Leites and collaborators \cite{LS, GL} used
homological algebra techniques and deduced relations from
certain spectral sequences.

The approach developed here is quite different from the
methods in \cite{GL, GLP, BGLL} and in \cite{Y1, Y2}
at both the conceptual and technical level.
It has the advantage of automatically
generating a complete set of relations that is minimal. Conceptually
the approach is quite transparent in the sense that one can see how
the defining relations arise. It also provides an alternative
approach to Serre's theorem for finite dimensional semi-simple Lie
algebras, see Remark \ref{Lie-algebras}.

We also note that the proof in \cite{GK} of the generalised Serre theorem
for Kac-Moody algebras with symmetrisable Cartan
matrices relied on structural properties of Verma modules such as
their embeddings, and also made use of the quadratic Casimir
operator. The authors of both \cite{Y2} and \cite{GL} commented on
obstacles in generalising the proof to Lie superalgebras, especially
difficulties related to the quadratic Casimir operator. We may also
add that one no longer has the properties of (generalised) Verma
modules required by \cite{GK} in the context of Lie superalgebras,
and this appears to be a more serious difficulty.

\subsubsection{}
The organisation of the paper is as follows.  Section
\ref{superalgebras} reviews Kac's classification of finite
dimensional simple classical Lie superalgebras \cite{Kac1}, and also
clarifies certain subtle points about Cartan matrices and Dynkin
diagrams in this context. Section \ref{presentations} contains the
statements of the main results, Theorem \ref{generating} and Theorem
\ref{Serre-g}, which give presentations of contragredient Lie
superalgebras in arbitrary root systems. The proof of Theorem
\ref{generating}, which implies Theorem \ref{Serre-g} as a
corollary, is given by using the key lemma, Lemma \ref{key}.
Sections \ref{proof-distinguished} and \ref{proof-non-distinguished}
are devoted to the proof of the key lemma. An outline of the proof
is given in Section \ref{comments-proof} to explain its conceptual
aspects. We end the paper with a discussion of possible
generalisation of the method developed here to affine Kac-Moody
superalgebras to construct Serre type presentations in Section
\ref{affine}.

Two appendices are also included. Appendix \ref{Dynkin-diagrams}
gives the root systems and Dynkin diagrams of all simple contragredient Lie
superalgebras \cite{Kac1, FSS, CCLL}. The material is used throughout the paper,
and is also necessary in order
to make precise the description of Dynkin diagrams in non-distinguished
root systems.
Appendix \ref{modules} describes the structure of some generalised
Verma modules of lowest weight type and their irreducible quotients,
which enter the proof of Lemma \ref{key}.

\subsection*{Acknowledgement}
I wish to thank Professor Dimitry Leites for helpful suggestions.

\section{Finite dimensional simple Lie superalgebras}\label{superalgebras}

In this section, we present some background material,
and clarify some tricky points about Cartan matrices and Dynkin
diagrams of Lie superalgebras.

\subsection{Finite dimensional simple Lie superalgebras}

We work over the field $\C$ of complex numbers throughout the paper.

\subsubsection{Classification}

A Lie superalgebra $\fg$ is a $\Z_2$-graded vector space
$\fg=\fg_\0\oplus\fg_\1$ endowed with a bilinear map $[\ , \  ]:
\fg\times\fg\longrightarrow \fg$, $(X, Y)\mapsto [X, Y]$, called the
Lie superbracket, which is homogeneous of degree $0$, graded
skew-symmetric and satisfies the super Jacobian identity. The even
subspace $\fg_\0$ of a Lie superalgebra $\fg=\fg_\0\oplus\fg_\1$ is
a Lie algebra in its own right, which is called the even subalgebra
of $\fg$. The odd subspace $\fg_\1$ forms a $\fg_\0$-module under
the restriction of the adjoint action defined by the Lie
superbracket. If $\fg_\0$ is a reductive Lie algebra and $\fg_\1$ is
a semi-simple $\fg_\0$-module, $\fg$ is called {\em classical}
\cite{Kac1, Sch79}.

The classification of the finite dimensional simple
Lie superalgebras was completed in the late 70s.
The theorem below is taken from \cite{Kac1},
which is still the best reference on Lie superalgebras.
Historical information and further references on
the classification can be found in \cite{K, K-collect} (also see \cite{Sch79}).

\begin{theorem}\label{classification}
The finite dimensional simple classical Lie superalgebras
comprise of the simple contragredient Lie superalgebras
\[
\begin{aligned}
& A(m, n), \quad  B(0, n), \quad   B(m, n), \ m>0,
\quad C(n), \  n>2, \quad  D(m, n), \  m>1, \\
&F(4),  \quad  G(3),  \quad D(2, 1; \alpha), \ \alpha\in\C\backslash\{0, -1\},
\end{aligned}
\]
and simple strange Lie superalgebras $P(n)$ and $Q(n)$ ($n\ge 1$).
\end{theorem}

The simple contragredient Lie superalgebras admit non-degenerate
invariant bilinear forms, while the strange Lie superalgebras $P(n)$
and $Q(n)$ do not. In the remainder of the paper, we shall consider
only contragredient simple Lie superalgebras.

The $A$, $B$, $C$ and $D$ series are essentially the special linear
and orthosymplectic Lie superalgebras, which are familiar examples
of Lie superalgebras. The exceptional Lie superalgebras $F(4), G(3)$
and $D(2, 1; \alpha)$ are less well-known, but one can understand their structures
given the description of their roots in Appendix \ref{roots}.

Let $\fg=\fg_\0\oplus\fg_\1$ be a simple contragredient Lie
superalgebra, and choose a Cartan subalgebra $\fh$ for $\fg$, which
by definition is just a Cartan subalgebra of $\fg_\0$. Denote by
$\fg_\alpha$ the root space of the root $\alpha$, and call $\alpha$
even (resp. odd) if $\fg_\alpha\subset \fg_\0$ (resp.
$\fg_\alpha\subset \fg_\1$). Denote by $\Delta_0$ and $\Delta_1$ the
sets of the even and odd roots respectively, and set
$\Delta=\Delta_0\cup\Delta_1$. Let $(\ , \ ):
\fh^*\times\fh^*\longrightarrow \C$ denote the Weyl
group invariant non-degenerate symmetric bilinear form on $\fh^*$, where
the Weyl group of $\fg$ is by definition the Weyl group of $\fg_\0$.
A root $\beta$ will be called isotropic if $(\beta, \beta)=0$. Note
that all isotropic roots are odd.

A Borel subalgebra of $\fg$ is a maximal soluble Lie super
subalgebra containing a Borel subalgebra of $\fg_\0$. A new feature
in the present context is that Borel subalgebras are not always
conjugate under the Weyl groups. All the conjugacy classes of Borel
subalgebras were given in \cite[pp. 51-52]{Kac1} \cite[Proposition
1.2]{Kac2}. In particular, Kac described a particularly convenient
Borel subalgebra, which he called distinguished, for each simple
contragredient Lie superalgebra.  We shall call a root system with
the set of simple roots determined by this Borel subalgebra the {\em
distinguished root system}. In this case, there exists only one odd
simple root.

\subsubsection{Cartan matrices and Dynkin diagrams}

The precise forms of the Cartan matrices and Dynkin diagrams will be crucial in
Section \ref{presentations}. However, there do not exist canonical definitions
for them in the Lie superalgebra setting, thus we spell out the details of
our definitions here.

Let $\Pi_\fb=\{\alpha_1, \alpha_2, \dots, \alpha_r\}$ be the set of
simple roots of a simple contragrediant Lie superalgebra $\fg$
relative to a Borel subalgebra $\fb$. The Cartan matrix and Dynkin
diagram provide a convenient way to describe $\Pi_\fb$. We define a
Cartan matrix in the following way.  Denote by $\Theta\subset\{1, 2,
\dots, r\}$ the subset such that $\alpha_t\in\Delta_1$  for all
$t\in\Theta$. Let $l_m^2$ be the minimum of $|(\beta, \beta)|$ for
all non-isotropic $\beta\in\Delta$ if $\fg\ne D(2, 1; \alpha)$. If
$\fg$ is $D(2, 1; \alpha)$, let $l_m^2$ be the minimum of all
$|(\beta, \beta)|>0$ ($\beta\in\Delta$), which are independent of
the arbitrary parameter $\alpha$. Let
\[
\begin{aligned}
\kappa&=\left\{\begin{array}{l l} 0, &\text{if $\fg$ is of type $B$},\\
1, &\text{ otherwise};
\end{array}
\right.\quad\quad
d_i&=\left\{\begin{array}{l l} \frac{(\alpha_i,
\alpha_i)}{2}, &\text{if\ } (\alpha_i,
\alpha_i)\ne 0,\\
\frac{l_m^2}{2^\kappa}, &\text{if\ } (\alpha_i, \alpha_i)= 0.
\end{array}
\right.
\end{aligned}
\]
Introduce the matrices
\[
\begin{aligned}
B&=(b_{i j})_{i, j=1}^r, \quad b_{i j} = (\alpha_i, \alpha_j),\\
D&=\text{diag}(d_1, \dots, d_r),
\end{aligned}
\]
then the Cartan matrix $A$ associated to the set of simple
roots $\Pi_\fb$ is defined by
\begin{eqnarray}
A=D^{-1} B.
\end{eqnarray}
When it is necessary to indicate the dependence on $\Theta$, we
write $(A, \Theta)$ for the Cartan matrix.

Note that if $\alpha_i$ is non-isotropic, $a_{i t}=\frac{2(\alpha_i,
\alpha_t)}{(\alpha_i, \alpha_i)}$ is a non-positive integer for all
$t$. However, if $\alpha_t$ is isotropic, then $a_{t j} =
\frac{2}{l_m^2}(\alpha_t, \alpha_j)$ can be an integer of any sign
or zero (except in type $D(2, 1; \alpha)$). If $b_{i j}\ne 0$, we
define
\begin{eqnarray}
sgn_{i j} = \text{sign of $b_{i j}$}.
\end{eqnarray}
As we shall see in Section \ref{problems-Dynkin-diagrams}, these
signs provide the additional information required to recover a
Cartan matrix from its Dynkin diagram.

\begin{remark}
Our definition of the Cartan matrix differs from the usual one due
to Kac \cite{Kac1}. In Kac's definition,  if $b_{s s}=0$, then
$d_s=(\alpha_s, \alpha_{s+k})$ for the smallest $k$ such that
$d_s\ne 0$. Note that in our definition, none of the signs $sgn_{i
j}$ is lost.
\end{remark}

The Dynkin diagram associated with $(A, \Theta)$ consists of $r$
nodes, which are connected by lines. The $i$-th node is coloured
white if $i\not\in\Theta$, black if $i\in\Theta$ but $\alpha_i$ is
not isotropic, and grey if $\alpha_i$ is isotropic.

If $(A, \Theta)$  is of type $D(2, 1; \alpha)$, the Dynkin diagram is
obtained by simply connecting the $i$-th and $j$-th nodes by one
line if $a_{i j}\ne 0$ and write $b_{i j}$ at the line.

In all other cases, we join the $i$-th
and $j$-th nodes by $n_{ij}$ lines, where
\[
\begin{aligned}
&n_{i j} = {\rm{max}}(|a_{i j}|, |a_{j i}|), &\quad \text{if $a_{i i} + a_{j j}\ge 2$};\\
&n_{i j} = |a_{i j}|, &\quad \text{if $a_{i i}=a_{j j}=0$}.
\end{aligned}
\]
When the $i$-th and $j$-th nodes are not both grey, say, the $i$-th
one is not grey, and connected by more than one lines, we draw an
arrow pointing to the $j$-th node if $-a_{i j}=1$ and pointing to
the $i$-th node if $-a_{i j}>1$.

The Dynkin diagrams of the simple contragredient Lie superalgebras
are given in the tables in Appendix \ref{append-diagrams}.

\subsection{Comments on Dynkin diagrams}\label{problems-Dynkin-diagrams}

From the Cartan matrices in our definition, one can recover the
corresponding root systems. Dynkin diagrams also uniquely represent
Cartan matrices, except in the cases of $\mathfrak{osp}_{4|2}$ and
$\mathfrak{sl}_{2|2}$. The Dynkin diagrams of these superalgebras
relative to the distinguished root systems are exactly the same, but
the two Lie superalgebras are non-isomorphic.

This problem can be resolved by incorporating the signs $sgn_{i j}$
into the Dynkin diagram, e.g., by placing $sgn_{i j}$ at the line(s)
connecting two grey nodes $i$ and $j$. Then the modified Dynkin
diagram are respectively given by
\begin{eqnarray}\label{sl22-osp42}
%
\begin{picture}(120, 20)(-20, 5)
\put(-40, 10){$\mathfrak{sl}_{2|2}$:}
\put(0, 10){\circle{10}}
\put(5, 10){\line(1, 0){20}}
{\color{gray} \put(30, 10){\circle*{10}}}
\put(35, 10){\line(1, 0){20}}
\put(60, 10){\circle{10}}

\put(10, 13){\tiny $-$}
\put(40, 13){\tiny $+$}
\put(68, 5){,}
\end{picture}
\begin{picture}(120, 20)(-50, 5)
\put(-50, 10){$\mathfrak{osp}_{4|2}$:}
\put(0, 10){\circle{10}}
\put(5, 10){\line(1, 0){20}}
{\color{gray} \put(30, 10){\circle*{10}}}
\put(35, 10){\line(1, 0){20}}
\put(60, 10){\circle{10}}

\put(10, 13){\tiny $-$}
\put(40, 13){\tiny $-$}
\put(68, 5){.}
\end{picture}
\end{eqnarray}
As we shall see, the signs enter the construction of higher order Serre relations.

In this paper we did not include the additional information of these
signs in the definition of Dynkin diagrams, as they would make the
diagrams look cumbersome. Also, there is no ambiguity about the
signs in all the other Dynkin diagrams.

Similar signs were also discussed in \cite{Y2}.

Recall that if we remove a subset of vertices (i.e., nodes) and all
the edges connected to these vertices from a Dynkin diagram of a
semi-simple Lie algebra, we obtain the Dynkin diagram of another
semi-simple Lie algebra of a smaller rank. This corresponds to
taking regular subalgebras. In the context Lie superalgebras, the
notion of regular subalgebras still exists, but some explanation is
required at the level of Dynkin diagrams.

\begin{definition}
Call a sub-diagram $\Gamma'$ of a Dynkin diagram $\Gamma$
{\em full} if for any two nodes $i$ and $j$ in $\Gamma'$, the edges
between them in $\Gamma$, the arrows on the edges, and also the $b_{i j}$ labels of the edges
when $\Gamma$ is of type $D(2, 1; \alpha)$, are all present in
$\Gamma'$.
\end{definition}

Consider for example the Dynkin diagram
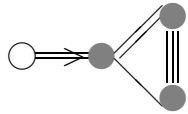
\begin{center}
\begin{picture}(75, 45)(-5,-20)
\put(5, 0){\circle{10}}
\put(10, 1){\line(1, 0){20}}
\put(10, -1){\line(1, 0){20}}
\put(20, -3.5){$>$}
{\color{gray}\put(35, 0){\circle*{10}}}
\put(58, 19){\line(-1, -1){18}}
\put(58, 15){\line(-1, -1){16}}
\put(58, -19){\line(-1, 1){18}}
{\color{gray}\put(62, 15){\circle*{10}}}
\put(60,-10){\line(0, 1){19}}
\put(62,-10){\line(0, 1){19}}
\put(64,-10){\line(0, 1){19}}
{\color{gray}\put(62, -16){\circle*{10}}}
\end{picture}
\end{center}
of $F(4)$, which has the following full sub-diagrams beside others:
\begin{eqnarray}\label{nonstandard}
\begin{picture}(150, 30)(0,-10)
\put(5, 0){\circle{10}}
\put(10, 1){\line(1, 0){20}}
\put(10, -1){\line(1, 0){20}}
\put(20, -3.5){$>$}
{\color{gray}\put(35, 0){\circle*{10}}}
\put(40, 1){\line(1, 0){20}}
\put(40, -1){\line(1, 0){20}}
{\color{gray}\put(65, 0){\circle*{10}}}
\put(70, -3){,}

{\color{gray}\put(100, 0){\circle*{10}}} \put(105, 2){\line(1,
0){20}} \put(105, 0){\line(1, 0){20}} \put(105, -2){\line(1, 0){20}}
{\color{gray}\put(130, 0){\circle*{10}}} \put(135, -3){.}
\end{picture}
\end{eqnarray}
Note that none of these appears in Tables 1 and 2.

The reason is that the sub-matrices in the Cartan matrix of $F(4)$
associated with these full sub-diagrams are not Cartan matrices in
the strict sense. The problem lies in the definition of $a_{i j}$
when the node $i$ is grey, which involves the number $\l_m$. The
$\l_m$ for $F(4)$ is not the correct ones for the full sub-diagrams.
By properly renormalising the bilinear forms on the weight spaces
associated with them, the full sub-diagrams can be cast into the
form
\begin{center}
\begin{picture}(150, 25)(0,-10)
\put(5, 0){\circle{10}}
\put(10, 0){\line(1, 0){20}}
{\color{gray}\put(35, 0){\circle*{10}}}
\put(40, 0){\line(1, 0){20}}
{\color{gray}\put(65, 0){\circle*{10}}}
\put(70, -3){,}

{\color{gray}\put(100, 0){\circle*{10}}}
\put(105, 0){\line(1, 0){20}}
{\color{gray}\put(130, 0){\circle*{10}}}
\put(135, -3){,}
\end{picture}
\end{center}
which are respectively Dynkin diagrams for $\mathfrak{sl}_{3|1}$ and
$\mathfrak{sl}_{2|1}$.

We call the Dynkin diagrams in Table 1 and Table 2 standard, and the ones
like those in \eqref{nonstandard} non-standard.

We mention that if a Lie superalgebra $\fg$ is contained as a
regular subalgebra in another Lie superalgebra, defining relations
of $\fg$ can in principle be extracted from relations of the latter
by considering sub-diagrams of Dynkin diagrams. However, this
involves subtleties,  as we have just discussed, and requires more care
than hitherto exercised in the literature.

\section{Presentations of Lie superalgebras}\label{presentations}

In this section, we generalise Serre's theorem for semi-simple Lie
algebras to contragredient Lie superalgebras,  obtaining
presentations for the Lie superalgebras in terms of Chevalley
generators and defining relations.

\subsection{An auxiliary Lie superalgebra}
We start by defining an auxiliary Lie superalgebra following the
strategy of \cite{Kac3}. Let $(A, \Theta)$ with $A=(a_{i j})_{i, j=1}^r$ be the
Cartan matrix of one of the simple contragredient Lie superalgebras
relative to a given Borel subalgebra $\fb$. Let $\Pi_\fb$ be the set of simple roots
relative to this Borel subalgebra.

\begin{definition}\label{auxiliary}
Let $\tilde\fg(A, \Theta)$ be the Lie superalgebra generated by
homogeneous generators $e_i, f_i, h_i$ ($i=1, 2, \dots, r$), where
$e_s$, $f_s$ for all $s\in\Theta$ are odd while the rest are even,
subject to the following relations
\begin{eqnarray}  \label{g-tilde}
\begin{aligned}
&{[}h_i, h_j] =0, \\
&[h_i, e_j] =a_{i j} e_j, \quad [h_i, f_j] =-a_{i j} f_{j}, \\
&[e_i, f_j] =\delta_{i j} h_i, \quad \forall i, j.
\end{aligned}
\end{eqnarray}
\end{definition}

Let $\tilde\fn^+$ (resp. $\tilde\fn^-$) be the subalgebra
generated by all $e_i$ (resp. all $f_i$) subject to the relevant
relations, and $\fh=\oplus_{i=1}^r\C\fh_i$, the Cartan subalgebra.
Then it is well known and easy to prove  (following the reasoning of \cite[\S 1]{Kac3}) that
$\tilde\fg(A, \Theta) =\tilde\fn^+\oplus \fh \oplus \tilde\fn^-$.
The Lie superalgebra is  graded
$\tilde\fg(A, \Theta)=\oplus_{\nu\in Q}\tilde\fg_\nu$ by $Q=\Z\Pi_\fb$, with $\tilde\fg_0=\fh$. Note hat
$\tilde\fn^+_{\nu}$ (rep. $\tilde\fn^-_{-\nu}$) is zero unless $\nu\in Q_{\N}$, where $\N=\{1, 2, \dots \}$
and $Q_\N=\N\Pi_\fb$, that is,
\begin{eqnarray}  \label{triangular}
\tilde\fn^+ = \oplus_{\nu\in Q_\N} \tilde\fn^+_{\nu}, \quad \tilde\fn^- = \oplus_{\nu\in Q_\N} \tilde\fn^-_{-\nu}.
\end{eqnarray}

Let $\fr(A, \Theta)$ be the maximal $\Z_2$-graded ideal of
$\tilde\fg(A, \Theta)$ that intersects $\fh$ trivially. Set
$\fr^{\pm}=\fr(A, \Theta)\cap\tilde\fn^{\pm}$. Then
$\fr(A, \Theta)=\fr^+\oplus\fr^-$. The following fact follows from
the maximality of $\fr(A, \Theta)$.
\begin{lemma}\label{ideal-inclusion}
Let $\Sigma=\Sigma^+\cup\Sigma^-$ with $\Sigma^\pm \subset
\tilde\fn^\pm$ be a subset of $\tilde\fg(A, \Theta)$ consisting of
homogeneous elements. If $[f_i, \Sigma^+]\subset\C\Sigma^+$ and
$[e_i, \Sigma^-]\subset\C\Sigma^-$ for all $i$, then $\Sigma \subset
\fr(A, \Theta)$.
\end{lemma}
\begin{proof}
The given conditions on $\Sigma$ imply that the ideal generated by
$\fr(A, \Theta)\cup\Sigma$ intersects $\fh$ trivially, hence must be
equal to $\fr(A, \Theta)$ by the maximality of the latter.
\end{proof}
In particular, if $X^\pm\in\tilde\fn^\pm$ satisfy $[f_i, X^+]=0$,
and $[e_i, X^-]=0$ for all $i$, then they belong to $\tilde\fn^\pm$
respectively.

Let us define the Lie superalgebra
\[
L(A, \Theta) := \frac{\tilde\fg(A, \Theta)}{\fr(A, \Theta)}.
\]
We have the following result.
\begin{theorem}\label{LA}
Let $\fg$ be a finite dimensional simple contragredient Lie
superalgebra, and let $(A, \Theta)$ be the Cartan matrix of $\fg$
relative to a given Borel subalgebra. Then $L(A, \Theta)$ is
isomorphic to $\fg$ unless $\fg=A(n, n)$, and in the latter case
$L(A, \Theta)\cong \mathfrak{sl}_{n+1|n+1}$.
\end{theorem}
\begin{proof}
This follows from Kac's classification \cite{Kac1} of the simple
contragredient Lie superalgebras (see Theorem \ref{classification})
except in the case of $A(n, n)$. In the latter case, we have
$\det{A}=0$. Therefore, $L(A, \Theta)$ contains a
$1$-dimensional center, and the quotient of $L(A, \Theta)$ by the
center is $A(n, n)$. Hence $L(A, \Theta)$ is isomorphic to
$\mathfrak{sl}_{n+1|n+1}$.
\end{proof}

\subsection{Main theorem}
\subsubsection{Standard and higher order Serre elements}

Let us first define some elements of $\tilde\fg(A, \Theta)$, which
will play a crucial role in studying the presentation of Lie
superalgebras.

Call the following elements the {\em standard Serre elements}:
\begin{eqnarray*}
\begin{aligned}
&(ad_{e_i})^{1-a_{i j}}(e_j), \quad
(ad_{f_i})^{1-a_{i j}}(f_j),
    \quad \text{for $i\ne j$, with $a_{i i}\ne 0$ or $a_{i j}=0$};\\
&[e_s, e_s], \quad  [f_s, f_s],
   \quad \text{for $a_{ss}=0$}.
\end{aligned}
\end{eqnarray*}
We also introduce {\em higher order Serre elements} if the Dynkin
diagram of $(A, \Theta)$ contains full sub-diagrams of the following
kind:
\begin{enumerate}
\item \label{case-1-ideal}
\begin{picture}(75, 15)(0, 7)
\put(10, 6.5){$\times$}
\put(15, 10){\line(1, 0){20}}
{\color{gray} \put(40, 10){\circle*{10}} }
\put(45, 10){\line(1, 0){20}}
\put(62, 6.5){$\times$}
\put(10, 16){\tiny $j$}
\put(40, 16){\tiny $t$}
\put(63, 16){\tiny $k$}
\end{picture}
with $sgn_{j t}sgn_{t k}=-1$, the
associated higher order Serre elements are
\[
[e_t, [e_j, [e_t, e_k]]],  \quad [f_t, [f_j, [f_t, f_k]]];
\]

\item \label{case-2-ideal}
\begin{picture}(80, 20)(0, 7)
\put(10, 6.5){$\times$}
\put(15, 10){\line(1, 0){20}}
{\color{gray}\put(40, 10){\circle*{10}} }
\put(45, 11){\line(1, 0){20}}
\put(45, 9){\line(1, 0){20}}
\put(55, 7){$>$} \put(70,
10){\circle{10}} \put(10, 16){\tiny $j$}
\put(40, 16){\tiny $t$}
\put(70, 16){\tiny $k$}
\put(76, 7){,}
\end{picture}
the associated higher order Serre elements are
\[
[e_t, [e_j, [e_t, e_k]]],  \quad [f_t, [f_j, [f_t, f_k]]];
\]


\item \label{case-3-ideal}
\begin{picture}(80, 20)(0, 7)
\put(10, 6.5){$\times$}
\put(15, 10){\line(1, 0){20}}
{\color{gray}\put(40, 10){\circle*{10}} }
\put(45, 11){\line(1, 0){20}}
\put(45, 9){\line(1, 0){20}}
\put(55, 7){$>$}
\put(70, 10){\circle*{10}}
\put(10, 16){\tiny $j$}
\put(40, 16){\tiny $t$} \put(70, 16){\tiny $k$}
\put(76, 7){,}
\end{picture}
the associated higher order Serre elements are
\[
[e_t, [e_j, [e_t, e_k]]],  \quad [f_t, [f_j, [f_t, f_k]]];
\]

\item \label{case-4-ideal}
\begin{picture}(80, 20)(0, 7)
{\color{gray}\put(10, 10){\circle*{10}} }
\put(15, 10){\line(1, 0){20}} {\color{gray}
\put(40, 10){\circle*{10}} }
\put(45, 11){\line(1, 0){20}}
\put(45, 9){\line(1, 0){20}}
\put(46, 7){$<$}
\put(70, 10){\circle{10}}
\put(10, 16){\tiny $j$}
\put(40, 16){\tiny $t$}
\put(70, 16){\tiny $k$}
\put(76, 7){,}
\end{picture}
the associated higher order Serre elements are
\[
\begin{aligned}
{[}[e_j, e_t], [[e_j, e_t], [e_t, e_k]]],  \\
{[}[f_j, f_t], [[f_j, f_t], [f_t, f_k]]];
\end{aligned}
\]

\item \label{case-5-ideal}
\begin{picture}(120, 20)(-10, 7)
\put(0, 6.5){$\times$}
\put(5, 10){\line(1, 0){20}}
\put(30, 10){\circle{10}}
\put(35, 10){\line(1, 0){20}}
{\color{gray}\put(60, 10){\circle*{10}} }
\put(65, 11){\line(1, 0){20}}
\put(65, 9){\line(1, 0){20}}
\put(66, 7){$<$}
\put(90, 10){\circle{10}}
\put(6, 16){\tiny $i$}
\put(30, 16){\tiny $j$}
\put(60, 16){\tiny $t$}
\put(90, 16){\tiny $k$}
\put(96, 7){,}
\end{picture}
the associated higher order Serre elements are
\[
\begin{aligned}
{[[}e_i, [e_j, e_t]], [[e_j, e_t], [e_t, e_k]]],  \\
{[[}f_i, [f_j, f_t]], [[f_j, f_t], [f_t, f_k]]];
\end{aligned}
\]

\item \label{case-6-ideal}
\begin{picture}(65, 30)(0, 7)
\put(10, 6.8){$\times$}
\put(15, 10){\line(1, 1){20}}
\put(15,10){\line(1, -1){20}}
{\color{gray} \put(40, 30){\circle*{10}}}
\put(39, 25){\line(0, -1){30}}
\put(41, 25){\line(0, -1){30}}
{\color{gray} \put(40, -10){\circle*{10}}}
\put(10, 16){\tiny $i$}
\put(47, 28){\tiny $t$}
\put(47, -13){\tiny $s$}
\put(50, 7){,}
\end{picture}
the associated higher order Serre elements are
\[
\begin{aligned}
{[}e_t, [e_s, e_i]]-[e_s, [e_t, e_i]], \\
{[}f_t, [f_s, f_i]]-[f_s, [f_t, f_i]];
\end{aligned}
\]

\item\label{case-7-1-ideal}

\begin{picture}(125, 20)(-15,-5)
\put(0, 0){\circle{10}}
\put(5, 2){\line(1, 0){20}}
\put(5, 0){\line(1, 0){20}}
\put(5, -2){\line(1, 0){20}}
\put(15, -3.25){$>$}
{\color{gray} \put(30, 0){\circle*{10}}}
\put(35, 1){\line(1, 0){20}}
\put(35, -1){\line(1, 0){20}}

\put(35, -3.25){$<$}

\put(60, 0){\circle{10}}
\put(65, 0){\line(1, 0){20}}
\put(90, 0){\circle{10}}

\put(0, 7){\tiny 1}
\put(30, 7){\tiny 2}
\put(60, 7){\tiny 3}
\put(90, 7){\tiny 4}

\put(96, -2){,}
\end{picture}
which is a Dynkin diagram of $F(4)$, the associated higher order
Serre elements are
\[
\begin{aligned}
{[} E, [ E, [e_2, [e_3, e_4]]  ] {]}, \\
{[} F, [ F, [f_2, [f_3, f_4]]  ]  {]},
\end{aligned}
\]
where $E=[[e_1, e_2], [e_2, e_3]]$ and $F=[[f_1, f_2], [f_2, f_3]]$;

\item \label{case-7-2-ideal}
\begin{picture}(125, 25)(-15,-5)
\put(0, 0){\circle{10}}
\put(5, 2){\line(1, 0){20}}
\put(5, 0){\line(1, 0){20}}
\put(5, -2){\line(1, 0){20}}
\put(15, -3.25){$>$}
{\color{gray}\put(30, 0){\circle*{10}}}
\put(35, 0){\line(1, 0){20}}
\put(60, 0){\circle{10}}
\put(65, 1){\line(1, 0){20}}
\put(65, -1){\line(1, 0){20}}
\put(65, -3.25){$<$}
\put(90, 0){\circle{10}}

\put(0, 7){\tiny 1}
\put(30, 7){\tiny 2}
\put(60, 7){\tiny 3}
\put(90, 7){\tiny 4}

\put(96, -2){,}
\end{picture}
which is a Dynkin diagram of $F(4)$, the associated higher order
Serre elements are
\[
\begin{aligned}
{[}[e_1, e_2], [[e_2, e_3], [e_3, e_4]]
-[[e_2, e_3], [[e_1, e_2], [e_3, e_4]],\\
{[}[f_1, f_2], [[f_2, f_3], [f_3, f_4]]
-[[f_2, f_3], [[f_1, f_2], [f_3, f_4]];
\end{aligned}
\]

\item \label{case-8-ideal}

\begin{picture}(100, 25)(-20,-3)
\put(0, 0){\circle{10}}
\put(5, 1){\line(1, 0){20}}
\put(5, -1){\line(1, 0){20}}
\put(15, -3.5){$>$}
{\color{gray}\put(30, 0){\circle*{10}}}
\put(35, 1){\line(1, 0){20}}
\put(35, -1){\line(1, 0){20}} {\color{gray}
\put(60, 0){\circle*{10}}}

\put(0, 7){\tiny k}
\put(30, 7){\tiny t}
\put(60, 7){\tiny j}

\put(70, 0){,}
\end{picture}
which only appears in Dynkin diagrams of  $F(4)$, the associated
higher order Serre elements are
\[
\begin{aligned}
{[} e_t, [e_j,[e_t, e_k]] ],\\
{[} f_t, [f_j, [f_t, f_k]] ];
\end{aligned}
\]

\item \label{case-9-ideal}
\begin{picture}(75, 25)(-15,5)
{\color{gray}\put(5, 10){\circle*{10}}}
\put(5, 17){\tiny i}

\put(10, 10){\line(1, 1){20}}
\put(10, 10){\line(1, -1){20}}
\put(10, 8){\line(1, -1){18}}
{\color{gray}\put(32, 27){\circle*{10}}}
\put(40, 27){\tiny j}

\put(30,-2){\line(0, 1){24}}
\put(32,-2){\line(0, 1){24}}
\put(34,-2){\line(0, 1){24}}
{\color{gray}\put(32,-7){\circle*{10}}}
\put(40,-10){\tiny k}

\put(42,7){,}
\end{picture}
which only appears in one of the Dynkin diagrams of  $F(4)$,

\vspace{.8cm}
\noindent the associated higher order Serre elements are
\[
\begin{aligned}
2[e_i, [e_k, e_j]]+3[e_j, [e_k, e_i]], \\
2[f_i, [f_k, f_j]]+3[f_j, [f_k, f_i]];
\end{aligned}
\]

\item \label{case-10-ideal}
\begin{picture}(100, 20)(10,-5)
{\color{gray}\put(30, 0){\circle*{10}}} \put(35, 0){\line(1,
0){20}} {\color{gray}\put(60, 0){\circle*{10}}} \put(65,
2){\line(1, 0){20}} \put(65, 0){\line(1, 0){20}} \put(65,
-2){\line(1, 0){20}} \put(65, -3.25){$<$} \put(90,
0){\circle{10}}

\put(30, 7){\tiny 1} \put(60, 7){\tiny 2} \put(90, 7){\tiny 3}

\put(100, -5){,}
\end{picture}
which is one of the Dynkin diagrams of $G(3)$, the associated
higher order Serre elements are
\[
\begin{aligned}
{[} [e_1, e_2], [ [e_1, e_2], [ [e_1, e_2], [e_2, e_3] ] ]], \\
{[} [f_1, f_2], [ [f_1, f_2], [ [f_1, f_2],[f_2, f_3] ] ]];
\end{aligned}
\]


\item \label{case-11-ideal}
\begin{picture}(100, 20)(10,-3)
\put(30, 0){\circle*{10}}
\put(35, 1){\line(1, 0){20}}
\put(35, -1){\line(1, 0){20}}
\put(35, -3.25){$<$} {\color{gray}
\put(60, 0){\circle*{10}}}
\put(65, 2){\line(1, 0){20}}
\put(65, 0){\line(1, 0){20}}
\put(65, -2){\line(1, 0){20}}
\put(65, -3.25){$<$}
\put(90, 0){\circle{10}}
\put(100, -3){,}

\put(30, 7){\tiny 1}
\put(60, 7){\tiny 2}
\put(90, 7){\tiny 3}

\end{picture}
which is one of the Dynkin diagrams of $G(3)$, the associated
higher order Serre elements are
\[
\begin{aligned}
{[}[e_2, e_1], [e_3, [e_2, e_1]]] - [[e_2, e_3], [[e_1, e_1], e_2]], \\
{[}[f_2, f_1], [f_3, [f_2, f_1]]] - [[f_2, f_3], [[f_1, f_1], f_2]];
\end{aligned}
\]

\item \label{case-12-ideal}
\begin{picture}(65, 35)(-15, 0)
\put(5, 0){\circle{10}}
\put(5, 7) {\tiny 1}

\put(10, 1){\line(1, -1){20}}
\put(10, -1){\line(1, -1){20}}
\put(12, -2){\line(2, -1){10}}
\put(12, -2){\line(1, -2){5}}
{\color{gray}
\put(35, -20){\circle*{10}}}
\put(42, -22){\tiny 3}

\put(33, 15){\line(0, -1){30}}
\put(35, 15){\line(0, -1){30}}
\put(37, 15){\line(0, -1){30}}

{\color{gray}
\put(35, 20){\circle*{10}}}
\put(42, 18){\tiny 2}

\put(30, 20){\line(-1, -1){19}}
\put(44, 0){,}
\end{picture}
which is one of the Dynkin diagrams of $G(3)$, the associated

\vspace{1cm} \noindent higher order
Serre elements are
\[
\begin{aligned}
{[}e_2, [e_3, e_1]]-2[e_3, [e_2, e_1]], \\
{[}f_2, [f_3, f_1]]-2[f_3, [f_2, f_1]];
\end{aligned}
\]

\vspace{1cm}

\item \label{case-13-ideal}
\begin{picture}(100, 10)(50,20)
{\color{gray}\put(75, 20){\circle*{10}}} \put(80, 20){\line(1,
1){20}} \put(80, 20){\line(1, -1){20}} \put(82, 30){\tiny $1$}
\put(82 , 5){\tiny $\alpha$} \put(105,15){\tiny $-(1+\alpha)$}
{\color{gray}\put(102, 35){\circle*{10}}} \put(102, 5){\line(0,
1){25}} {\color{gray}\put(102, 0){\circle*{10}}}
\put(138, 10){,}
\end{picture}
which is one of the Dynkin diagram for $D(2, 1; \alpha)$.

\vspace{1cm}
\noindent The higher order
Serre elements are
\[
\begin{aligned}
\alpha [e_1, [e_2, e_3]] + (1+\alpha)[e_2, [e_1, e_3]], \\
\alpha [f_1, [f_2, f_3]] + (1+\alpha)[f_2, [f_1, f_3]],
\end{aligned}
\]
where we label the left, top and bottom nodes
by $1, 2$ and $3$ respectively.
\end{enumerate}

\begin{remark}\label{new}
Cases (7) - (14) were not considered  before in the literature.
\end{remark}

\begin{remark}
The Dynkin diagrams of $D(2, 1)$ and $D(2, 1; \alpha)$ in their respective distinguished
root systems are not among the full sub-diagrams listed above. Also,
the diagram (\ref{case-8-ideal}) above is a non-standard diagram of $\mathfrak{sl}_{3|1}$
(see Section \ref{problems-Dynkin-diagrams}).
\end{remark}

Denote by $\mathcal{S}^+(A, \Theta)$ (resp.
$\mathcal{S}^-(A, \Theta)$) the set of all the standard
and higher order Serre elements (if defined) which involve
generators $e_k$ (resp. $f_k$) only. Set
$\mathcal{S}(A, \Theta)=\mathcal{S}^+(A, \Theta)\cup\mathcal{S}^-(A, \Theta)$.
We have the following result.
\begin{lemma}\label{s-in-r}
The set $\mathcal{S}(A, \Theta)$ is contained in the maximal ideal $\fr(A, \Theta)$
of $\tilde\fg$.
\end{lemma}
\begin{proof}
Direct calculations show that
\[
[f_i, \mathcal{S}^+(A, \Theta)]\subset \C\mathcal{S}^+(A, \Theta), \quad
[e_i, \mathcal{S}^-(A, \Theta)]\subset \C\mathcal{S}^-(A, \Theta), \quad \forall i.
\]
Hence $\mathcal{S}(A, \Theta)\subset\fr(A, \Theta)$ by Lemma
\ref{ideal-inclusion}. We leave out the details of the calculations.
\end{proof}

\begin{definition}
Let $\fs(A, \Theta)$ be the $\Z_2$-graded ideal of
$\tilde\fg(A, \Theta)$ generated by the elements of
$\mathcal{S}(A, \Theta)$.
\end{definition}

Then $\fs(A, \Theta)\subset\fr(A, \Theta)$ by Lemma \ref{s-in-r}.
Define the Lie superalgebra
\begin{eqnarray}
\fg(A, \Theta):=\frac{\tilde\fg(A, \Theta)}{\fs(A, \Theta)}.
\end{eqnarray}
There exists a natural surjective Lie superalgebra map
$\fg(A, \Theta)\longrightarrow L(A, \Theta)$. We shall show that
it is in fact an isomorphism.

\subsubsection{$\Z$-gradings}\label{Z-grading}

Let us discuss $\Z$-gradings for the Lie supealgebras
$\fg(A, \Theta)$ and $L(A, \Theta)$. Fix a positive integer $d\le
r$, where $r$ is the size of $A$. We assign degrees to the
generators of $\tilde\fg(A, \Theta)$ as follows:
\begin{eqnarray}
\begin{aligned}
&deg(h_j)=0, \quad \forall j, \\
&deg(e_i)=deg(f_i)=0, \quad \forall i\ne d, \\
&deg(e_d)=-deg(f_d)=1.
\end{aligned}
\end{eqnarray}
This introduces a $\Z$-grading to the auxiliary Lie superalgebra
$\tilde\fg(A, \Theta)$, which is not required to be compatible with
the $\Z_2$-grading upon reduction modulo $2$.
In view of the $Q$-grading of $\tilde\fg(A, \Theta)$ and \eqref{triangular}, the maximal ideal $\fr(A, \Theta)$ is
$\Z$-graded. Since all elements in $\mathcal{S}(A, \Theta)$ are homogeneous with
the $\Z$-grading, $\fs(A, \Theta)$ is $\Z$-graded as well.

The Lie superalgebra $L(A, \Theta)$ inherits a $\Z$-grading from
$\tilde\fg(A, \Theta)$ and $\fr(A, \Theta)$. Write
$L(A, \Theta)=\oplus_{k\in\Z}L_k$. Since the roots of $L(A, \Theta)$
are known, we have a detailed understanding of all $L_k$  as
$L_0$-modules.

The Lie superalgebra $\fg(A, \Theta)$ inherits a $\Z$-grading from
$\tilde\fg(A, \Theta)$ and $\fs(A, \Theta)$. Write
$\fg(A, \Theta)=\oplus_{k\in\Z}\fg_k$, where $\fg_k$ is the
homogeneous component of degree $k$. Note that $\fg_1$ (resp.
$\fg_{-1}$) generates $\fg_{k}$ (resp. $\fg_{-k}$) for all $k>0$.
Thus if $\fg_{p}=0$ (resp. $\fg_{-p}=0$) for some $p>0$, then
$\fg_{q}=0$ (resp. $\fg_{-q}=0$) for all $q>p$. Also each $\fg_k$
forms a $\fg_0$-module in the obvious way.

We have the following result.
\begin{lemma}\label{key}
There exist $\Z$-gradings for $\fg(A, \Theta)$ and $L(A, \Theta)$
determined by some $d$ such that $\fg_0=L_0$ as Lie superalgebras
and $\fg_k=L_k$ as $\fg_0$-modules for all nonzero $k\in\Z$.
\end{lemma}

This is the key lemma needed for establishing Theorem
\ref{generating} below. Its proof is elementary but very lengthy,
thus we relegate it to later sections.
Here we consider some general properties of the Lie
superalgebras $\fg(A, \Theta)$ and $L(A, \Theta)$, which will
significantly simplify the proof of Lemma \ref{key}.

Recall that an anti-involution $\omega$ of a Lie superalgebra
$\mathfrak{a}$ is a linear map on $\mathfrak{a}$ satisfying
$\omega([X, Y])=[\omega(Y), \omega(X)]$ for all $X,
Y\in\mathfrak{a}$, and $\omega^2=\id_{\mathfrak{a}}$. The Lie
superalgebra $\tilde\fg(A, \Theta)$ admits an anti-involution
defined by
\[
\omega(e_i)=f_i, \quad \omega(f_i)=e_i, \quad \omega(h_i)=h_i, \quad
\forall i.
\]
Note that $\omega({\mathcal S}^+)\subset -{\mathcal S}^-\cup
{\mathcal S}^-$ and $\omega({\mathcal S}^-) \subset -{\mathcal
S}^+\cup {\mathcal S}^+$, where ${\mathcal S}^\pm= {\mathcal S}^\pm(A, \Theta)$
and $-{\mathcal S}^\pm$ are respectively
the sets consisting of the negatives of the elements of ${\mathcal
S}^\pm$. Therefore, $\omega$ descents to an anti-involution on
$\fg(A, \Theta)$, which sends $\fg_k$ to $\fg_{-k}$ for all $k\in\Z$
and provides a $\fg_0$-module isomorphism between $\fg_{-k}$ and
the dual space of $\fg_k$.

The anti-involution of $\tilde\fg(A, \Theta)$ also descends to an
anti-involution of $L(A, \Theta)$, which maps $L_k$ to $L_{-k}$ for
all $k\in\Z$, and provides an isomorphism between the $L_0$-module
$L_{-k}$ and the dual $L_0$-module of $L_k$.

Therefore, if $\fg_0=L_0$ and $\fg_k=L_k$ for all $k>0$ as
$\fg_0$-modules, the existence of the anti-involutions immediately
implies that $\fg_{-k}=L_{-k}$ for all $k>0$. Hence in order to
prove Lemma \ref{key}, we only need to show that it holds for all
$k>0$.

The arguments above may be summarised as follows.
\begin{lemma} \label{parabolic}
If $\fg_0=L_0$ as Lie superalgebras
and $\fg_k=L_k$ for all $k>0$ as $\fg_0$-modules,
then Lemma \ref{key} holds.
\end{lemma}

This result will play an essential role in the proof of Lemma
\ref{key}.

\subsubsection{Main theorem}

The following theorem is the main result of this paper.
\begin{theorem}\label{generating}
The Lie superalgebra $\fg(A, \Theta)$ coincides with $L(A, \Theta)$,
or equivalently, the ideal $\fs(A, \Theta)$ of
$\tilde\fg(A, \Theta)$ is equal to the maximal ideal
$\fr(A, \Theta)$.
\end{theorem}
\begin{proof}
Note that Lemma \ref{key} immediately implies the claim. Indeed, we
have already shown in Lemma \ref{s-in-r} that
$\fs(A, \Theta)\subset\fr(A, \Theta)$, and this is an inclusion of
$\Z$-graded ideals of $\tilde\fg(A, \Theta)$. If
$\fs(A, \Theta)\ne\fr(A, \Theta)$, there would exist a surjective
Lie superalgebra homomorphism $\fg(A, \Theta) \longrightarrow
L(A, \Theta)$ with a nonzero kernel. Thus for some $k$, the
degree-$k$ homogeneous components of $L(A, \Theta)$ and
$\fg(A, \Theta)$ are not equal. This contradicts Lemma \ref{key}.
\end{proof}

\subsection{Presentations of Lie superalgebras}\label{Borels} 

Since the generators of $\fs(A, \Theta)$ are known explicitly,
Theorem \ref{generating} provides a presentation for each simple
contragredient Lie superalgebra and $\mathfrak{sl}_{n+1|n+1}$ in an
arbitrary root system. We have the following
result for the Lie superalgebra $L(A, \Theta)$.

\begin{theorem}\label{Serre-g}
The Lie superalgebra $L(A, \Theta)$ is generated by the generators
$e_i, f_i, h_i$ ($1\le i\le r$), where $e_i$ and $f_i$ are odd if
$i\in\Theta$, and even otherwise, subject to\\
the quadratic relations
\begin{eqnarray}  \label{relations-1g}
\begin{aligned}
&{[}h_i, h_j] =0, \\
&[h_i, e_j] =a_{i j} e_j, \quad [h_i, f_j] =-a_{i j} f_{j}, \\
&[e_i, f_j] =\delta_{i j} h_i, \quad \forall i, j;
\end{aligned}
\end{eqnarray}
standard Serre relations
\begin{eqnarray}\label{relations-2g}
\begin{aligned}
&(ad_{e_i})^{1-a_{i j}}(e_j)=0, \\
& (ad_{f_i})^{1-a_{i j}}(f_j)=0,
    \quad \text{for  $i\ne j$,  with $a_{i i}\ne 0$ or $a_{i j}=0$};\\
&[e_t, e_t]=0, \quad  [f_t, f_t]=0, \quad \text{for $a_{t t}=0$};
\end{aligned}
\end{eqnarray}
and higher order Serre relations if the Dynkin diagram of $(A, \Theta)$ contains
any of the following diagrams as full sub-diagrams:
\begin{enumerate}
\item \label{HOS-g-1}
\begin{picture}(75, 15)(0, 7)
\put(10, 6.5){$\times$}
\put(15, 10){\line(1, 0){20}}
{\color{gray} \put(40, 10){\circle*{10}} }
\put(45, 10){\line(1, 0){20}}
\put(62, 6.5){$\times$}
\put(10, 16){\tiny $j$}
\put(40, 16){\tiny $t$}
\put(63, 16){\tiny $k$}
\end{picture}
with $sgn_{j t}sgn_{t k}=-1$, the associated
higher order Serre relations are
\[
[e_t, [e_j, [e_t, e_k]]]=0,  \quad [f_t, [f_j, [f_t, f_k]]]=0;
\]

\item \label{HOS-g-2}
\begin{picture}(80, 20)(0, 7)
\put(10, 6.5){$\times$}
\put(15, 10){\line(1, 0){20}}
{\color{gray}\put(40, 10){\circle*{10}} }
\put(45, 11){\line(1, 0){20}}
\put(45, 9){\line(1, 0){20}}
\put(55, 7){$>$} \put(70,
10){\circle{10}} \put(10, 16){\tiny $j$}
\put(40, 16){\tiny $t$}
\put(70, 16){\tiny $k$}
\put(76, 7){,}
\end{picture}
the associated higher order Serre relations are
\[
[e_t, [e_j, [e_t, e_k]]=0,  \quad [f_t, [f_j, [f_t, f_k]]]=0;
\]


\item \label{HOS-g-3}
\begin{picture}(80, 20)(0, 7)
\put(10, 6.5){$\times$}
\put(15, 10){\line(1, 0){20}}
{\color{gray}\put(40, 10){\circle*{10}} }
\put(45, 11){\line(1, 0){20}}
\put(45, 9){\line(1, 0){20}}
\put(55, 7){$>$}
\put(70, 10){\circle*{10}}
\put(10, 16){\tiny $j$}
\put(40, 16){\tiny $t$} \put(70, 16){\tiny $k$}
\put(76, 7){,}
\end{picture}
the associated higher order Serre relations are
\[
[e_t, [e_j, [e_t, e_k]]]=0,  \quad [f_t, [f_j, [f_t, f_k]]]=0;
\]

\item \label{HOS-g-4}
\begin{picture}(80, 20)(0, 7)
{\color{gray}\put(10, 10){\circle*{10}} }
\put(15, 10){\line(1, 0){20}} {\color{gray}
\put(40, 10){\circle*{10}} }
\put(45, 11){\line(1, 0){20}}
\put(45, 9){\line(1, 0){20}}
\put(46, 7){$<$}
\put(70, 10){\circle{10}}
\put(10, 16){\tiny $j$}
\put(40, 16){\tiny $t$}
\put(70, 16){\tiny $k$}
\put(76, 7){,}
\end{picture}
the associated higher order Serre relations are
\[
\begin{aligned}
{[}[e_j, e_t], [[e_j, e_t], [e_t, e_k]]]=0,  \\
{[}[f_j, f_t], [[f_j, f_t], [f_t, f_k]]]=0;
\end{aligned}
\]

\item \label{HOS-g-5}
\begin{picture}(120, 20)(-10, 7)
\put(0, 6.5){$\times$}
\put(5, 10){\line(1, 0){20}}
\put(30, 10){\circle{10}}
\put(35, 10){\line(1, 0){20}}
{\color{gray}\put(60, 10){\circle*{10}} }
\put(65, 11){\line(1, 0){20}}
\put(65, 9){\line(1, 0){20}}
\put(66, 7){$<$}
\put(90, 10){\circle{10}}
\put(6, 16){\tiny $i$}
\put(30, 16){\tiny $j$}
\put(60, 16){\tiny $t$}
\put(90, 16){\tiny $k$}
\put(96, 7){,}
\end{picture}
the associated higher order Serre relations are
\[
\begin{aligned}
{[[}e_i, [e_j, e_t]], [[e_j, e_t], [e_t, e_k]]]=0,  \\
{[[}f_i, [f_j, f_t]], [[f_j, f_t], [f_t, f_k]]]=0;
\end{aligned}
\]

\item \label{HOS-g-6}
\begin{picture}(65, 30)(0, 7)
\put(10, 6.8){$\times$}
\put(15, 10){\line(1, 1){20}}
\put(15,10){\line(1, -1){20}}
{\color{gray} \put(40, 30){\circle*{10}}}
\put(39, 25){\line(0, -1){30}}
\put(41, 25){\line(0, -1){30}}
{\color{gray} \put(40, -10){\circle*{10}}}
\put(10, 16){\tiny $i$}
\put(47, 28){\tiny $t$}
\put(47, -13){\tiny $s$}
\put(50, 7){,}
\end{picture}
the associated higher order Serre relations are
\[
\begin{aligned}
{[}e_t, [e_s, e_i]]-[e_s, [e_t, e_i]]=0, \\
{[}f_t, [f_s, f_i]]-[f_s, [f_t, f_i]]=0;
\end{aligned}
\]

\item \label{HOS-g-7}
\begin{picture}(125, 20)(-15,-3)
\put(0, 0){\circle{10}}
\put(5, 2){\line(1, 0){20}}
\put(5, 0){\line(1, 0){20}}
\put(5, -2){\line(1, 0){20}}
\put(15, -3.25){$>$}
{\color{gray} \put(30, 0){\circle*{10}}}
\put(35, 1){\line(1, 0){20}}
\put(35, -1){\line(1, 0){20}}

\put(35, -3.25){$<$}

\put(60, 0){\circle{10}}
\put(65, 0){\line(1, 0){20}}
\put(90, 0){\circle{10}}

\put(0, 7){\tiny 1}
\put(30, 7){\tiny 2}
\put(60, 7){\tiny 3}
\put(90, 7){\tiny 4}

\put(96, -2){,}
\end{picture}
the associated higher order Serre relations are
\[
\begin{aligned}
{[} E, [ E, [e_2, [e_3, e_4]]  ] {]}=0, \\
{[} F, [ F, [f_2, [f_3, f_4]]  ]  {]}=0,
\end{aligned}
\]
where $E=[[e_1, e_2], [e_2, e_3]]$ and $F=[[f_1, f_2], [f_2, f_3]]$;

\item \label{HOS-g-8}
\begin{picture}(125, 20)(-15,-3)
\put(0, 0){\circle{10}}
\put(5, 2){\line(1, 0){20}}
\put(5, 0){\line(1, 0){20}}
\put(5, -2){\line(1, 0){20}}
\put(15, -3.25){$>$}
{\color{gray}\put(30, 0){\circle*{10}}}
\put(35, 0){\line(1, 0){20}}
\put(60, 0){\circle{10}}
\put(65, 1){\line(1, 0){20}}
\put(65, -1){\line(1, 0){20}}
\put(65, -3.25){$<$}
\put(90, 0){\circle{10}}

\put(0, 7){\tiny 1}
\put(30, 7){\tiny 2}
\put(60, 7){\tiny 3}
\put(90, 7){\tiny 4}

\put(96, -2){,}
\end{picture}
the associated higher order Serre relations are
\[
\begin{aligned}
{[}[e_1, e_2], [[e_2, e_3], [e_3, e_4]]
-[[e_2, e_3], [[e_1, e_2], [e_3, e_4]]=0,\\
{[}[f_1, f_2], [[f_2, f_3], [f_3, f_4]]
-[[f_2, f_3], [[f_1, f_2], [f_3, f_4]]=0;
\end{aligned}
\]

\item \label{HOS-g-9}
\begin{picture}(100, 25)(-20,-3)
\put(0, 0){\circle{10}}
\put(5, 1){\line(1, 0){20}}
\put(5, -1){\line(1, 0){20}}
\put(15, -3.5){$>$}
{\color{gray}\put(30, 0){\circle*{10}}}
\put(35, 1){\line(1, 0){20}}
\put(35, -1){\line(1, 0){20}} {\color{gray}
\put(60, 0){\circle*{10}}}

\put(0, 7){\tiny k}
\put(30, 7){\tiny t}
\put(60, 7){\tiny j}

\put(70, 0){,}
\end{picture}
the associated higher order Serre relations are
\[
\begin{aligned}
{[} e_t, [e_j,[e_t, e_k]] ]=0,\\
{[} f_t, [f_j, [f_t, f_k]] ]=0;
\end{aligned}
\]

\item \label{HOS-g-10}
\begin{picture}(75, 25)(-15,5)
{\color{gray}\put(5, 10){\circle*{10}}}
\put(5, 17){\tiny i}

\put(10, 10){\line(1, 1){20}}
\put(10, 10){\line(1, -1){20}}
\put(10, 8){\line(1, -1){18}}
{\color{gray}\put(32, 27){\circle*{10}}}
\put(40, 27){\tiny j}

\put(30,-2){\line(0, 1){24}}
\put(32,-2){\line(0, 1){24}}
\put(34,-2){\line(0, 1){24}}
{\color{gray}\put(32,-7){\circle*{10}}}
\put(40,-10){\tiny k}

\put(42,7){,}
\end{picture}
the associated higher order Serre relations are
\[
\begin{aligned}
2[e_i, [e_k, e_j]]+3[e_j, [e_k, e_i]]=0, \\
2[f_i, [f_k, f_j]]+3[f_j, [f_k, f_i]]=0;
\end{aligned}
\]

\item \label{HOS-g-11}
\begin{picture}(100, 20)(10,-5)
{\color{gray}\put(30, 0){\circle*{10}}} \put(35, 0){\line(1,
0){20}} {\color{gray}\put(60, 0){\circle*{10}}} \put(65,
2){\line(1, 0){20}} \put(65, 0){\line(1, 0){20}} \put(65,
-2){\line(1, 0){20}} \put(65, -3.25){$<$} \put(90,
0){\circle{10}}

\put(30, 7){\tiny 1} \put(60, 7){\tiny 2} \put(90, 7){\tiny 3}

\put(100, -5){,}
\end{picture}
the higher order Serre relations are
\[
\begin{aligned}
{[} [e_1, e_2], [ [e_1, e_2], [ [e_1, e_2], [e_2, e_3] ] ]]=0, \\
{[} [f_1, f_2], [ [f_1, f_2], [ [f_1, f_2],[f_2, f_3] ] ]]=0;
\end{aligned}
\]


\item \label{HOS-g-12}
\begin{picture}(100, 20)(10,-3)
\put(30, 0){\circle*{10}}
\put(35, 1){\line(1, 0){20}}
\put(35, -1){\line(1, 0){20}}
\put(35, -3.25){$<$} {\color{gray}
\put(60, 0){\circle*{10}}}
\put(65, 2){\line(1, 0){20}}
\put(65, 0){\line(1, 0){20}}
\put(65, -2){\line(1, 0){20}}
\put(65, -3.25){$<$}
\put(90, 0){\circle{10}}
\put(100, -3){,}

\put(30, 7){\tiny 1}
\put(60, 7){\tiny 2}
\put(90, 7){\tiny 3}

\end{picture}
the higher order Serre relations are
\[
\begin{aligned}
{[}[e_2, e_1], [e_3, [e_2, e_1]]] - [[e_2, e_3], [[e_1, e_1], e_2]]=0, \\
{[}[f_2, f_1], [f_3, [f_2, f_1]]] - [[f_2, f_3], [[f_1, f_1], f_2]]=0;
\end{aligned}
\]

\item \label{HOS-g-13}
\begin{picture}(65, 35)(-15, 0)
\put(5, 0){\circle{10}}
\put(5, 7) {\tiny 1}

\put(10, 1){\line(1, -1){20}}
\put(10, -1){\line(1, -1){20}}
\put(12, -2){\line(2, -1){10}}
\put(12, -2){\line(1, -2){5}}
{\color{gray}
\put(35, -20){\circle*{10}}}
\put(42, -22){\tiny 3}

\put(33, 15){\line(0, -1){30}}
\put(35, 15){\line(0, -1){30}}
\put(37, 15){\line(0, -1){30}}

{\color{gray}
\put(35, 20){\circle*{10}}}
\put(42, 18){\tiny 2}

\put(30, 20){\line(-1, -1){19}}
\put(44, 0){,}
\end{picture}
the higher order Serre relations are
\[
\begin{aligned}
{[}e_2, [e_3, e_1]]-2[e_3, [e_2, e_1]]=0, \\
{[}f_2, [f_3, f_1]]-2[f_3, [f_2, f_1]]=0;
\end{aligned}
\]

\item \label{HOS-g-14}
\begin{picture}(100, 30)(50,10)
{\color{gray}\put(75, 20){\circle*{10}}} \put(80, 20){\line(1,
1){20}} \put(80, 20){\line(1, -1){20}} \put(82, 30){\tiny $1$}
\put(82 , 5){\tiny $\alpha$} \put(105,15){\tiny $-(1+\alpha)$}
{\color{gray}\put(102, 35){\circle*{10}}} \put(102, 5){\line(0,
1){25}} {\color{gray}\put(102, 0){\circle*{10}}}
\put(138, 10){, }
\end{picture}
the higher order Serre relations are
\[
\begin{aligned}
\alpha [e_1, [e_2, e_3]] + (1+\alpha)[e_2, [e_1, e_3]]=0, \\
\alpha [f_1, [f_2, f_3]] + (1+\alpha)[f_2, [f_1, f_3]]=0,
\end{aligned}
\]
where the left node is labeled by $1$, the top node by $2$ and
bottom one by $3$.
\end{enumerate}
\end{theorem}

When $(A, \Theta)$ is given in the distinguished root system, Theorem
\ref{Serre-g} simplifies considerably. We have the following result.

\begin{theorem}\label{Serre-distin}
Let $(A, \Theta)$ with $\Theta=\{s\}$ be the Cartan matrix of a contragredient Lie
superalgebra in the distinguished root system.
Then $L(A, \Theta)$ is generated by generators $e_i, f_i, h_i$
($i=1, 2, \dots, r$), where $e_s$ and $f_s$ are odd and the rest
even, subject to \\
the quadratic relations
\begin{eqnarray}  \label{relations-1}
\begin{aligned}
&{[}h_i, h_j] =0, \\
&[h_i, e_j] =a_{i j} e_j, \quad [h_i, f_j] =-a_{i j} f_{j}, \\
&[e_i, f_j] =\delta_{i j} h_i, \quad \forall i, j;
\end{aligned}
\end{eqnarray}
standard Serre relations
\begin{eqnarray}\label{relations-2}
\begin{aligned}
&(ad_{e_i})^{1-a_{i j}}(e_j)=0, \\
& (ad_{f_i})^{1-a_{i j}}(f_j)=0,
    \quad \text{for \ } i\ne j, \ a_{i i}\ne 0;\\
&[e_s, e_s]=0, \quad  [f_s, f_s]=0, \quad \text{for $a_{ss}=0$};
\end{aligned}
\end{eqnarray}
and higher order Serre relations
\begin{eqnarray}
[e_s, [e_{s-1}, [e_s, e_{s+1}]]]=0,  \quad [f_s, [f_{s-1}, [f_s, f_{s+1}]]]=0,
\end{eqnarray}
if the Dynkin diagram of $A$ contains a full sub-diagram of the form
\begin{center}
\begin{picture}(85, 20)(0, 7)
\put(10, 10){\circle{10}}
\put(15, 10){\line(1, 0){20}}
{\color{gray} \put(40, 10){\circle*{10}} }
\put(45, 10){\line(1, 0){20}}
\put(70, 10){\circle{10}}
\put(5,  18){\tiny $s-1$}
\put(40, 18){\tiny $s$}
\put(60, 18){\tiny $s+1$}
\end{picture}
with $sgn_{s-1,s} sgn_{s,s+1}=-1$, \quad or
\begin{picture}(90, 20)(0, 7)
\put(10, 10){\circle{10}}
\put(15, 10){\line(1, 0){20}}
{\color{gray}
\put(40, 10){\circle*{10}} }
\put(45, 11){\line(1, 0){20}}
\put(45, 9){\line(1, 0){20}}
\put(55, 7){$>$}
\put(70, 10){\circle{10}}
\put(5, 18){\tiny $s-1$}
\put(40, 18){\tiny $s$}
\put(60, 18){\tiny $s+1$}
\put(78, 7){.}
\end{picture}
\end{center}
\end{theorem}

\begin{remark}
Note the importance of the signs $sgn_{i j}$ in the above theorem.
There are higher order Serre relations associated with the first
Dynkin diagram in \eqref{sl22-osp42}, but none with the second. The
Dynkin diagrams in \eqref{sl22-osp42} are respectively those of
$\mathfrak{sl}_{2|2}$ and $\mathfrak{osp}_{4|2}$ in their
distinguished root systems. The Lie superalgebra
$D(2, 1; \alpha)$ in the distinguished root system has no
higher order Serre relations either.
\end{remark}

\section{Proof of key lemma for distinguished root systems}\label{proof-distinguished}

Throughout this section, we assume that the Cartan matrix
$(A, \Theta)$ is associated with the distinguished root system of a
simple Lie superalgebra. Thus $\Theta$ contains only one element,
which we denote by $s$. To simplify notation, we write $\tilde\fg(A)$ for
$\tilde\fg(A, \Theta)$, $\fg(A)$ for
$\fg(A, \Theta)$, and $L(A)$ for $L(A, \Theta)$

\subsection{The proof}

The proof of Lemma \ref{key} will make essential use of Lemma
\ref{parabolic}. Define the $\Z$-gradings for $\fg(A)$ and $L(A)$ as
in Section \ref{Z-grading} by taking $d=s$.

\begin{lemma}\label{g0}
As reductive Lie algebras, $\fg_0=L_0$.
\end{lemma}
\begin{proof}
In this case, both $\fg_0$ and $L_0$ are generated by purely even
elements. Let $\fg_0'=[\fg_0, \fg_0]$ and $L_0'=[L_0, L_0]$ be the
derived algebras. Then by Serre's theorem for
semi-simple Lie algebras $\fg_0'=L_0'$. Now the claim immediately follows.
\end{proof}

We now consider the $\fg_0$-modules $\fg_1$ and $L_1$.
\begin{remark}
For convenience,  we continue to use $e_i$, $h_i$ and $f_i$ to
denote the images of these elements in $\fg(A)$.
\end{remark}
Examine the following relations in $\fg(A)$:
\begin{eqnarray}\label{g1-irrep}
[h_i, e_s]=a_{i s} e_s, \quad [f_i, e_s]=0, \quad
(ad_{e_i})^{1-a_{i s}} e_s =0, \quad \forall i\ne s.
\end{eqnarray}

The first two relations imply that $e_s$ is a lowest weight vector
of the $\fg_0$-module $\fg_1$, with weight $\alpha_s$. Since $a_{i
s}$ are non-positive integers for all $i\ne s$, by \cite[Theorem 21.4]{H},
the third relation implies  that $\fg_1$ is an
irreducible finite dimensional $\fg_0$-module. The relations
\eqref{g1-irrep} also hold in $L(A)$. This immediately shows the
following result.
\begin{lemma}
Both $\fg_1$ and $L_1$ are irreducible $\fg_0$-modules, and
$\fg_1=L_1$.
\end{lemma}

Note that $\fg_2$ is generated by $\fg_1$, that is $\fg_2=[\fg_1,
\fg_1]$. By induction one can show that $\fg_{k+1}=
\left(ad_{\fg_1}\right)^{k}(\fg_1)$ for all $k\ge 1$. If $\fg_i=0$
for some $i>1$, then $\fg_j=0$ for all $j\ge i$. We have the
$\fg_0$-module decomposition $\fg_1\otimes\fg_1= S^2_s(\fg_1)\oplus
\wedge^2_s(\fg_2)$, where $S^2_s(\fg_1)$ denotes the second
$\Z_2$-graded symmetric power, and $\wedge^2_s(\fg_1)$ the second
$\Z_2$-graded skew power, of $\fg_1$.

\begin{remark}\label{notations-S-wedge}
Throughout the paper, we use $S^k_s(V)$ and $\wedge^k_s(V)$ to
denote the $\Z_2$-graded symmetric and skew symmetric tensors of
rank $k$ in the $\Z_2$-graded vector space $V$, and $S^k(V)$ and
$\wedge^k(V)$ to denote the usual symmetric and skew symmetric
tensors of rank $k$, ignoring the $\Z_2$-grading of $V$.
\end{remark}

We have the following result:
\begin{lemma}\label{g2}
The Lie superbracket defines a surjective $\fg_0$-map
$\fg_1\otimes\fg_1\longrightarrow \fg_2$, $X\otimes Y\mapsto [X,
Y]$. The $\fg_0$-submodule $S_s^2(\fg_1)$ is in the kernel of this
map, and $\wedge^2_s(\fg_1)$ is mapped surjectively onto $\fg_2$.
\end{lemma}
\begin{proof}
For any $X, Y\in \fg_1$, an element $Z\in\fg_0$ acts on $X\otimes Y$ by
\[
Z\cdot(X\otimes Y)=[Z, X]\otimes Y + X\otimes [Z, Y].
\]
The Lie superbracket maps $Z\cdot(X\otimes Y)$ to $[[Z, X],  Y] +
[X, [Z, Y]] =[Z, [X, Y]]$. This proves the first claim. The second
claim follows from the $\Z_2$-graded skew symmetry of the Lie
superbracket.
\end{proof}

Therefore, the $\fg_0$-map $\Psi: \wedge^2_s(\fg_1)\longrightarrow\fg_2$
defined by the composition
\[
\wedge^2_s(\fg_1)\hookrightarrow\fg_1\otimes\fg_1\stackrel{[\ , \ ]}{\longrightarrow}
\fg_2
\]
is also surjective, where the map on the left is the natural
embedding. The structure of $\wedge^2_s(\fg_1)$ as a $\fg_0$-module
can be understood; this enables us to understand the structure of
$\fg_2$.

Recall that in the distinguished root systems, $L_2=0$ if $L(A)$ is
of type I, and $L_2\ne 0$ but $L_3=0$ if $L(A)$ is of type II. Thus
in order to show that $\fg_k=L_k$ for all $k>0$, it remains to prove
that $\fg_2=0$ if the Cartan matrix $A$ is of type I, and
$\fg_2=L_2$ and $\fg_3=0$ if $A$ is of type II. In view of Lemma
\ref{parabolic}, the proof of Lemma \ref{key} is done once this is
accomplished.

The rest of the proof will be based on a case by case study. Let us start with
the type I Lie superalgebras.

\subsubsection{The case of $\mathfrak{sl}_{m|n}$}

If the Cartan matrix $A$ is that of $\mathfrak{sl}_{m|n}$, the Lie
superalgebra $\fg(A)$ has
$\fg_0=\mathfrak{gl}_m\oplus\mathfrak{sl}_n$, and $\fg_1\cong
\C^m\otimes \overline{\C}^n$ up to parity change, where $\C^m$
denotes the natural module for $\mathfrak{gl}_m$, and
$\overline{\C}^n$ denotes the dual of the natural module for
$\mathfrak{sl}_n$. Assuming that both $m$ and $n$ are greater than
$1$. Then $\wedge^2_s(\fg_1) =S^2(\C^m)\otimes
S^2(\overline{\C}^n)\oplus
\wedge^2(\C^m)\otimes\wedge^2(\overline{\C}^n)$.

The lowest weight vectors of the irreducible submodules are
respectively given by
\[
\begin{aligned}
 v(2):=&e_s\otimes e_s; \\
v(1^2):=&e_{s-1, s+2}\otimes e_{s, s+1} + e_{s, s+1}\otimes e_{s-1, s+2}\\
&- (e_{s-1, s+1}\otimes e_{s, s+2} + e_{s, s+2}\otimes e_{s-1,
s+1}),
\end{aligned}
\]
where $s=m$, and
\[
\begin{aligned}
&e_{s, s+1}=e_{s}, \quad
e_{s, s+2}=[e_{s}, e_{s+1}],\\
&e_{s-1, s+1}=[e_{s-1}, e_{s}],\quad e_{s-1, s+2}=[e_{s-1}, e_{s,
s+2}].
\end{aligned}
\]

We have $\Psi(v(2))=[e_s, e_s]=0$ by one of the standard
Serre relations.  It follows that the entire irreducible
$\fg_0$-submodule $S^2(\C^m)\otimes S^2(\overline{\C}^n)$ is mapped
to zero. In particular, we have
\begin{eqnarray}\label{L2-relation}
[e_{s-1, s+2}, e_{s, s+1}] + [e_{s-1, s+1}, e_{s, s+2}]=0.
\end{eqnarray}

The first term of \eqref{L2-relation} vanishes by the higher order
Serre relation; this in turn forces the second term to vanish as
well. Hence
\[
\Psi(v(1^2))=[e_{s-1, s+2}, e_{s, s+1}] - [e_{s-1, s+1}, e_{s,
s+2}]=0.
\]
Therefore, $v(1^2)$ is in the kernel of $\Psi$, implying that the
entire submodule $\wedge^2(\C^m)\otimes\wedge^2(\overline{\C}^n)$ is
mapped to zero by $\Psi$. This shows that $\fg_2=0$, and hence
$\fg_k=0$ for all $k\ge 2$.

Note that if $min(m, n)=1$, say, $n=1$, $\wedge^2_s(\fg_1)$ is
irreducible as $\fg_0$-module and is equal to $S^2(\C^m)\otimes\C$.
The above proof obviously goes through but in a much simplified fashion.

Therefore, we have proved that
$\fg_k=L_k$ for all $k\ge 2$ in the case $L(A)=\mathfrak{sl}_{m|n}$.

\subsubsection{The case of $C(n+1)$ with $n>1$}

In this case, $\fg_0=\mathfrak{sp}_{2n}\oplus \C$ and
$\fg_1=\C^{2n}$. The $\Z_2$-graded skew symmetric tensor $\wedge^2_s(\fg_1)$ is an
irreducible $\fg_0$-module with the lowest weight vector $e_1\otimes
e_1$. Since $\Psi(e_1\otimes e_1) =[e_1, e_1]=0$ by the standard
Serre relation, it immediately follows that $\fg_k=0$ for all $k\ge
2$.

\subsubsection{The case of $D(m, n)$ with $m>2$} \label{Dm-distinguished}

In this case, $\fg_0=\mathfrak{gl}_n\oplus \mathfrak{so}_{2m}$, and
$\fg_1$ is isomorphic to $\C^n\otimes \C^{2m}$ as $\fg_0$-module (up
to parity) with $e_n$ being the lowest weight vector. Let us first
assume that $n>1$. Then we have
\[
\begin{aligned}
\wedge^2_s(\fg_1)= S^2(\C^n)\otimes
\frac{S^2(\C^{2m})}{\C}\oplus \wedge^2(\C^n)\otimes \wedge^2(\C^{2m})
\oplus S^2(\C^n)\otimes\C.
\end{aligned}
\]
Lowest weight vectors of the first two irreducible submodules can be
explicitly constructed in exactly the same way as in the case of
$\mathfrak{sl}_{m|n}$. The same arguments used there also show that
the Lie superbracket maps both submodules to zero. Hence $\fg_2\cong
S^2(\C^n)\otimes\C$. Inspecting the roots of $D(m, n)$
given in Appendix \ref{roots}, we can see that $\fg_2=L_2$.

Let us examine $\fg_2$ in more detail. We use notation from
Appendix \ref{roots} for roots of the Lie superalgebra $D(m, n)$. Let
$X_{\delta_i\pm\epsilon_p}$, where $1\le i\le n$ and $1\le p\le m$,
be a weight basis of $\fg_1$. Then in $\fg_2$, we have
\begin{eqnarray}\label{D-g2-weight-vectors}
\begin{aligned}
&{[}X_{\delta_i-\epsilon_p}, X_{\delta_j-\epsilon_q}]
={[}X_{\delta_i+\epsilon_p}, X_{\delta_j+\epsilon_q}]=0, \quad \forall i, j, p, q, \\
&{[}X_{\delta_i+\epsilon_p}, X_{\delta_j-\epsilon_q}]=0,  \quad \forall i, j, p\ne q,
\end{aligned}
\end{eqnarray}
and there exist scalars $c_{i j, p q}$ such that
\[
[X_{\delta_i-\epsilon_p}, X_{\delta_j+\epsilon_p}]
=c_{i j, p q}[X_{\delta_i+\epsilon_p}, X_{\delta_j-\epsilon_p}]\ne 0,
\quad \forall i, j, p, q.
\]
By multiplying the elements $X_{\delta_i\pm\epsilon_p}$ by
appropriate scalars if necessary, we may assume
\[
[X_{\delta_i-\epsilon_p}, X_{\delta_j+\epsilon_p}]
=[X_{\delta_i-\epsilon_q}, X_{\delta_j+\epsilon_q}],
\quad  \forall i, j, p, q,
\]
which we denote by $X_{\delta_i+\delta_j}$. Then the subset of
$X_{\delta_i+\delta_j}$ with $1\le i\le j\le n$ forms a basis of
$\fg_2$.

Now we consider $\fg_3$. It immediately follows from
\eqref{D-g2-weight-vectors} that $[X_{\delta_i+\delta_j},
X_{\delta_k\pm\epsilon_p}]=0$ for all $k, p$ and $i\le j$, that is,
\begin{eqnarray}\label{deltax3}
\fg_3=[\fg_1, \fg_2]=0.
\end{eqnarray}
Hence $\fg_k=0$ for all $k\ge 3$.

When $n=1$, the proof goes through much more simply. This completes
the proof of Lemma \ref{key} for the case of $D(m, n)$ with $m>2$.

In contrast to the type I case, the complication here is that
$\fg_3$ needs to be analysed separately as $\fg_2\ne 0$.

\subsubsection{The case of  $D(2, n)$}\label{D2-distinguished}
In this case,
$\fg_0=\mathfrak{gl}_n\oplus\mathfrak{sl}_2\oplus\mathfrak{sl}_2$,
and $\fg_1=\C^n\otimes\C^2\otimes\C^2$. The $\Z_2$-graded skew
symmetric rank two tensor $\wedge^2_s(\fg_1)$ decomposes into the
direct sum of four irreducible $\fg_0$-modules if $n>1$:
\[
\begin{aligned}
\wedge^2_s(\fg_1) = &L^n_{(2)}\otimes L^2_{(2)}\otimes L^2_{(2)}
            \oplus L^n_{(1, 1)}\otimes L^2_{(2)}\otimes L^2_{(0)}\\
            &\oplus L^n_{(1, 1)}\otimes L^2_{(0)}\otimes L^2_{(2)}
            \oplus L^n_{(2)}\otimes L^2_{(0)}\otimes L^2_{(0)}.
\end{aligned}
\]
If $n=1$, then $L^n_{(1, 1)}=0$, the two modules in the middle are absent.

The lowest weight vectors of the first three submodules can be
easily worked out. Below we give the explicit formulae for their
images under the Lie superbracket. Let
\[
\begin{aligned}
&e_{s; s+1} = [e_s, e_{s+1}], \quad e_{s; s+2} = [e_s, e_{s+2}],
\quad e_{s-1; s} = [e_{s-1}, e_s], \\
&e_{s-1; s+1} = [e_{s-1}, e_{s; s+1}], \quad  e_{s-1; s+2} =
[e_{s-1}, e_{s; s+2}].
\end{aligned}
\]
Then the images of the lowest weight vectors are given by
\begin{eqnarray}\label{B-g2}
[e_s, e_s], \quad [e_{s-1; s+1}, e_s]-[e_{s-1; s}, e_{s, s+1}],
\quad [e_{s-1; s+2}, e_s]-[e_{s-1; s}, e_{s, s+2}].
\end{eqnarray}

We have the Serre relation $[e_s, e_s]=0$. This implies that the
entire irreducible submodule $L^n_{(2)}\otimes L^2_{(2)}\otimes
L^2_{(2)}$ is mapped to zero by the Lie superbracket.

In the case $n>1$, this in particular implies
\[
[e_{s-1; s+1}, e_s]+[e_{s-1; s}, e_{s, s+1}]=0, \quad [e_{s-1; s+2},
e_s]+[e_{s-1; s}, e_{s, s+2}]=0.
\]
Note that $[e_{s-1; s+1}, e_s]=0$ and $[e_{s-1; s+2}, e_s]=0$ are
the two higher order Serre relations involving $e_s$. Thus all the
four terms on the left hand sides of the above equations should
vanish separately. It then follows that the second and third
elements in \eqref{B-g2} are zero, that is, the lowest weight
vectors of the irreducible submodules $L^n_{(1, 1)}\otimes
L^2_{(2)}\otimes L^2_{(0)}$ and $L^n_{(1, 1)}\otimes
L^2_{(0)}\otimes L^2_{(2)}$ are in the kernel of the Lie
superbracket. Thus both irreducible submodules are mapped to zero by
the Lie superbracket. The above analysis in vacuous if $n=1$.

Therefore, $\fg_2\cong L^n_{(2)}\otimes L^2_{(0)}\otimes L^2_{(0)}$,
and this shows that $\fg_2\cong L_2$.

To analyse $\fg_3$, we note that equation \eqref{deltax3} still
holds here as can be shown by adapting the arguments in
the $m>2$ case. This completes the proof in this case.

\subsubsection{The case of $B(m, n)$}
When $m\ge 1$, the proof  is much the same as in the case of $D(m,
n)$ with $m>2$. We omit the details.

If $m=0$, then $\fg_0=\mathfrak{gl}_n$, $\fg_1=\C^n$ and $\fg_2\cong
\wedge^2_s(\fg_1)\cong L_2$. Every root vector in $\fg_1$ is of the
form $[X, e_s]$ for some positive root vector $X\in \fg_0$, where
$s=n$. Thus it follows from the relation $(ad_{e_s})^3(e_{s-1})=0$
that $[\fg_1, [e_s, e_s]]=0$. Since $[e_s, e_s]$ is a $\fg_0$ lowest
weight vector of $\fg_2$, this implies $\fg_3=0$.

\begin{remark}\label{osp=so}
The Lie superalgebra $B(0, n)$ is essentially the same as the
ordinary Lie algebra $B_n$. As a matter of fact, the corresponding
quantum supergroup is isomorphic to the smash product of ${\rm
U}_q(B_n)$ with the group algebra of $\Z_2^n$ \cite{Z92, Lan}. The usual proof of
Serre presentations for semi-simple Lie algebras (see, e.g.,
\cite{H}) works for $B(0, n)$. We gave the alternative proof here
for the sake of uniformity.
\end{remark}

\subsubsection{The case of $F(4)$}

Let us order the nodes in the Dynkin diagram from the right to left:
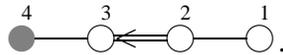
\begin{center}
\begin{picture}(110, 15)(30, -5)
{\color{gray} \put(40, 0){\circle*{10}} }
\put(45, 0){\line(1, 0){20}}
\put(70, 0){\circle{10}}
\put(75, -1){\line(1, 0){20}}
\put(75, 1){\line(1, 0){20}}
\put(75, -3){$<$}
\put(100, 0){\circle{10}}
\put(105, 0){\line(1, 0){20}}
\put(130, 0){\circle{10}}

\put(40, 7){\tiny 4}
\put(70, 7){\tiny 3}
\put(100, 7){\tiny 2}
\put(130, 7){\tiny 1}
\end{picture}.
\end{center}
We may express the simple roots as $\alpha_1=\epsilon_1-\epsilon_2$,
$\alpha_2=\epsilon_2-\epsilon_3$, $\alpha_2=\epsilon_3$ and  $
\alpha_4 = \frac{1}{2}
\Big(\delta-\epsilon_1-\epsilon_2-\epsilon_3\Big)$. The symmetric
bilinear form on the weight space is defined in Appendix
\ref{roots}, where further details about roots of $F(4)$ are given.

The first three simple roots are the standard simple roots of
$\mathfrak{so}_7$, thus $\fg_0=\mathfrak{so}_7\oplus
\mathfrak{gl}_1$. The subspace $\fg_1$ is an irreducible
$\fg_0$-module, which has $e_4$ as a lowest weight vector, and
restricts to the spinor module for $\mathfrak{so}_7$. Now
$\wedge^2_s(\fg_1)$ decomposes into the direct sum of two
irreducibles $\fg_0$-submodules, one of which is $1$-dimensional,
the other is $35$-dimensional with lowest weight vector $e_4\otimes
e_4$.

The Serre relation  $[e_4, e_4]=0$ implies that the $35$-dimensional
submodule is in the kernel of the Lie superbracket, and hence
$\fg_2$ is $1$-dimensional. A basis element for $\fg_2$ is $E=[e_4,
e_{\alpha_4+\epsilon_1+\epsilon_2+\epsilon_3}]$.

For any weight $\beta$ of $\fg_1$, we use $e_\beta\in\fg_1$ to
denote a basis vector of the associated weight space, and set
$e_{\alpha_4}=e_4$. Then we have
\begin{eqnarray}
[e_\beta, E]=0, \quad \text{for all odd positive root $\beta$.}
\end{eqnarray}
This is trivially true for $\beta=\alpha_4$ or
$\alpha_4+\epsilon_1+\epsilon_2+\epsilon_3$.  For
$\beta=\alpha_4+\epsilon_1$ or $\alpha_4+\epsilon_i+\epsilon_j$
($i\ne j$), we have $[e_\beta, E]=[[e_\beta, e_{\alpha_4}],
e_{\alpha_4+\epsilon_1+\epsilon_2+\epsilon_3} ] - [ e_{\alpha_4},
[e_\beta, e_{\alpha_4+\epsilon_1+\epsilon_2+\epsilon_3}] ]$, where
both terms vanish as they involve images in $\fg_2$ of elements in
the $35$-dimensional submodule of $\wedge^2_s(\fg_1)$. Therefore,
$\fg_k=\{0\}$ for all $k\ge 3$.

\subsubsection{The case of $G(3)$}

In this case, $\fg_0$ is isomorphic to the reductive Lie algebra
$G_2\oplus \mathfrak{gl}_1$, and $\fg_1$ is an irreducible
$\fg_0$-module which restricts to the $7$-dimensional irreducible
$G_2$-module. The $\Z_2$-graded skew symmetric tensor $\wedge^2_s(\fg_1)$
decomposes into the direct sum $L(2\alpha_1) \oplus L(0)$ of two
irreducible $\fg_0$-submodules. The submodule
$L(2\alpha_1)$ has $e_1\otimes e_1$ as
lowest weight vector, thus its image under the Lie superbracket is zero
by the Serre relation $[e_1, e_1]=0$.
The submodule $L(0)$ is $1$-dimensional. Since the Lie
superbracket maps $\wedge^2_s(\fg_1)$ surjectively to $\fg_2$, we
immediately conclude that $\dim\fg_2=1$.

Let $X=e_{2\alpha_2+\alpha_3}$ be the root vector of $G_2\subset\fg$
associated with the positive root $2\alpha_2+\alpha_3$. Then
$e^+:=[X, [X, e_1]]$ is the highest weight vector of $\fg_1$ as a
$\fg_0$-module. Since $\fg_2$ is one-dimensional, it must be spanned
by $E=[e_1, e^+]$.

If $e_\beta\in \fg_1$ is a weigh vector not proportional to $e_1$ or
$e^+$, both $[e_\beta, e_1]$ and $[e_\beta, e^+]$ vanish since
they lie in the image of $L(2\alpha_1)\subset \wedge^2_s(\fg_1)$
under the Lie superbracket. Hence $[e_\beta, E]=0$. We also have
$[e^+, e^+]=0$, and the Serre relation $[e_1, e_1]=0$. Thus $[e_1,
E]=[e^+, E]=0$. Therefore, $[\fg_1, E]=0$, which implies
$\fg_k=\{0\}$, for all $ k\ge 3$.

\subsubsection{The case of $D(2, 1; \alpha)$}

We have
$\fg_0=\mathfrak{sl}_2\oplus\mathfrak{sl}_2\oplus\mathfrak{gl}_1$,
and $\fg_1\cong \C^2\otimes\C^2$. The tensor $\wedge^2_s(\fg_1)$
decomposes into the direct sum of two irreducible
$\fg_0$-submodules,
\[
\wedge^2_s(\fg_1) =L_{(2;2)}\oplus L_{(1^2;1^2)}, \quad
L_{(2;2)}=L_{(2)}\otimes L_{(2)}, \quad
L_{(1^2;1^2)}=L_{(1^2)}\otimes L_{(1^2)}.
\]
The notation here only reflects the
$\mathfrak{sl}_2\oplus\mathfrak{sl}_2$-module structure, as there is
no need to specify the $\mathfrak{gl}_1$-action explicitly (see
Remark \ref{D21alpha-viz-D21} below).

We have $\dim L_{(2;2)}=9$ and $\dim L_{(1^2;1^2)}=1$. The lowest
weight vector for $L_{(2;2)}$ is $v(2)=e_1\otimes e_1$. Let
\[
\begin{aligned}
e_{--}=e_1, \quad e_{+-}=[e_1, e_2], \quad e_{-+}=[e_1, e_3], \quad
e_{++}=[e_{+-}, e_3], \\
v(1^2)=e_{--}\otimes e_{++} + e_{++} \otimes e_{--}
       -e_{+-}\otimes e_{-+}-e_{-+} \otimes e_{+-}.
\end{aligned}
\]
Then the vector $v(1^2)$ spans $L_{(1^2;1^2)}$.

The Lie superbracket maps $L_{(2;2)}$ to zero because
$[e_1, e_1]=0$. Note that $[e_{--}, e_{++}] +[e_{+-},
e_{-+}]$ belongs to the image of $L_{(2;2)}$, thus is zero. Hence
$\fg_2$ is spanned by $E=[e_{--}, e_{++}]$. Now it is easy to show
that $[E, \fg_1]=0$.

\begin{remark}\label{D21alpha-viz-D21}
This proof is essentially the same as that in the case of $D(2, 1)$,
except for that the $\mathfrak{gl}_1$ subalgebra of $\fg_0$ acts on
$\fg_1$ by different scalars in the two cases. However, this scalar
is not important in the proof of Lemma \ref{key}, and that is the
reason why we did not specify it explicitly.
\end{remark}

\subsection{Comments on the proof}\label{comments-proof}

Let us recapitulate the proof of Lemma \ref{key} in the distinguished root
systems.
\begin{enumerate}
\item \label{comments-1} By Lemma \ref{parabolic},  the proof of Lemma \ref{key}
    is reduced to showing that the
    parabolic subalgebras  $\fg(A)_{\ge 0}$ and $L(A)_{\ge 0}$
    are the same.

\item \label{comments-2} The elements $\{h_s\}\cup\{h_i, e_i, f_i \mid i\ne
s\}$  and those defining relations of $\fg(A)$ obeyed by them
    give a Serre presentation for the reductive Lie algebra
    $\fg_0$. Then it essentially follows from Serre's theorem
    that $\fg_0=L_0$, see Lemma \ref{g0}.

\item \label{comments-3} Given item (\ref{comments-2}), it suffices to show that
    $\fg(A)_{>0}=\oplus_{k>0}\fg_k$ and
    $L(A)_{>0}=\oplus_{k>0}L_k$ are isomorphic as
    $\fg_0$-modules.

\item Equation \eqref{g1-irrep} gives the necessary and sufficient conditions
    for $\fg_1$ to be a finite dimensional irreducible
    $\fg_0$-module with lowest weight $\alpha_s$, hence
    $\fg_1=L_1$ as $\fg_0$-modules.

\item The standard and higher order Serre relations involving $e_s$ are
    conditions imposed on $\fg_0$-lowest weight vectors of
    $[\fg_1, \fg_1]$, which are the necessary and sufficient to
    guarantee that $\fg_2=L_2$.

\item  The fact that $\fg_3=0$ follows (trivially in the type I case) from
    the result on $\fg_2$ and graded skew symmetry of the Lie
    superbracket, thus no additional relations are required. The
    vanishing of $\fg_3$ implies that for all $k\ge 3$,
    $\fg_k=0$, and hence $\fg_k=L_k$.
\end{enumerate}

In non-distinguished root systems, one can still prove Lemma
\ref{key} by following a similar strategy, as we shall see in the
next section. However, there are important differences in several
aspects.

There are many such $\Z$-gradings as defined in
    Section \ref{Z-grading} for the Lie superalgebras
    $\fg(A, \Theta)$ and $L(A, \Theta)$. This works
    to our advantage.

Given any such $\Z$-grading $\fg(A, \Theta)=\oplus_{k\in\Z}\fg_k$,
    the degree zero subspace $\fg_0$ forms a Lie superalgebra,
    which is not an ordinary Lie in general. Thus the
    requirement that $\fg_1$ be an irreducible $\fg_0$-module is
    much more difficult to implement, and usually leads to
    unfamiliar higher order Serre relations.

In general $\fg_3\ne 0$. In order for $\fg_k$ to be equal
    to $L_k$ for $k\ge 3$, higher order Serre relations are
    needed at degree $k\ge 3$.

\section{Proof of key lemma for non-distinguished root systems}
\label{proof-non-distinguished}

In this section we prove Lemma \ref{key} in non-distinguished root
systems by following a similar strategy as that in Section
\ref{proof-distinguished}. In particular, Lemma \ref{parabolic} will
be used in an essential way.

Assume that the Cartan matrix $A$ is of size $r\times r$. Fix a positive integer
$d\le r$, we consider the corresponding $\Z$-gradings for
$\fg(A, \Theta)$ and $L(A, \Theta)$ defined in Section
\ref{Z-grading}. We shall first establish that $\fg_0=L_0$. Since the roots of
$L(A, \Theta)$ are known explicitly (see Appendix \ref{roots}),
we have a complete
understanding of the $\fg_0$-module structure of every $L_k$. Thus
once we have a description of the weight spaces of each $\fg_k$ as
$\fg_0$-module for all $k>0$, an easy comparison with the root spaces
of $L_k$ will enable us to prove the key lemma.

\begin{remark}
In the proof of Lemma \ref{key} given below, we shall only describe
the weight spaces of $\fg_k$ ($k>0$), and leave out the easy step of
comparing them with those of $L_k$ in most cases.
\end{remark}

For convenience, we introduce the parity map $p: \{1, 2, \dots, r\}
\longrightarrow \{0, 1\}$ such that $p(i)=1$ if $i\in\Theta$ and
$p(i)=0$ otherwise. Then $e_i$ and $f_i$ are odd if $p(i)=1$, and
even if $p(i)=0$.

\subsection{Proof in type $A$}\label{A-nondistingushed}

We use induction on the rank $r$ together with the help of Lemma
\ref{parabolic} to prove Lemma \ref{key} and
Theorem \ref{Serre-g}.

If $r=2$, the Dynkin diagram in the non-distinguished root system
has two grey nodes. In this case, there exists no relation between
$e_1$ and $e_2$, and $[e_1, e_2]$ is another positive root vector.
Note that $[e_1, [e_1, e_2]]=0$ and $[e_2, [e_1, e_2]]=0$ by the
graded skew symmetry of the Lie superbracket. Thus Lemma \ref{key}
is valid and $\fg(A, \Theta)=L(A, \Theta)$

When $r>2$, we take $d=r$. Then $\fg'_0=[\fg_0, \fg_0]$ is a special linear
superalgebra of rank $r-1$ by the induction hypothesis, and thus
$\fg_0$ is a general linear superalgebra.

Define the following elements of $\fg_0$:
\begin{eqnarray}\label{matrix-units}
X_{i j} = ad_{e_i}\cdots
ad_{e_{j-2}}(e_{j-1}),\quad  i<j\le r,
\end{eqnarray}
where $X_{j, j+1}=e_j$. In view of the general linear superalgebra
structure of $\fg_0$, we conclude that $\fg_1$ is isomorphic to the
irreducible $\fg_0$-module with lowest weight $\alpha_r$ (which is
in fact the natural module possibly upon a parity change) if and
only if
\[
[X_{i k}, [X_{j r}, e_r]]=0, \quad j\ne k.
\]
By using the $\fg_0$-action, we can show that these conditions are
equivalent to the relation
\begin{eqnarray} \label{gl-HS}
{[}e_{r-1}, [[e_{r-2}, e_{r-1}], e_r]]=0
\end{eqnarray}
and the relevant relations in \eqref{g-tilde}. For $p(r-1)=1$,
\eqref{gl-HS} is a higher order Serre relation associated with the
sub-diagram
\begin{picture}(75, 20)(-5,-3)
\put(0, -3){$\times$}
\put(5, 0){\line(1, 0){20}}
{\color{gray} \put(30, 0){\circle*{10}}}
\put(35, 0){\line(1, 0){20}}
\put(55, -3){$\times$}
\put(0, 7){\tiny r-2}
\put(24, 7){\tiny r-1}
\put(55, 7){\tiny r}
\end{picture}
with $sgn_{r-2, r-1}=-sgn_{r-1, r}$.
If $p(r-1)=0$, it can be derived from
\begin{eqnarray}\label{gl-HS1}
[e_{r-1}, [e_{r-1}, e_r]]=0,
\end{eqnarray}
which is a standard Serre relation.

Consider $\wedge^2_s\fg_1$, which is an irreducible
$\fg_0$-module. The lowest weight vector is
\[
\begin{aligned}
&e_r\otimes e_r, \quad \text{if $p(r)=1$}, \quad \text{or} \\
&e_r\otimes [e_{r-1}, e_r] - [e_{r-1}, e_r]\otimes e_r, \quad
\text{if $p(r)=0$}.
\end{aligned}
\]
Thus $\fg_2=0$ if and only if
\begin{eqnarray}\label{gl-g2}
\begin{aligned}
&[e_r, e_r]=0, \quad \text{if $p(r)=1$}, \quad \text{or} \\
&[e_r, [e_r, e_{r-1}]]=0, \quad \text{if $p(r)=0$},
\end{aligned}
\end{eqnarray}
both of which are standard Serre relations. This proves that Lemma
\ref{key}, and hence Theorem \ref{generating}, are valid at rank
$r$.

\begin{remark}\label{Lie-algebras}
The proof presented here includes an alternative proof for Serre's
theorem in the case of $\mathfrak{sl}_n$. This can be generalised to
all finite dimensional simple Lie algebras. In particular, the proof
for the other classical Lie algebras can be extracted
from the next two sections.
\end{remark}
\subsection{Proof in type $B$}\label{B-nondistingushed}

Consider the first Dynkin diagram of type $B$ in Table 2, where the
last (that is, $r$-th) node is white, and take $d=r$. In this case,
$\fg_0$ is a general linear superalgebra, and we have already
obtained a Serre presentation for it in Section
\ref{A-nondistingushed}.

We require $\fg_1$ be isomorphic to the irreducible $\fg_0$-module
with lowest weight $\alpha_r$, which is in fact the natural module
for $\fg_0$. This is achieved by relations formally the same as
\eqref{gl-HS} or \eqref{gl-HS1}.

As $\fg_0$-module,  $\fg_2$ is isomorphic to $\wedge_s^2\fg_1$,
which is irreducible with the lowest weight vector $E:= [e_r, [e_r,
e_{r-1}]]$. Now $\fg_3=0$ if and only if $[E, \fg_1]=0$. This in
particular requires that
\begin{eqnarray}\label{B-S}
\left(ad_{e_r}\right)^3(e_{r-1})=0.
\end{eqnarray}
We shall show that this in fact is the necessary and sufficient
condition.

If $p({r-1})=1$, then $[E, [e_{r-1}, e_r]]=0$ trivially since
$[e_{r-1}, e_{r-1}]=0$ in $\fg_0$. For $K=[e_{r-2}, [e_{r-1},
e_r]]$, we also have $[K, E]=0$. This follows from $[K, e_{r-1}]=0$,
which is one of the higher order Serre relations associated with a
sub-diagram of type $A$. Applying $ad_{e_r}$ to it twice and using
\eqref{B-S}, we obtain the desired relation. These relations imply
that $[X, E]=0$ for all $X\in \fg_1$ in this case. If $p({r-1})=0$,
the  fact that $[X, E]=0$, for all $X\in \fg_1$, follows from
\[
[[e_{r-1}, e_r], [[e_{r-1}, e_r], e_r]]=0,
\]
which can be derived from \eqref{B-S}.

The other Dynkin diagram (where the last node is black) can be
treated in essentially the same way. We omit the details.

\subsection{Proof in types $C$ and $D$}\label{CD-nondistingushed}

The Dynkin diagrams of type $C$ formally have the same forms as two of the
Dynkin diagrams of $D$. The only difference is in the numbers of grey nodes,
see Remark \ref{rem:C-D-diagrams}. This enables us to treat both types of Lie
superalgebras simultaneously.

\subsubsection{Case 1}

Consider the Dynkin diagram
\begin{center}
\begin{picture}(130, 20)(-5,-5)
\put(0, -3){$\times$}
\put(6, 0){\line(1, 0){20}}
\put(24, -3){$\times$}
\put(30, 0){\line(1, 0){10}}
\put(41, 0){\dots}
\put(54, 0){\line(1, 0){10}}
\put(62, -3){$\times$}
\put(68, 0){\line(1, 0){20}}
\put(86, -3){$\times$}
\put(94, 1){\line(1, 0){20}}
\put(94, -1){\line(1, 0){20}}
\put(95, -3.2){$<$}
\put(120, 0){\circle{10}}
\end{picture}.
\end{center}
We label the nodes from left to right, thus $r$-the node is the one
at the right end. Set $d=r$, then $\fg_0$ is a general linear
superalgebra.

As a $\fg_0$-module, $\fg_1$ is generated by $e_r$. We require it be
isomorphic to the irreducible module $\overline{L}_{\alpha_r}$ with
lowest weight $\alpha_r$. Appendix \ref{S2-gl} describes the
structure of the generalised Verma module $\overline{V}_{\alpha_r}$
with lowest weight $\alpha_r$ and the irreducible quotient
$\overline{L}_{\alpha_r}$. We immediately
see that the relevant relations in
\eqref{g-tilde} and the relations
\begin{eqnarray}\label{C1-HS}
{[} X_{i r}, [X_{j r}, [X_{k r}, e_r]]]=0, \quad \forall i\le j\le k\le r-1,
\end{eqnarray}
are necessary and sufficient conditions to guarantee that
$\fg_1\cong \overline{L}_{\alpha_r}$. Here $X_{i r}$ are elements
of $\fg_0$ defined by \eqref{matrix-units}. The conditions
\eqref{C1-HS} are equivalent to
\begin{eqnarray}
\begin{aligned}
&[e_{r-1}, [e_{r-1}, [e_{r-1}, e_r]]]=0,
&\quad& \text{if $e_{r-1}$ is even},\\
&{[} X_{r-2, r}, [X_{r-2, r}, [e_{r-1}, e_r]]]=0,
&\quad& \text{if $e_{r-1}$, $e_{r-2}$ are both odd}, \\
&{[} X_{r-3, r}, [X_{r-2, r}, [e_{r-1}, e_r]]]=0,
&\quad& \text{if $e_{r-1}$ is odd, $e_{r-2}$ is even}
\end{aligned}
\end{eqnarray}
because of the $\fg_0$-action. Here Remark \ref{D-small-r} is also
in force.

Note that the different situations where the relations apply are
mutually exclusive. The first relation is a standard Serre relation.
The second and third are higher order Serre relations respectively
associated with the sub-diagrams
\begin{center}
\begin{picture}(80, 20)(5, 5)
{\color{gray}\put(10, 10){\circle*{10}} }
\put(15, 10){\line(1, 0){20}}
{\color{gray}\put(40, 10){\circle*{10}} }
\put(45, 11){\line(1, 0){20}}
\put(45, 9){\line(1, 0){20}}
\put(46, 7){$<$}
\put(70, 10){\circle{10}}
\end{picture}
\quad or \quad
\begin{picture}(120, 20)(-25, 5)
\put(-20, 6.5){$\times$}
\put(-15, 10){\line(1, 0){20}}
\put(10,
10){\circle{10}}
\put(15, 10){\line(1, 0){20}}
{\color{gray}\put(40, 10){\circle*{10}} }
\put(45, 11){\line(1, 0){20}}
\put(45, 9){\line(1, 0){20}}
\put(46, 7){$<$}
\put(70, 10){\circle{10}}
\put(76, 7){.}
\end{picture}
\end{center}

Recall that $\fg_2$ is the image of $\wedge^2_s\fg_1$ under the Lie
superbracket. As $\fg_0$-module, $\wedge^2_s\fg_1$ is irreducible
with the lowest weight vector $e_r\otimes [e_{r-1}, e_r] -[e_{r-1},
e_r]\otimes e_r$. Thus $\fg_2=0$ if and only if
\begin{eqnarray}
[e_r, [e_r, e_{r-1}]]=0.
\end{eqnarray}
This is again a standard Serre relation.

\subsubsection{Case 2}

Now we consider the case with the Dynkin diagram
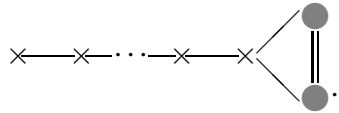
\begin{center}
\begin{picture}(128, 35)(-5,-15)
\put(0, -3){$\times$}
\put(6, 0){\line(1, 0){20}}
\put(24, -3){$\times$}
\put(30, 0){\line(1, 0){10}}
\put(41, 0){\dots}
\put(54, 0){\line(1, 0){10}}
\put(62, -3){$\times$}
\put(68, 0){\line(1, 0){20}}
\put(86, -3){$\times$}
\put(111, 17){\line(-1, -1){16}}
\put(111, -17){\line(-1, 1){16}}
{\color{gray}\put(117, 15){\circle*{10}}}
\put(116,-10){\line(0, 1){19}}
\put(118,-10){\line(0, 1){19}}
{\color{gray}\put(117, -16){\circle*{10}}}
\end{picture}.
\end{center}

Let us first assume that $r=3$.  We have the Dynkin diagram of
$\mathfrak{osp}_{2|4}$ (resp. $\mathfrak{osp}_{4|2}$) if $p(1)=0$
(resp. $p(1)=1$). Label by $1$ the node marked by $\times$, and take
$d=1$. The diagram obtained by deleting this node is
\begin{center}
\begin{picture}(45, 15)(0,-5)
{\color{gray}\put(5, 0){\circle*{10}}}
\put(10,-1){\line(1, 0){20}}
\put(10,1){\line(1, 0){20}}
{\color{gray}\put(35, 0){\circle*{10}}}
\end{picture}.
\end{center}
This is a non-standard diagram of
$\mathfrak{osp}_{2|2}\cong\mathfrak{sl}_{2|1}$. Equation
\eqref{g-tilde} by itself suffices to define this Lie superalgebra.

Now $\fg_0=\mathfrak{osp}_{2|2}\oplus\mathfrak{gl}_1$ (isomorphic to
$\mathfrak{gl}_{2|1}$). Let $\overline{\fb}_0$ be the Borel
subalgebra of $\fg_0$ generated by $f_2, f_3$ and all $h_i$, and
define the lowest weight Verma module
$\overline{V}_{\alpha_1}:=U(\fg_0)\otimes_{U(\fb^-_0)}\C_{\alpha_1}$
for $\fg_0$,  where $\C_{\alpha_1}$ is the irreducible
$\overline{\fb}_0$-module with lowest weight $\alpha_1$. Direct
computations show that the maximal submodule $M_{\alpha_1}$ is
generated by the vector $(e_2 e_3 - e_3 e_2)\otimes 1$. The
irreducible quotient $\overline{L}_{\alpha_1}$ is four dimensional,
with a basis consisting of the images of $1\otimes 1$, $e_2 \otimes
1$, $e_3 \otimes 1$, and $[e_2, e_3]\otimes 1$. Its restriction to
$\mathfrak{osp}_{2|2}$ is the natural module.

We need $\fg_1\cong\overline{L}_{\alpha_1}$, possibly up to a parity
change depending on the parity of $e_1$. From the description of
$\overline{V}_{\alpha_1}$ and $M_{\alpha_1}$ above, we see that the
necessary and sufficient conditions are the relevant quadratic
relations involving $e_1$ in \eqref{g-tilde}, and
\begin{eqnarray}\label{D-triangle}
[e_2, [e_3, e_1]]-[e_3, [e_2, e_1]]=0.
\end{eqnarray}
Note that this is a higher order Serre relation associated with the
sub-diagram (\ref{HOS-g-6}) given in Theorem \ref{Serre-g}.

To proceed further, we need to specify the parity of $e_1$.

If $e_1$ is even, the Lie superalgebra $L(A, \Theta)$ is
$\mathfrak{osp}_{2|4}$. Now $\wedge^2_s\fg_1$ is the direct sum of a
seven dimensional indecomposable $\fg_0$-submodule and a one
dimensional $\fg_0$-submodule. The seven dimensional submodule is
generated by the two lowest weight vectors
\[
e_1\otimes [e_2, e_1] - [e_2, e_1]\otimes e_1,
\quad e_1\otimes [e_3, e_1]-  [e_3, e_1]\otimes e_1,
\]
and the one dimensional submodule by
\[
[e_2, e_1]\otimes[e_3, e_1]+ [e_3, e_1]\otimes[e_2, e_1]
+ e_1\otimes [[e_2, e_3], e_1] -  [[e_2, e_3], e_1]\otimes e_1.
\]
In this case, we need $\fg_2$ to be isomorphic to a one dimensional
$\fg_0$-module with weight $2\alpha_1+\alpha_2+\alpha_3$. Thus the
seven dimensional indecomposable submodule of $\wedge^2_s\fg_1$ is
sent to zero by the Lie superbracket, or equivalently,
\begin{eqnarray}\label{D-dim7}
[e_1, [e_1, e_2]]=0, \quad [e_1, [e_1, e_3]]=0,
\end{eqnarray}
which are standard Serre relations.  The image of the one
dimensional submodule is $\fg_2$, which is spanned by
\[
[[e_1, e_2], [e_1, e_3]] - [e_1, [e_1, [e_2, e_3]]]=-[[e_1, e_2], [e_1, e_3]],
\]
where \eqref{D-dim7} is used to obtain the identity.
By using \eqref{D-triangle} and \eqref{D-dim7}, one can easily show that
$
 [\fg_2, \fg_1]=0,
$
and hence $\fg_3=0$.

If $e_1$ is odd,  the Lie superalgebra $L(A, \Theta)$ is
$\mathfrak{osp}_{4|2}$. By dimension counting, we need $\fg_2=0$.
Now $\wedge^2_s\fg_1$ is also a direct sum of a seven dimensional
indecomposable $\fg_0$-submodule and a one dimensional submodule.
Given the condition $[e_1, e_1]=0$, the seven dimensional submodule
vanishes automatically under the Lie superbracket, and the image of
the one dimensional submodule is spanned by $[[e_1, e_2], [e_1,
e_3]]$. Taking the Lie superbraket of $e_1$ with both sides of
\eqref{D-triangle}, we obtain $[[e_1, e_2], [e_1, e_3]]=0$. Hence
$\fg_2=0$.

Now assume $r\ge 4$. We take $d=r-3$, then $\fg_0$ is the direct sum
of a general linear superalgebra and $\mathfrak{osp}_{4|2}$ or
$\mathfrak{osp}_{2|4}$.

If $e_{r-2}$ is even, the condition that $\fg_1$ is an irreducible
$\fg_0$-module of lowest weight $\alpha_{r-3}$ is given by the
relevant relations in \eqref{g-tilde},
\[
[e_{r-2},[e_{r-2}, e_{r-3}]]=0,
\]
and also
\begin{eqnarray}\label{gl-in-B}
\begin{aligned}
&[e_{r-4}, [e_{r-4}, e_{r-3}]]=0, &\quad& \text{if }\  p(r-4)=0, \\
&[e_{r-4}, [e_{r-5}, [e_{r-4}, e_{r-3}]]]=0, &\quad& \text{if }\  p(r-4)=1.
\end{aligned}
\end{eqnarray}

As $\fg_0$-module, $\wedge_s^2\fg_1$ is the direct sum of three
irreducibles. The $\mathfrak{osp}_{2|4}$ subalgebra of $\fg_1$ acts
trivially on one of the irreducible submodules, and $\fg_2$ is
isomorphic to it. The necessary and sufficient condition for the Lie
superbracket to annihilate the other two irreducible submodules is
\begin{eqnarray}\label{CD-2-g2}
\begin{aligned}
&{[}e_{r-3}, [e_{r-3}, e_{r-2}]]=0, \quad [e_{r-3},
[e_{r-3}, e_{r-4}]]=0,&& \text{if } p(r-3)=0,\\
&{[}e_{r-3}, e_{r-3}]=0, \quad [e_{r-3}, [e_{r-4},
[e_{r-3}, e_{r-2}]]]=0,&& \text{if } p(r-3)=1,
\end{aligned}
\end{eqnarray}
as can be shown by examining lowest weight vectors of the
submodules.

\begin{remark} \label{relations-from-gl}
Let $E=[[e_{r-3}, e_{r-2}], e_{r-1}]$ and $E'=[[e_{r-3}, e_{r-2}],
e_r]$.  Then at least one of the vectors $[X, E]$ and $[X, E']$
vanishes for any $X\in\fg_1$.
\end{remark}

Let $v$ denote a lowest weight vector of $\fg_2$. We can take $v=[E,
E']$ if $e_{r-3}$ is even, and $v=[[e_{r-4}, E], E']$ if $e_{r-3}$
is odd. Then by Remark \ref{relations-from-gl}, we have $[v, X]=0$
for any $X\in\fg_1$. Hence $\fg_3=0$.

If $e_{r-2}$ is odd, the condition that $\fg_1$ is an irreducible
$\fg_0$-module of lowest weight $\alpha_{r-3}$ translates into the
relations \eqref{gl-in-B},
\[
\begin{aligned}
&[e_{r-2}, [[e_{r-2},e_{r-1}], e_{r-3}]]=0, \\
&[e_{r-2}, [[e_{r-2},e_{r}], e_{r-3}]]=0,
\end{aligned}
\]
plus the relevant relations in \eqref{g-tilde}.  Here we have used
some facts about generalised Verma modules for
$\mathfrak{osp}_{4|2}$.

As $\fg_0$-module, $\wedge^2_s\fg_1$ is again a direct sum of three
irreducibles. One of them restricts to a direct sum of one
dimensional $\mathfrak{osp}_{4|2}$-modules, and $\fg_2$ is
isomorphic to it. The other two irreducibles are both mapped to zero
by the Lie superbracket. The necessary and sufficient condition for
this to happen is still \eqref{CD-2-g2}.

Note that Remark \ref{relations-from-gl} remains valid in the
present case if we define $E$ and $E'$ in the same way. Let $v=[E,
E']$ if $e_{r-3}$ is odd, and $v=[[e_{r-4}, E], E']$ if $e_{r-3}$ is
even. Then $v$ is a nonzero lowest weight vector of $\fg_2$. It
follows from Remark \ref{relations-from-gl} that $[v, X]=0$ for any
$X\in\fg_1$. Hence $\fg_3=0$.

\subsubsection{Case 3}

Finally we consider the Dynkin diagram
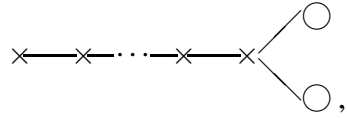
\begin{center}
\begin{picture}(125, 50)(-5,-25)
\put(0, -3){$\times$}
\put(6, 0){\line(1, 0){20}}
\put(24, -3){$\times$}
\put(30, 0){\line(1, 0){10}}
\put(41, 0){\dots}
\put(54, 0){\line(1, 0){10}}
\put(62, -3){$\times$}
\put(68, 0){\line(1, 0){20}}
\put(86, -3){$\times$}
\put(117, 15){\circle{10}}
\put(111, 17){\line(-1, -1){16}}
\put(111, -17){\line(-1, 1){16}}
\put(117, -16){\circle{10}}
\put(125, -20){, }
\end{picture}
\end{center}
assuming that there are at least two grey nodes (as otherwise this
would correspond to the distinguished root system of type $D$). This
forces $r\ge 4$.

This case is quite easy, thus we shall be brief. We choose $d$ to be
the largest integer such that $p(d)=1$. Then $\fg_0$ is the direct
sum of a general linear superalgebra and an even dimensional
orthogonal Lie algebra.

From Section \ref{A-nondistingushed}, we see that the necessary and
sufficient conditions for $e_d$ (which must be odd) to generate an
irreducible $\fg_0$-module are the relevant relations in
\eqref{g-tilde} and the higher order Serre relation involving $e_d$
associated with the following sub-diagram
\begin{picture}(70, 20)(-5,-5)
\put(-1, -3){$\times$}
\put(5, 0){\line(1, 0){20}}
{\color{gray}\put(25, 0){\circle*{10}}}
\put(30, 0){\line(1, 0){20}}
{\color{gray} \put(55, 0){\circle*{10}}}
\put(20, 7){\tiny d-1}
\put(55, 7){\tiny d}
\end{picture}
of the Dynkin diagram if $p(d-1)=1$.  Note that if $d=2$, this
becomes vacuous.

As $\fg_0$-module, $\fg_1$ is the tensor product of the natural
modules $V_A$ and $V_D$ respectively for the general linear
superalgebra and orthogonal algebra contained in $\fg_0$. Here $V_D$
is purely even, and the grading of $V_A$ gives rise to the grading
of $\fg_1$.

Now $\wedge^2_s\fg_1\cong \wedge_s^2(V_A)\otimes
\big(S^2(V_D)/\C\big)\oplus S_s^2(V_A)\otimes \wedge^2(V_D)\oplus
\wedge_s^2(V_A)\otimes\C$ as $\fg_0$-module. The images of the first
two irreducibles under the Lie superbracket are set to zero by the
relation $[e_d, e_d]=0$ and the higher order Serre relation(s)
associated with the sub-diagram(s) of the form
\begin{picture}(75, 20)(-5,-3)
\put(0, -3){$\times$}
\put(5, 0){\line(1, 0){20}}
{\color{gray} \put(30, 0){\circle*{10}}}
\put(35, 0){\line(1, 0){20}}
\put(60, 0){\circle{10}}
\put(0, 7){\tiny d-1}
\put(24, 7){\tiny d}
\put(65, -3){.}
\end{picture}
Note that if $d<r-2$, there is only one such diagram, but there are
two if $d=r-2$, as the last node can be $(r-1)$ or $r$. We have
$\fg_2\cong \wedge_s^2(V_A)\otimes\C$.

One can show that $[\fg_2, \fg_1]=0$ by using the same arguments as
those in Section \ref{Dm-distinguished} and Section
\ref{D2-distinguished}, thus $\fg_3=0$.

\subsection{Proof in type $F(4)$}\label{F4-nondistingushed}
Now we turn to $F(4)$, which is considerably more complicated than
the other type of Lie superalgebras.
\subsubsection{Case 1}
Consider first the root system corresponding to the Dynkin diagram

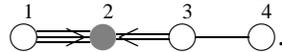
\begin{center}
\begin{picture}(100, 15)(-5,-5)
\put(0, 0){\circle{10}}
\put(5, 2){\line(1, 0){20}}
\put(5, 0){\line(1, 0){20}}
\put(5, -2){\line(1, 0){20}}
\put(15, -3.25){$>$}
{\color{gray}\put(30, 0){\circle*{10}}}
\put(35, 1){\line(1, 0){20}}
\put(35, -1){\line(1, 0){20}}
\put(35, -3.25){$<$}
\put(60, 0){\circle{10}}
\put(65, 0){\line(1, 0){20}}
\put(90, 0){\circle{10}}

\put(0, 7){\tiny 1}
\put(30,7){\tiny 2}
\put(60, 7){\tiny 3}
\put(90, 7){\tiny 4}

\put(96, -3){.}
\end{picture}
\end{center}
We take $d=2$. Then $\fg_0=\mathfrak{sl}_2\oplus \mathfrak{gl}_3$.
The standard Serre relations plus the relevant relations in
\eqref{g-tilde} are the necessary and sufficient conditions
rendering the $\fg_0$-module $\fg_1$ irreducible. We have
$\fg_1\cong \C^2\otimes \C^3$ up to a parity change.

As $\fg_0$-module, $\wedge^2_s \fg_1$ is a direct sum of two
irreducibles. The condition $[e_2, e_2]=0$ forces one of the
irreducibles to be in the kernel of the map $\wedge^2_s
\fg_1\longrightarrow\fg_2$. Thus $\fg_2$ is an irreducible
$\fg_0$-module generated by the lowest weight vector $E=[[e_1, e_2],
[e_2, e_3]]$. We have $\fg_2=\C\otimes \wedge^2(\C^3)$.

Now $\fg_3=[\fg_2, \fg_1]\cong \C^2\otimes\C$  with a basis
consisting of vectors $[E, [e_2, [e_3, e_4]]  ]$ and $[E, [E',
e_4]]$, where $E'=[e_1, [e_2, e_3]]$. One immediately sees that
\[
[\fg_3, e_2]=\C[E, [E, e_4]],
\]
which generates $\fg_4=\C\otimes\C^3$.

To consider $\fg_5$, we only need to look at $[\fg_4, \fg_1]$. If
$X\in \fg_1$ is any lowest weight vector for $\mathfrak{sl}_2\subset
\fg_0$, the higher order Serre relation associated with the Dynkin
diagram (see diagram (\ref{HOS-g-7}) in Theorem \ref{Serre-g})
renders $[\fg_4, X]=0$. Since the $\mathfrak{sl}_2$ subalgebra of
$\fg_0$ acts trivially on $\fg_4$, it follows that $[\fg_4,
\fg_1]=0$, that is, $\fg_5=0$.

\subsubsection{Case 2}

For the Dynkin diagram
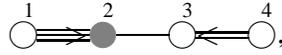
\begin{center}
\begin{picture}(100, 20)(-5,-5)
\put(0, 0){\circle{10}}
\put(5, 2){\line(1, 0){20}}
\put(5, 0){\line(1, 0){20}}
\put(5, -2){\line(1, 0){20}}
\put(15, -3.25){$>$}
{\color{gray}\put(30, 0){\circle*{10}}}
\put(35, 0){\line(1, 0){20}}
\put(60, 0){\circle{10}}
\put(65, 1){\line(1, 0){20}}
\put(65, -1){\line(1, 0){20}}
\put(65, -3.25){$<$}
\put(90, 0){\circle{10}}

\put(0, 7){\tiny 1}
\put(30,7){\tiny 2}
\put(60, 7){\tiny 3}
\put(90, 7){\tiny 4}

\put(96, -3){,}
\end{picture}
\end{center}
we also take $d=2$ as in the previous case. Then
$\fg_0=\mathfrak{gl}_2\oplus\mathfrak{sp}_4$. The relevant relations
in \eqref{g-tilde} and standard Serre relations guarantee that $e_2$
generates an irreducible $\fg_0$-module, which is isomorphic to the
tensor product $\C^2\otimes\C^4$ of the natural modules for
$\mathfrak{gl}_2$ and $\mathfrak{sp}_4$ up to a parity change.

Now $\wedge^2_s\fg_1$ decomposes into the direct sum of three
irreducible $\fg_0$-modules, which are respectively isomorphic to
$S^2(\C^2)\otimes S^2(\C^4)$, $\wedge^2(\C^2)\otimes
\left(\wedge^2(\C^2)/\C\right)$ and $\wedge^2(\C^2)\otimes\C$. The
necessary and sufficient conditions for the Lie superbracket to map
the first and the third submodules to zero are $[e_2, e_2]=0$ and
the higher order Serre relation
\begin{eqnarray}\label{hos-F4-2}
[[e_1, e_2], [[e_2, e_3], [e_3, e_4]]
-[[e_2, e_3], [[e_1, e_2], [e_3, e_4]]=0
\end{eqnarray}
associated with the Dynkin diagram (see diagram (\ref{HOS-g-8}) in
Theorem \ref{Serre-g}).  Now $\fg_2$ is isomorphic to
$\wedge^2(\C^2)\otimes \frac{\wedge^2(\C^2)}{\C}$ with lowest weight
vector
\[
E=[[e_1, e_2], [e_2, e_3]].
\]
Formally $[\fg_2, \fg_1]$  decomposes into the direct sum of two
irreducibles, respectively having lowest weight vectors
\[
[E, e_2], \quad [e_2, [E, [e_3, e_4]]].
\]
The first vector vanishes by $[e_2, e_2]=0$. The second vector is
the supercommutator of $e_2$ with the left hand side of
\eqref{hos-F4-2}, thus is also zero. This shows that $\fg_3=0$.

\subsubsection{Case 3}

Consider the Dynkin diagram
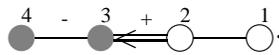
\begin{center}
\begin{picture}(125, 20)(-5,0)
{\color{gray}\put(0, 0){\circle*{10}}}
\put(5, 0){\line(1, 0){20}}
{\color{gray}\put(30, 0){\circle*{10}}}
\put(35, 1){\line(1, 0){20}}
\put(35, -1){\line(1, 0){20}}
\put(35, -3.25){$<$}
\put(60, 0){\circle{10}}
\put(65, 0){\line(1, 0){20}}
\put(90, 0){\circle{10}}
\put(96, 0){.}
\put(0, 6){\tiny 4}
\put(30, 6){\tiny 3}
\put(60, 6){\tiny 2}
\put(90, 6){\tiny 1}

\put(15, 5){\tiny -}
\put(45, 5){\tiny +}
\end{picture}
\end{center}
We take $d=4$, and delete the 4-th node from the diagram to obtain
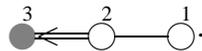
\begin{center}
\begin{picture}(100, 15)(30,-5)
{\color{gray}\put(30, 0){\circle*{10}}}
\put(35, 1){\line(1, 0){20}}
\put(35, -1){\line(1, 0){20}}
\put(35, -3.25){$<$}
\put(60, 0){\circle{10}}
\put(65, 0){\line(1, 0){20}}
\put(90, 0){\circle{10}}
\put(96, 0){.}
\put(30, 6){\tiny 3}
\put(60, 6){\tiny 2}
\put(90, 6){\tiny 1}
\end{picture}
\end{center}
This is a non-standard diagram for $\mathfrak{sl}_{1|3}$, where the
double edges can be got rid of by a normalisation of the bilinear
form on the weight space thus are immaterial. The presentation for
$\mathfrak{sl}_{1|3}$ involves no higher order Serre relation. We
have $\fg_0=\mathfrak{gl}_{1|3}$.

Let $\overline\fp$ be the lower triangular maximal parabolic
subalgebra of $\fg_0$ with Levi subalgebra
$\fl:=\mathfrak{gl}_{3}\oplus \mathfrak{gl}_1$. Let
$\overline{L}^0_{\alpha_4}=\C v_0$ be the $1$-dimensional
$\overline\fp$-module with lowest weight $\alpha_4$, which is assume
to be a purely odd superspace. Since $\alpha_4$ is a typical $\fg_0$
weight, the generalised Verma module $\overline{V}_{\alpha_4} =
U(\fg_0)\otimes_{U(\overline\fp)}\overline{L}^0_{\alpha_4}$ is
irreducible, i.e.,
$\overline{L}_{\alpha_4}=\overline{V}_{\alpha_4}$. It is
multiplicity free, and the set of weights is given by
\begin{eqnarray}\label{weights-g1}
\Delta^+\backslash \left\{\Delta^+(\fg_0)\cup \Delta^+_2\right\},
\end{eqnarray}
where $\Delta^+$ is the set of the positive roots of $F(4)$ relative
to the Borel subalgebra under consideration, $\Delta^+(\fg_0)$ is
the set of the positive roots of the subalgebra $\fg_0$, and
\begin{eqnarray} \label{weights-g2}
\Delta^+_2=\left\{\frac{1}{2}(\delta+\epsilon_1+\epsilon_2+\epsilon_3),
\ \epsilon_i + \epsilon_j, i\ne j\right\}.
\end{eqnarray}

The $\fg_0$-module $\wedge^2_s \overline{L}_{\alpha_4}$ is not
semi-simple. To avoid the laborious task of determining the
indecomposable submodules, we simply examine the $\fl$ lowest weight
vectors in $\wedge^2_s \overline{L}_{\alpha_4}$. Of particular
importance to us are the vectors
\[
\begin{aligned}
z_1:=  &v_0\otimes v_0; \\
z_2:=  &e_3 v_0\otimes [e_2, e_3] e_3 v_0 -
        [e_2, e_3] e_3 v_0\otimes e_3 v_0;\\
z_3:= & v_0 \otimes e_3 v_0 - e_3 v_0 \otimes v_0;\\
w_1:=  & v_0 \otimes [e_2, e_3] e_3 v_0 + [e_2, e_3] e_3 v_0 \otimes v_0; \\
w_2:=  &v_0\otimes [e_1, [e_2, e_3]] [e_2, e_3] e_3 v_0 -
        [e_1, [e_2, e_3]] [e_2, e_3] e_3 v_0\otimes v_0.
\end{aligned}
\]

The space of $\fl$-lowest weight vectors of $\wedge^2_s
\overline{L}_{\alpha_4}$ is spanned by $w_1$, $w_2$ and the
$\fl$-lowest weight vectors in the $\fg_0$-submodule $M$ generated
by $z_1$ and $z_2$. It is important to observe that $w_1$ and $w_2$
are not in $M$, but $w_1\in U(\fg_0)w_2$. Furthermore, one can
verify that $\wedge^2_s \overline{L}_{\alpha_4}/M$ is multiplicity
free with the set of weights $\Delta_2^+$.

Now we take $v_0=e_4$ and require $\overline{\fp}$ act on it by the
adjoint action. Then $\fg_1=\overline{L}_{\alpha_4}$. We require
that the Lie superbracket maps $z_1$ and $z_2$ to zero. This leads
to the following relations:
\begin{eqnarray}\label{F4-1}
\begin{aligned}
&[e_4, e_4]=0; \\
&{[} [e_3, e_4], [[e_3, e_4], [e_2, e_3]]  ]=0.
\end{aligned}
\end{eqnarray}
Under the first condition, the Lie superbracket automatically maps
$z_3$ to zero.
Note that the second relation in equation \eqref{F4-1} is the
desired higher order Serre relation associated with the sub-diagram
\begin{center}
\begin{picture}(125, 15)(-10,-5)
{\color{gray}\put(0, 0){\circle*{10}}}
\put(5, 0){\line(1, 0){20}}
{\color{gray}\put(30, 0){\circle*{10}}}
\put(35, 1){\line(1, 0){20}}
\put(35, -1){\line(1, 0){20}}
\put(35, -3.25){$<$}
\put(60, 0){\circle{10}}
\put(68, -5){.}
\end{picture}
\end{center}

The vectors $w_1$ and $w_2$ have non-zero images under the Lie
superbracket, and we have $\fg_2\cong \wedge^2_s
\overline{L}_{\alpha_4}/M$. By considering the possible $\fl$-lowest
weight vectors, we can show that $[\fg_1, \fg_2]=0$, thus $\fg_3=0$.

Now the proof of Lemma \ref{key} in this case is completed by
comparing the weights in \eqref{weights-g1} and \eqref{weights-g2}
with the roots in $L_1$ and $L_2$.

\subsubsection{Case 4}
Consider the Dynkin diagram
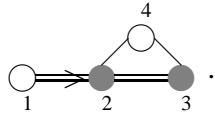
\begin{center}
\begin{picture}(100, 40)(30,-10)
\put(30, 0){\circle{10}}
\put(35, 1){\line(1, 0){20}}
\put(35, -1){\line(1, 0){20}}
\put(45, -3.5){$>$}
{\color{gray}\put(60, 0){\circle*{10}}}
\put(60, 5){\line(1, 1){10}}
\put(65, 1){\line(1, 0){20}}
\put(65, -1){\line(1, 0){20}}
{\color{gray}\put(90, 0){\circle*{10}}}
\put(90, 5){\line(-1, 1){10}}
\put(75, 15){\circle{10}}

\put(30, -12){\tiny 1}
\put(60, -12){\tiny 2}
\put(90, -12){\tiny 3}
\put(75, 23){\tiny 4}

\put(100, 0){.}
\end{picture}
\end{center}
Take $d=1$, then $\fg_0=\mathfrak{osp}_{2|4}\oplus\mathfrak{gl}_1$.
The presentation of $\mathfrak{osp}_{2|4}$ relative to the Dynkin
diagram
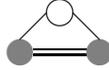
\begin{center}
\begin{picture}(55, 30)(50,-5)
{\color{gray}\put(60, 0){\circle*{10}}}
\put(60, 5){\line(1, 1){10}}
\put(65, 1){\line(1, 0){20}}
\put(65, -1){\line(1, 0){20}}
{\color{gray}\put(90, 0){\circle*{10}}}
\put(90, 5){\line(-1, 1){10}}
\put(75, 15){\circle{10}}
\end{picture}
\end{center}
has been constructed, thus the defining relations among $e_i, f_i,
h_i$ for $i> 1$ are all known. The parabolic subalgebra of
$\mathfrak{osp}_{2|4}$ defined in Appendix \ref{osp24-module}
together with the ideal $\mathfrak{gl}_1$ form a parabolic of
$\fg_0$. Then $e_1$ spans a $1$-dimensional module for this
parabolic, which induces a generalised Verma module
$\overline{V}_{\alpha_1}$ of lowest weight type for $\fg_0$. The
structure of $\overline{V}_{\alpha_1}$ can be understood by using
results of Section \ref{osp24-module}. In particular, imposing the
condition \eqref{generator-M-C}, which in the present case reads
\begin{eqnarray}\label{generator-M-C-1}
[e_2, [[e_2, e_3], e_1]]=0,
\end{eqnarray}
sends $\overline{V}_{\alpha_1}$ to the irreducible quotient, which
is $\fg_1$. Note that \eqref{generator-M-C-1} is a higher order
Serre relation associated with
diagram (\ref{HOS-g-9}) in Theorem \ref{Serre-g}. It is a
non-standard diagram of $\mathfrak{sl}_{1|3}$.

Now $\fg_1$ forms is $10$-dimensional. A basis for it can be deduced
from Section \ref{osp24-module}. For every vector $b$ in this basis,
we have $[b, e_1]=0$. This holds trivially for most basis vectors,
but for $b=[e_2, [[[e_1, e_2], e_3], e_4]]$, we have
\[
\begin{aligned}
{[}e_1, b] &= [ [[e_1, e_2], e_3], [[e_1, e_2], e_4]]\\
            &=\frac{1}{2} \left(ad_{e_1}\right)^2[ [e_2, e_3], [e_2, e_4] ].
\end{aligned}
\]
One can deduce from the defining relations for
$\mathfrak{osp}_{2|4}$ that $[ [e_2, e_3], [e_2, e_4] ]$  $=0$,
hence $[b, e_1]=0$. This implies that the commutator of $e_1$ with
all the remaining basis vectors are all zero. Therefore,
$\fg_2=[\fg_1, \fg_1]=0$.

\subsubsection{Case 5}
In the case of the Dynkin diagram
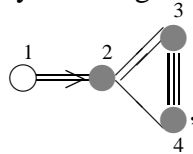
\begin{center}
\begin{picture}(60, 50)(50,-25)
\put(45, 0){\circle{10}}
\put(50, 1){\line(1, 0){20}}
\put(50, -1){\line(1, 0){20}}
\put(60, -3.5){$>$}
{\color{gray}\put(75, 0){\circle*{10}}}
\put(98, 19){\line(-1, -1){18}}
\put(98, 15){\line(-1, -1){16}}
\put(98, -19){\line(-1, 1){18}}
{\color{gray}\put(102, 15){\circle*{10}}}
\put(100,-10){\line(0, 1){19}}
\put(102,-10){\line(0, 1){19}}
\put(104,-10){\line(0, 1){19}}
{\color{gray}\put(102, -16){\circle*{10}}}

\put(45, 7){\tiny 1}
\put(75, 7){\tiny 2}
\put(102, 23){\tiny 3}
\put(102, -28){\tiny 4}

\put(108, -16){,}
\end{picture}
\end{center}
we take $d=4$ and delete the $4$-th node to obtain the diagram
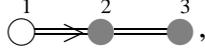
\begin{center}
\begin{picture}(75, 20)(40,-3)
\put(45, 0){\circle{10}}
\put(50, 1){\line(1, 0){20}}
\put(50, -1){\line(1, 0){20}}
\put(60, -3.5){$>$}
{\color{gray}\put(75, 0){\circle*{10}}}
\put(80, 1){\line(1, 0){20}}
\put(80, -1){\line(1, 0){20}}
{\color{gray}\put(105, 0){\circle*{10}}}

\put(45, 7){\tiny 1}
\put(75, 7){\tiny 2}
\put(105,7){\tiny 3}

\put(112, -2){,}
\end{picture}
\end{center}
which is a non-standard diagram of $\mathfrak{sl}_{1|3}$. Thus we
have a relation formally the same as \eqref{generator-M-C-1}.

Now $\fg_0=\mathfrak{gl}_{1|3}$. The Verma module of lowest weight
type for $\fg_0$ generated by $e_4$ contains the primitive vector
$2[e_2, [e_3, e_4]] - 3[[e_2, e_3], e_4]$, which generates the
maximal submodule. Thus the higher order Serre relation
\begin{eqnarray}\label{relation-HOS-g-10}
2[e_2, [e_3, e_4]] - 3[[e_3, e_2], e_4]=0,
\end{eqnarray}
associated with diagram (\ref{HOS-g-10}) in Theorem \ref{Serre-g},
is all that is needed to guarantee that $\fg_1$ is an irreducible
$\fg_0$-module. This module is typical relative to the distinguished
Borel subalgebra, and has dimension $8$.

Restricted to a module for $\mathfrak{gl}_3\subset \fg_0$,  the even subspace of $\fg_1$
is the direct sum of the natural $\mathfrak{gl}_3$-module and a $1$-dimensional module,
while the odd subspace is the direct sum of the dual natural module
(twisted by a scalar) and a $1$-dimensional module.

Now consider $[\fg_1, \fg_1]$.  We can easily work out its decomposition into
irreducible $\mathfrak{gl}_3$-submodules.  The corresponding  $\mathfrak{gl}_3$
lowest weight vectors can be worked out, which include the following vectors:
\[
[e_2, e_2], \quad \left(ad_{[e_2, e_4]}\right)^2 [e_2, e_1], \quad [[e_2, e_4], [e_3, e_4]].
\]
It follows from the higher order Serre relation \eqref{relation-HOS-g-10}
that
\[
[[e_2, e_4], [e_3, e_4]]=0.
\]
Now we impose the relations
\[
[e_2, e_2]=0, \quad \left(ad_{[e_2, e_4]}\right)^2 [e_2, e_1]=0,
\]
where the first is a standard Serre relation, and the second is a
higher order Serre relations associated with
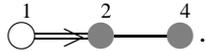
\begin{center}
\begin{picture}(75, 20)(40,-3)
\put(45, 0){\circle{10}}
\put(50, 1){\line(1, 0){20}}
\put(50, -1){\line(1, 0){20}}
\put(60, -3.5){$>$}
{\color{gray}\put(75, 0){\circle*{10}}}
\put(80, 0){\line(1, 0){20}}
{\color{gray}\put(105, 0){\circle*{10}}}

\put(45, 7){\tiny 1}
\put(75, 7){\tiny 2}
\put(105,7){\tiny 4}

\put(112, -2){.}
\end{picture}
\end{center}
Under these conditions, all other $\mathfrak{gl}_3$ lowest vectors
in $[\fg_1, \fg_1]$ vanish, except
\[
\left(ad_{[e_2, e_4]}\right)^2 e_1, \quad [[e_3, e_4], [[e_1, e_2], [e_2, e_4]]],
\]
where the first one is actually a $\fg_0$ lowest weight vector. It generates an $4$-dimensional
irreducible $\fg_0$-module containing the second vector. This module is isomorphic to the dual of the
natural $\fg_0$-module twisted by a scalar. This gives us $\fg_2=[\fg_1, \fg_1]$.
We can further show that $[\fg_2, e_4]=0$, hence $\fg_3=0$.

\subsection{Proof in type $G(3)$}\label{G3-nondistingushed}
\subsubsection{Case 1}\label{G3-case-1}

Consider the Dynkin diagram
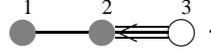
\begin{center}
\begin{picture}(100, 20)(25,-5)
{\color{gray}\put(30, 0){\circle*{10}}}
\put(35, 0){\line(1, 0){20}}
{\color{gray}\put(60, 0){\circle*{10}}}
\put(65, 2){\line(1, 0){20}}
\put(65, 0){\line(1, 0){20}}
\put(65, -2){\line(1, 0){20}}
\put(65, -3.25){$<$}
\put(90, 0){\circle{10}}

\put(30, 7){\tiny 1}
\put(60, 7){\tiny 2}
\put(90, 7){\tiny 3}

\put(100, 0){.}
\end{picture}
\end{center}
We take $d=3$, then $\fg_0=\mathfrak{gl}_{2|1}$. Let
$\overline{V}_{\alpha_3}$ be the lowest weight Verma module for
$\fg_0=\mathfrak{gl}_{2|1}$ with lowest weight $\alpha_3$. Denote by
$v_0$ the lowest weight vector, which is assumed to be even. Then
the maximal submodule of $\overline{V}_{\alpha_3}$ is generated by
$e_1 v_0$ and $e_2[e_1, e_2]^3 v_0$. The irreducible quotient
$\overline{L}_{\alpha_3}$ is multiplicity free and has weights
\[
\begin{aligned}
&\alpha_3+ k(\alpha_1+\alpha_2), && k=0, 1, 2, 3, \\
&\alpha_3+ p(\alpha_1+\alpha_2)+\alpha_2, && p=0, 1, 2.
\end{aligned}
\]

In fact $\overline{L}_{\alpha_3}$ is isomorphic to the third
$\Z_2$-graded symmetric power of the natural module for $\fg_0$
tensored with a $1$-dimensional module. Thus $\wedge^2_s
\overline{L}_{\alpha_3}$ is completely reducible; it is the direct
sum of two irreducibles.

Now we take $v_0$ to be $e_3$, and let $\fg_0$ act on it by the
adjoint action.  Then the generators of the maximal submodule of
$\overline{V}_{\alpha_3}$ in this case are $\left(ad_{[e_1,
e_2]}\right)^3 [e_2, e_3]$ and $[e_1, e_3]$. Thus $[e_1, e_3]=0$ and
the higher order Serre relation
\begin{eqnarray}\label{G3-1}
\left(ad_{[e_1, e_2]}\right)^3 [e_2, e_3]=0
\end{eqnarray}
(associated with diagram (\ref{HOS-g-11}) in Theorem \ref{Serre-g})
render $\fg_1=\overline{L}_{\alpha_3}$.

One of the irreducible submodules of $\wedge^2_s \fg_1$ has a
lowest weight vector of the form $e_3 \otimes [e_2, e_3] -   [e_2, e_3]\otimes
e_2$. We require that this submodule be in the kernel of the Lie
superbracket. This leads to the standard Serre relation $[e_3, [e_3,
e_2]]=0$.

The other irreducible submodule of $\wedge^2_s \fg_1$ is mapped
surjectively onto $\fg_2$. A lowest weight vector of $\fg_2$ is given by
$\overline{X}:=ad_{e_3}\left(ad_{[e_1, e_2]}\right)^2 [e_2, e_3]$.
This irreducible module is $4$-dimensional and has weights
\[
-(2\epsilon_3-\epsilon_1-\epsilon_2),
\ \delta+\epsilon_2-\epsilon_3,
\ \delta+\epsilon_1-\epsilon_3,
\ 2\delta,
\]
in the notation explained in Appendix \ref{roots}.
It is easy to see that $[X, \overline{X}]=0$ for all $X\in\fg_1$.
Thus $\fg_3=0$.

By examining the weights of $\fg_1$ and $\fg_2$, we see that Lemma
\ref{key} holds.

\subsubsection{Case 2}

Consider the Dynkin diagram
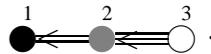
\begin{center}
\begin{picture}(100, 20)(25,-5)
\put(30, 0){\circle*{10}}
\put(35, 1){\line(1, 0){20}}
\put(35, -1){\line(1, 0){20}}
\put(35, -3.25){$<$}
{\color{gray}\put(60, 0){\circle*{10}}}
\put(65, 2){\line(1, 0){20}}
\put(65, 0){\line(1, 0){20}}
\put(65, -2){\line(1, 0){20}}
\put(65, -3.25){$<$}
\put(90, 0){\circle{10}}

\put(30, 7){\tiny 1}
\put(60, 7){\tiny 2}
\put(90, 7){\tiny 3}

\put(100, 0){.}
\end{picture}
\end{center}
We take $d=1$, and delete the first node from the Dynkin diagram to
obtain
\begin{center}
\begin{picture}(50, 15)(55,-5)
{\color{gray}\put(60, 0){\circle*{10}}}
\put(65, 2){\line(1, 0){20}}
\put(65, 0){\line(1, 0){20}}
\put(65, -2){\line(1, 0){20}}
\put(65, -3.25){$<$}
\put(90, 0){\circle{10}}
\put(98, -3){.}
\end{picture}
\end{center}
This is a nonstandard diagram for $\mathfrak{sl}_{1|2}$, which can
be cast into the usual Dynkin diagram
of $\mathfrak{sl}_{1|2}$ in the distinguished root system by normalising the bilinear form on the
weight space. Note that no higher order Serre relations are required
to present this Lie superalgebra. We have
$\fg_0=\mathfrak{gl}_{1|2}$.

Now the $\fg_0$ Kac module of lowest weight type generated by $e_1$
is typical thus irreducible, hence $\fg_1\cong
\overline{L}_{\alpha_1}$ with basis
\[
e_1, \quad [e_2, e_1], \quad  [[e_2, e_3], e_1], \quad [[e_2, e_3], [e_2, e_1]].
\]
As $\fg_0$-module $\wedge^2_s \fg_1$ is the direct sum of two
irreducible typical submodules, respectively generated by the lowest
weight vectors
$e_1\otimes e_1$ and $v - \frac{1}{2}v'$,
where
\[
\begin{aligned}
v= e_1\otimes [[e_2, e_3], [e_2, e_1]] + [[e_2, e_3], [e_2, e_1]]\otimes e_1, \\
v'=[e_2, e_1]\otimes [e_3, [e_2, e_1]]-  [ e_3, [e_2, e_1]]\otimes [e_2, e_1].
\end{aligned}
\]
We require that $v - \frac{1}{2}v'$ and thus the $\fg_0$-submodule
generated by it be mapped to zero by the Lie superbracket. This
leads to
\[
[[e_2, e_1], [e_3, [e_2, e_1]] - [[e_2, e_3], [[e_1, e_1], e_2]]=0,
\]
which is one of the higher order Serre relations associated with the
Dynkin diagram (see diagram (\ref{HOS-g-12}) in Theorem
\ref{Serre-g}). Therefore, $\fg_2\cong \overline{L}_{2\alpha_1}$ and
has a basis
\[
[e_1, e_1], \quad [[e_1, e_1], e_2],
\quad  [e_3, [[e_1, e_1], e_2]],
\quad [ [e_2, e_3], [[e_1, e_1], e_2] ].
\]

Now we consider $[\fg_2, \fg_1]$. One can easily see that
$\left(ad_{e_1}\right)^3 e_2$ is a $\fg_0$ lowest weight vector. We
require that the $\fg_0$-submodule generated by it be zero, hence we
have the standard Serre relation
\[
\left(ad_{e_1}\right)^3 e_2=0.
\]
This leaves $\fg_3=[\fg_2, \fg_1]$ to be an indecomposable
$\fg_0$-module cyclically generated by the lowest weight vector $[
[e_1, [e_2, e_3]], [[e_1, e_1], e_2] ]$, which is $7$-dimensional
and multiplicity free. One can easily write down a basis for this
module. We should remark that no $\fg_0$ lowest weight vector in
$\fg_3$ is annihilated by all $f_i$ for $i=1, 2, 3$.

One can show by direct computations that $[e_1, \fg_3]=0$ and
$[[e_1, e_1], \fg_2]=0$. Hence $\fg_4=0$.

An inspection of the weight spaces of $\fg_i$ for $1\le i \le 3$
shows that they agree with those of $L_i$ for $1\le i \le 3$. This
completes the proof in this case.

\subsubsection{Case 3}

The final case of $G(3)$ is the diagram
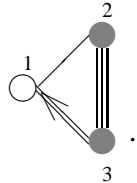
\begin{center}
\begin{picture}(90, 70)(70, -35)
\put(75, 0){\circle{10}}
\put(80, 1){\line(1, -1){20}}
\put(80, -1){\line(1, -1){20}}
\put(82, -2){\line(2, -1){10}}
\put(82, -2){\line(1, -2){5}}
{\color{gray}\put(105,-20){\circle*{10}}}

\put(103, 15){\line(0, -1){30}}
\put(105, 15){\line(0, -1){30}}
\put(107, 15){\line(0, -1){30}}

{\color{gray}\put(105, 20){\circle*{10}}}

\put(100, 20){\line(-1, -1){19}}

\put(75, 7){\tiny 1}
\put(105, 27){\tiny 2}
\put(105,-35){\tiny 3}

\put(115, -20){.}
\end{picture}
\end{center}
We take $d=3$, then $\fg_0=\mathfrak{gl}_{2|1}$.  The $\fg_0$ Kac
module of lowest weight generated by $e_3$ is atypical. We set the
primitive vector to zero to obtain
\[
2[[e_1, e_2], e_3]-[e_2, [e_1, e_3]]=0,
\]
which is a higher order Serre relation in the present case. Then $\fg_1$ is
an irreducible $\fg_0$-module with lowest weight $\alpha_3$, which is
isomorphic to the third $\Z_2$-graded symmetric power of the natural module
for $\fg_0$ twisted by a scalar. It has
$3$ odd and $4$ even dimensions. A basis for $\fg_1$ is given by
\[
\begin{aligned}
&e_3, \quad [e_1, e_3], \quad  [e_1, [e_1, e_3]], \quad [e_2, e_3],\\
&[[e_1, e_2], e_3],\quad  [[e_1, e_2], [e_1, e_3]],
\quad [[e_1, e_2], [e_1, [e_1, e_3]]].
\end{aligned}
\]

The rest of the analysis is similar to Section \ref{G3-case-1}. Now
$\wedge^2_s\fg_1$ is the direct sum of two irreducible
$\fg_0$-submodules. The images of theirs lowest weight vectors in
$[\fg_1, \fg_1]$ are repectively $[e_3, e_3]$ and $E=[ [e_1, e_3],
[e_1, e_3] ]$. Both generate typical $\fg_0$-submodules, which
respectively have dimensions $20$ and $4$. The standard Serre
relation $[e_3, e_3]=0$ removes the $20$-dimensional submodule, thus
$\fg_2$ is the $4$-dimensional irreducible $\fg_0$-module generated
by $E$.

We can also show that $\fg_3=0$ without imposing further relations.
Inspecting the weights of $\fg_1$ and $\fg_2$, we see that
the claim of Lemma \ref{key} indeed holds.

\subsection{Proof in type $D(2, 1; \alpha)$}\label{D21alpha-nondistingushed}

The Dynkin diagrams having only one grey node can be treated in
exactly the same way as for the distinguished root system, thus we
shall consider only the diagram with three gray nodes here. Set
$d=3$, then $\fg_0=\mathfrak{gl}_{2|1}$. The $\fg_0$ Verma module of
lowest weight type generated by $e_3$ contains the primitive vector
\[
\alpha [e_1, [e_2, e_3]] + (1+\alpha)[e_2, [e_1, e_3]],
\]
which in fact generates the maximal submodule. The higher order
Serre relation requires this vector to be zero. This is equivalent
to taking the irreducible quotient of the Verma module, and we
obtain $\fg_1$. A basis for $\fg_1$ is
\[
e_3, \quad [e_1, e_3], \quad [e_2, e_3], \quad [e_1, [e_2, e_3]].
\]
An easy computation using the higher order Serre relation shows that
$[\fg_1, e_3]=0$. Hence $\fg_2=0$. A quick inspection on the weights
of $\fg_1$ shows that Lemma \ref{key} indeed holds in this case.

\section{Remarks on affine Lie superalgebras}\label{affine}

We wish to mention that the generalisation of the method to affine
Lie superalgebras is in principle straightforward conceptually.
Consider, for example, the untwisted affine superalgebra $\hat{\fg}$
of a contragredient Lie superalgebra $\fg$. We want to present
$\hat{\fg}$ with the standard generators $e_i, f_i, h_i$ with $0\le
i\le r$ and relations. Here the generators $e_i, f_i, h_i$ with
$1\le i\le r$ are those for $\fg$. By results of earlier sections,
we may assume that all the Serre relations and higher order ones
obeyed by $e_i$ and $f_i$ with $1\le i\le r$ are given.

We introduce the standard $\Z$-grading of $\hat{\fg}$  by
decreeing that all $h_j$ and $e_i, f_i$ with $1\le i\le r$ have
degree $0$, but $e_0$ and $f_0$ have degrees $1$ and $-1$
respectively. Then $\hat{\fg}=\oplus_{k\in\Z} \hat{\fg}_k$,
with $\hat{\fg}_0=\fg\oplus\mathfrak{gl}_1$. Now we
require that as $\hat{\fg}_0$-modules, all $\hat{\fg}_k$ are
isomorphic to $\fg$. The (necessary and sufficient) conditions
meeting this requirement give rise to the defining relations of
$\hat{\fg}$.

To illustrate how this may work, we consider the untwisted affine
algebra $\hat{\fg}=\hat{\mathfrak{sl}}_{r+1}$. The relations
\[
[e_1, [e_1, e_0]]=0, \quad [e_r, [e_r, e_0]]=0, \quad [e_i, e_0]=0, \ i\ne 1, r
\]
arise from the requirement that $\hat{\fg}_1$ be an irreducible
$\hat{\fg}_0$-module. In $[\hat{\fg}_1, \hat{\fg}_1]$, there are
$\hat{\fg}_0$ lowest weight vectors $[[e_1, e_0], e_0]$ and $[[e_r,
e_0], e_0]$, which have weights different from any roots of
$\fg={\mathfrak{sl}}_{r+1}$. Thus the condition that $\hat{\fg}_2$
is isomorphic to $\fg$ as $\fg_0$-module requires
\[
[[e_1, e_0], e_0]=0, \quad  [[e_r, e_0], e_0]=0.
\]
Now we have derived at all the Serre relations needed for $e_0$, and
those for $f_0$ can be similarly obtained. Together with relations
defining $\fg$, these relations define $\hat{\fg}$.

We hope to treat the affine superalgebras on another occasion.

\begin{appendix}
\section{Dynkin diagrams}\label{Dynkin-diagrams}

We describe the Dynkin diagrams for both the distinguished and
non-distinguished root systems in this Appendix. The
roots of all the simple contragedient Lie superalgebras will also be listed
\cite{Kac1, Kac2}.

\subsection{Roots}\label{roots}
Let $\epsilon_i$ ($i=1, 2, \dots, k$) and
$\delta_j$ ($j=1, 2, \dots, l$) be a basis of a real vector space
$E(k, l)$ equipped with a non-degenerate symmetric bilinear form.
Then for each simple contragredient Lie superalgebra $\fg$, the dual space
$\fh^*$ of the cartan subalgebra is either $\C\otimes_{\R} E(k, l)$ for
appropriate $k, l$ or a subspace thereof, which inherits
a non-degenerate bilinear form that is  Weyl group invariant.

For the series $A$, $B$, $C$ or $D$, the bilinear form is defined by
\[
(\epsilon_i, \epsilon_{i'})=\delta_{i i'}, \quad (\delta_j,
\delta_{j'})=-\delta_{j j'}, \quad (\epsilon_i, \delta_j)=0,
 \quad \forall i, i', j, j'.
\]

The roots of the simple contragredient Lie superalgebras can be
described as follows.

\medskip
\noindent $A(m|n)$:
\[
\begin{aligned}
    \Delta_0 &=\{\epsilon_i-\epsilon_{i'} \mid i, i'\in[1, m+1],  i\ne i'\}
    \cup\{\delta_j-\delta_{j'} \mid j, j'\in[1, n+1],   j\ne j'\}, \\
    \Delta_1 &=\{\pm(\epsilon_i-\delta_j)\mid i\in[1, m+1], j\in[1, n+1]\},\\
    &\text{where $[1, N]$ denotes $\{1, \dots, N\}$ for any positive integer $N$.}
\end{aligned}
\]

\noindent $B(0, n)$:
\[
    \begin{aligned}
    \Delta_0 &=\{\pm\delta_j\pm\delta_{j'}, \ \pm2\delta_j \mid j, j'\in[1, n], j\ne j'\}, \\
    \Delta_1 &=\{\pm\delta_j\ \mid j\in[1, n]\}.
    \end{aligned}
\]

\noindent $B(m, n)$, $m>1$:
\[
    \begin{aligned}
    \Delta_0 &=\{\pm\epsilon_i\pm\epsilon_{i'}, \ \pm\epsilon_i \mid
   i, i'\in[1, m], i\ne i'\}\\
   &\cup\{\pm\delta_j\pm\delta_{j'}, \ \pm2\delta_j \mid j, j'\in[1, n], j\ne j'\}, \\
    \Delta_1 &=\{\pm\epsilon_i\pm\delta_j, \pm\delta_j\ \mid i\in[1, m], j\in[1, n]\},
    \end{aligned}
\]

\noindent $C(n+1)$:
\[
    \begin{aligned}
    \Delta_0 &=\{\pm\delta_j\pm\delta_{j'}, \ \pm2\delta_j \mid j, j'\in[1, n],  j\ne j'\}, \\
    \Delta_1 &=\{\pm\epsilon_1\pm\delta_j\mid j\in[1, n]\}.
    \end{aligned}
    \]

\noindent $D(m, n)$, $m>1$:
\[
    \begin{aligned}
    \Delta_0 =&\{\pm\epsilon_i\pm\epsilon_{i'} \mid i, i'\in[1, m], i\ne i'\}\\
            &\cup\{\pm\delta_j\pm\delta_{j'}, \ \pm2\delta_j \mid j, j'\in [1, n] j\ne j'\}, \\
    \Delta_1 =&\{\pm\epsilon_i\pm\delta_j\mid i\in[1, m], j\in[1, n]\}.
    \end{aligned}
\]

\noindent $F(4)$:
\[
    \begin{aligned}
    \Delta_0 &=\{\pm\epsilon_i\pm\epsilon_j, \ \pm\epsilon_i \mid i, j=1, 2, 3, \ i\ne j\}
    \cup\{\pm\delta \}, \\
    \Delta_1 &=\left\{\frac{1}{2}\big(\pm\epsilon_1\pm\epsilon_2\pm\epsilon_3\pm\delta\big)\right\},\\
    & (\delta, \delta)=-6, \quad (\epsilon_i, \epsilon_j)=2\delta_{i j},
    \quad (\epsilon_i, \delta)=0, \quad \forall i, j=1, 2, 3.
    \end{aligned}
\]

\noindent $G(3)$:
\[
    \begin{aligned}
    \Delta_0 =& \{\epsilon_i-\epsilon_j,
        \pm(2\epsilon_k-\epsilon_i-\epsilon_j)
    \mid  1\le i, j, k\le 3, \text{pairwise distinct} \} \\
            &\cup\{\pm2\delta \},\\
    \Delta_1 =&\{\pm\delta+(\epsilon_i-\epsilon_j), \ \pm\delta \mid i\ne j\},\\
&(\delta, \delta)=-2, \ (\epsilon_i, \epsilon_,)=\delta_{i j}, \
(\epsilon_i, \delta)=0, \
 \forall i, j=1, 2, 3.
\end{aligned}
\]

\noindent $D(2, 1; \alpha)$, $\alpha\in\C\backslash\{0, -1\}$:
\[
\begin{aligned}
    \Delta_0 &=\{\pm2\epsilon_i\mid \ i=1, 2\} \cup\{\pm2\delta \}, \\
    \Delta_1 &=\{\pm\delta\pm\epsilon_1\pm\epsilon_2\},\\
    &(\epsilon_1, \epsilon_1)=1, \quad (\epsilon_2, \epsilon_2)=\alpha,
    \quad (\delta, \delta)=-(1+\alpha), \quad (\epsilon_i,
    \delta)=0, \ \forall i.
\end{aligned}
\]

Denote by $\Pi=\{\alpha_1,\dots, \alpha_r\}$ the set of simple roots
of $\fg$ elative to the distinguished Borel subalgebra.  We have
\[
\begin{aligned}
&A(m|n):\quad
\Pi=\{\epsilon_1-\epsilon_2, \dots, \epsilon_m-\epsilon_{m+1},
    \epsilon_{m+1}-\delta_1, \delta_1-\delta_2, \dots,
    \delta_n-\delta_{n+1}\};\\
&B(0, n):\quad
    \Pi=\{\delta_1-\delta_2, \dots, \delta_{n-1}-\delta_n,   \delta_n\}; \\
&B(m, n), m>1:\\
&\Pi=\{ \delta_1-\delta_2, \dots,  \delta_{n-1}-\delta_n, \delta_n
-\epsilon_1,  \epsilon_1-\epsilon_2, \dots,
    \epsilon_{m-1}-\epsilon_m, \  \epsilon_m\};\\
&C(n+1):\quad
  \Pi=\{\epsilon_1-\delta_1, \  \delta_1-\delta_2, \dots,
\delta_{n-1}-\delta_n, 2\delta_n\};\\
&D(m, n), m>1:\\
&\Pi=\{\delta_1-\delta_2, \dots, \delta_{n-1}-\delta_n,
\delta_n-\epsilon_1, \ \epsilon_1-\epsilon_2,
    \epsilon_2-\epsilon_3, \dots,   \epsilon_{m-1}-\epsilon_m,
    \epsilon_{m-1}+\epsilon_m\};\\
&F(4):\quad
 \Pi=\left\{\frac{1}{2}(\epsilon_1+\epsilon_2+\epsilon_3+\delta), \
    -\epsilon_1, \  \epsilon_1-\epsilon_2, \ \epsilon_2-\epsilon_3\right\};
    \end{aligned}
\]
\[
\begin{aligned}
&G(3):\quad
 \Pi=\{\delta-\epsilon_1+\epsilon_3,\ \epsilon_1-\epsilon_2, \
        2\epsilon_2-\epsilon_1-\epsilon_3\};\\
&D(2, 1; \alpha), \alpha\in\C\backslash\{0, -1\}:\quad
 \Pi=\{\delta-\epsilon_1-\epsilon_2,\  2\epsilon_1,\
    2\epsilon_2\}.
\end{aligned}
\]
Note that there is a unique simple root, which we denote by $\alpha_s$, in each $\Pi$.
Thus $\Theta=\{s\}$.

\medskip

The simple roots relative to other Borel subalgebras can be obtained
by using odd reflections \cite{Se}. Let $\Pi_\fb=\{\alpha_1, \dots,
\alpha_r\}$ be the set of simple roots relative to a given Borel
subalgebra $\fb\subset \fg$. Take any isotropic odd simple root
$\alpha_t\in\Pi_\fb$, and define the odd reflection $s_t$ by
\[
\begin{aligned}
&s_t(\alpha_t)= -\alpha_t, \\
&s_t(\alpha_i)= \alpha_i + \alpha_t, &&\quad \text{if $i\ne t$
and $a_{i t}\ne 0$}, \\
&s_t(\alpha_i)= \alpha_i, &&\quad \text{if $i\ne t$ and $a_{i t}= 0$}.
\end{aligned}
\]
Then
$s_t(\Pi_\fb)=\{s_t(\alpha_1), \dots, s_t(\alpha_r)\}$ is the set of
simple roots relative to another Borel subalgebra, which is not
Weyl group conjugate to $\fb$. Further odd reflections can be
defined with respect to isotropic roots in $s_t(\Pi_\fb)$,
which turn $s_t(\Pi_\fb)$ into
sets of simple roots relative to other Borel
subalgebras. All the distinct sets obtained this way
correspond bijectively to the conjugacy classes of Borel
subalgebras.

\subsection{Dynkin diagrams}\label{append-diagrams}
\subsubsection{Dynkin diagrams in distinguished root systems}

The Dynkin diagrams in the distinguished
root systems are listed in Table 1 below, where $r$ is the number of nodes
and $s$ is the element of $\Theta$.
Note that the form of Dynkin diagrams in the distinguished
root systems is quite uniform in the literature.
Table 1 is essentially the corresponding table in
\cite{Kac1} with a slight modification in the Dynkin diagram for
$D(2, 1; \alpha)$.

\medskip

\begin{center}
{\bf Table 1. Dynkin diagrams in distinguished root systems}
\end{center}

\begin{picture}(400, 10)(0, 0)
\put(0, 0){\line(1, 0){390}}
\end{picture}


\begin{picture}(400, 10)(0, 5)
\put(5, 0){Lie superalgebra} \put(150, 0){Dynkin Diagram}
\put(320, 0){r} \put(370, 0){s}
\end{picture}

\begin{picture}(400, 20)(0, 0)
\put(0, 10){\line(1, 0){390}}
\end{picture}

\begin{picture}(400, 20)(0, 0)

\put(0, 10){A(m, n)}

\put(110, 10){\circle{10}}
\put(115, 10){\line(1, 0){10}}
\put(125, 10){...}
\put(135, 10){\line(1, 0){10}}
\put(150, 10){\circle{10}}
\put(155, 10){\line(1, 0){20}}
{ \color{gray} \put(175, 10){\circle*{10}} }
\put(180, 10){\line(1, 0){20}}
\put(205, 10){\circle{10}}
\put(210, 10){\line(1, 0){10}}
\put(220, 10){...}
\put(230, 10){\line(1, 0){10}}
\put(245, 10){\circle{10}}

\put(310, 10){m+n+1}
\put(360, 10){m+1}
\end{picture}

\begin{picture}(400, 30)(0, 0)

\put(0, 10){B(m, n), \ m$>$0}

\put(110, 10){\circle{10}}
\put(115, 10){\line(1, 0){10}}
\put(125,
10){...} \put(135, 10){\line(1, 0){10}}
\put(150, 10){\circle{10}}
\put(155, 10){\line(1, 0){20}}
{ \color{gray} \put(175,
10){\circle*{10}} }
\put(180, 10){\line(1, 0){20}}
\put(205, 10){\circle{10}}
\put(210, 10){\line(1, 0){10}}
\put(220, 10){...}
\put(230, 10){\line(1, 0){10}}
\put(245, 10){\circle{10}}
\put(250, 9){\line(1, 0){20}}
\put(250, 11){\line(1, 0){20}}
\put(260, 7){$>$}
\put(275, 10){\circle{10}}

\put(315, 10){m+n}
\put(365, 10){n}
\end{picture}

\begin{picture}(400, 30)(0, 0)

\put(0, 10){B(0, n)}

\put(110, 10){\circle{10}}
\put(115, 10){\line(1, 0){20}}
\put(140,
10){\circle{10}}
\put(145, 10){\line(1, 0){10}}
\put(155, 10){...}
\put(165, 10){\line(1, 0){10}}
\put(180, 10){\circle{10}}
\put(185, 10){\line(1, 0){20}}
\put(210, 10){\circle{10}}
\put(215, 9){\line(1, 0){20}}
\put(215, 11){\line(1, 0){20}}
\put(225, 7){$>$}
\put(240, 10){\circle*{10}}

\put(320, 10){n}
\put(370, 10){n}
\end{picture}

\begin{picture}(400, 30)(0, 0)

\put(0, 10){C(n), \ n$>$2}

{\color{gray} \put(110, 10){\circle*{10}} }
\put(115, 10){\line(1, 0){20}}
\put(140, 10){\circle{10}}
\put(145, 10){\line(1, 0){10}}
\put(155, 10){...}
\put(165, 10){\line(1, 0){10}}
\put(180, 10){\circle{10}}
\put(185, 10){\line(1, 0){20}}
\put(210, 10){\circle{10}}
\put(215, 9){\line(1, 0){20}}
\put(215, 11){\line(1, 0){20}}
\put(215, 7){$<$}
\put(240, 10){\circle{10}}

\put(320, 10){n}
\put(370, 10){1}
\end{picture}

\begin{picture}(400, 30)(0, 0)

\put(0, 10){D(m, n), \ m$>$1}

\put(110, 10){\circle{10}}
\put(115, 10){\line(1, 0){10}}
\put(125,
10){...}
\put(135, 10){\line(1, 0){10}}
\put(150, 10){\circle{10}}
\put(155, 10){\line(1, 0){20}}
{\color{gray} \put(180, 10){\circle*{10}} }
\put(185, 10){\line(1, 0){20}}
\put(210, 10){\circle{10}}
\put(215, 10){\line(1, 0){10}}
\put(225, 10){...}
\put(235, 10){\line(1, 0){10}}
\put(250, 10){\circle{10}}
\put(255, 10){\line(2, 1){20}}
\put(255, 10){\line(2, -1){20}}
\put(280, 20){\circle{10}}
\put(280, 0){\circle{10}}

\put(315, 10){m+n}
\put(370, 10){n}
\end{picture}

\begin{picture}(400, 30)(0, 0)
\put(0, 10){F(4)}
{ \color{gray} \put(140, 10){\circle*{10}} }
\put(145, 10){\line(1, 0){20}}

\put(170, 10){\circle{10}}
\put(175, 9){\line(1, 0){20}}
\put(175, 11){\line(1, 0){20}}
\put(175, 7){$<$}

\put(200, 10){\circle{10}}
\put(205, 10){\line(1, 0){20}}
\put(230, 10){\circle{10}}

\put(320, 10){4}
\put(370, 10){1}
\end{picture}

\begin{picture}(400, 30)(0, 0)

\put(0, 10){G(3)}

{ \color{gray} \put(150, 10){\circle*{10}} }
\put(155, 10){\line(1, 0){20}}
\put(180, 10){\circle{10}}
\put(185, 7){$<$}
\put(185, 11.5){\line(1, 0){20}}
\put(185, 10){\line(1, 0){20}}
\put(185, 8.5){\line(1, 0){20}}
\put(210, 10){\circle{10}}

\put(320, 10){3}
\put(370, 10){1}
\end{picture}

\begin{picture}(400, 40)(0, 0)

\put(0, 10){D(2, 1; $\alpha$)}
{ \color{gray} \put(165, 10){\circle*{10}} }
\put(170, 10){\line(2, 1){20}}
\put(170, 20){\tiny $-1$}
\put(170, -10){\tiny $-\alpha$}
\put(170, 10){\line(2, -1){20}}
\put(195, 20){\circle{10}}
\put(195, 0){\circle{10}}
\put(320, 10){3} \put(370, 10){1}
\end{picture}

\begin{picture}(400, 25)(0, 0)
\put(0, 10){\line(1, 0){390}}
\end{picture}

\subsubsection{Dynkin diagrams in non-distinguished root systems}

Table 2 gives the Dynkin diagrams of the non-distinguished root systems.
A nice graphical explanation can be found in \cite[\S 4]{CCLL} (see also \cite{FSS})
on how to obtain the Dynkin diagrams in Table 2 by applying odd reflections to those
in Table 1.

\medskip

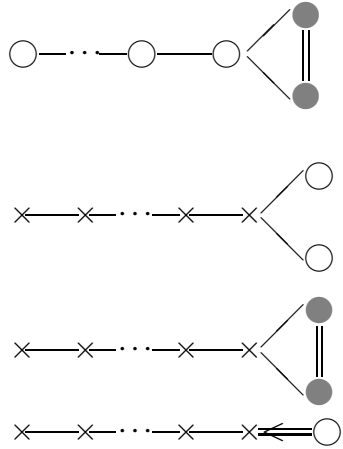
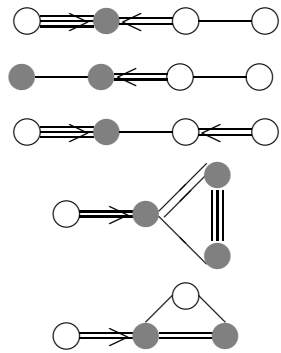
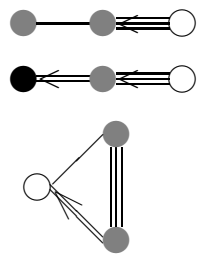
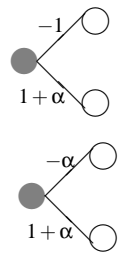
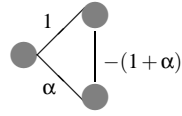
\begin{center}
{\bf Table 2. Dynkin diagrams in non-distinguished root systems}
\end{center}

\begin{picture}(400, 15)(0, -5)
\put(0, 0){\line(1, 0){390}}
\end{picture}

\begin{picture}(400, 10)(0, 0)
\put(30, 0){Lie superalgebra}
\put(200, 0){Dynkin Diagram}
\end{picture}

\begin{picture}(400, 10)(0, 0)
\put(0, 0){\line(1, 0){390}}
\end{picture}

\begin{picture}(300, 20)(0,5)
\put(50, 0){A(m, n)}
\put(180, -3){$\times$}
\put(190, 0){\line(1, 0){20}}
\put(210, -3){$\times$}
\put(220, 0){\line(1, 0){10}}
\put(230, 0){\dots}
\put(240, 0){\line(1, 0){10}}
\put(250, -3){$\times$}
\put(260, 0){\line(1, 0){20}}
\put(280, -3){$\times$}
\end{picture}

\begin{picture}(400, 30)(0,5)
\put(50, 0){B(m, n),  $\rm{m}>0$}
\put(180, -3){$\times$}
\put(190, 0){\line(1, 0){20}}
\put(210, -3){$\times$}
\put(220, 0){\line(1, 0){10}}
\put(230, 0){\dots}
\put(240, 0){\line(1, 0){10}}
\put(250, -3){$\times$}
\put(260, 0){\line(1, 0){20}}
\put(280, -3){$\times$}
\put(290, 1){\line(1, 0){20}}
\put(290, -1){\line(1, 0){20}}
\put(300, -3.2){$>$}
\put(315, 0){\circle{10}}
\end{picture}

\begin{picture}(400, 20)(0,5)
\put(180, -3){$\times$}
\put(190, 0){\line(1, 0){20}}
\put(210, -3){$\times$}
\put(220, 0){\line(1, 0){10}}
\put(230, 0){\dots}
\put(240, 0){\line(1, 0){10}}
\put(250, -3){$\times$}
\put(260, 0){\line(1, 0){20}}
\put(280, -3){$\times$}
\put(290, 1){\line(1, 0){20}}
\put(290, -1){\line(1, 0){20}}
\put(300, -3.2){$>$}
\put(315, 0){\circle*{10}}
\end{picture}

\begin{picture}(400, 40)(0,5)
\put(50, -3){C(n)}
\put(150, 0){\circle{10}}
\put(155, 0){\line(1, 0){10}}
\put(165, 0){\dots}
\put(175, 0){\line(1, 0){10}}
\put(190, 0){\circle{10}}
\put(195, 0){\line(1, 0){20}}
{\color{gray} \put(220, 0){\circle*{10}}}
\put(225, 0){\line(1, 0){20}}
{\color{gray} \put(250, 0){\circle*{10}}}
\put(255, 0){\line(1, 0){20}}
\put(280, 0){\circle{10}}
\put(285, 0){\line(1, 0){10}}
\put(295, 0){\dots}
\put(305, 0){\line(1, 0){10}}
\put(320, 0){\circle{10}}
\put(325, 1){\line(1, 0){20}}
\put(325, -1){\line(1, 0){20}}
\put(350, 0){\circle{10}}
\put(330, -3.2){$<$}
\end{picture}

\begin{picture}(400, 30)(-185,5)
\put(0, 0){\circle{10}}
\put(6, 0){\line(1, 0){10}}
\put(17, 0){\dots}
\put(29, 0){\line(1, 0){10}}
\put(45, 0){\circle{10}}
\put(51, 0){\line(1, 0){20}}
\put(77, 0){\circle{10}}
\put(101, 17){\line(-1, -1){16}}
\put(101, -17){\line(-1, 1){16}}
{\color{gray}\put(107, 15){\circle*{10}}}
\put(106,-10){\line(0, 1){19}}
\put(108,-10){\line(0, 1){19}}
{\color{gray}\put(107, -16){\circle*{10}}}
\end{picture}

\begin{picture}(400, 60)(-180,5)
\put(-130, -3){D(m, n), $\rm m>1$}
\put(0, -3){$\times$}
\put(6, 0){\line(1, 0){20}}
\put(24, -3){$\times$}
\put(30, 0){\line(1, 0){10}}
\put(41, 0){\dots}
\put(54, 0){\line(1, 0){10}}
\put(62, -3){$\times$}
\put(68, 0){\line(1, 0){20}}
\put(86, -3){$\times$}
\put(117, 15){\circle{10}}
\put(111, 17){\line(-1, -1){16}}
\put(111, -17){\line(-1, 1){16}}
\put(117, -16){\circle{10}}
\end{picture}

\begin{picture}(400, 50)(-180,5)
\put(0, -3){$\times$}
\put(6, 0){\line(1, 0){20}}
\put(24, -3){$\times$}
\put(30, 0){\line(1, 0){10}}
\put(41, 0){\dots}
\put(54, 0){\line(1, 0){10}}
\put(62, -3){$\times$}
\put(68, 0){\line(1, 0){20}}
\put(86, -3){$\times$}
\put(111, 17){\line(-1, -1){16}}
\put(111, -17){\line(-1, 1){16}}
{\color{gray}\put(117, 15){\circle*{10}}}
\put(116,-10){\line(0, 1){19}}
\put(118,-10){\line(0, 1){19}}
{\color{gray}\put(117, -16){\circle*{10}}}
\end{picture}

\begin{picture}(400, 30)(-180,5)
\put(0, -3){$\times$}
\put(6, 0){\line(1, 0){20}}
\put(24, -3){$\times$}
\put(30, 0){\line(1, 0){10}}
\put(41, 0){\dots}
\put(54, 0){\line(1, 0){10}}
\put(62, -3){$\times$}
\put(68, 0){\line(1, 0){20}}
\put(86, -3){$\times$}
\put(94, 1){\line(1, 0){20}}
\put(94, -1){\line(1, 0){20}}
\put(95, -3.2){$<$}
\put(120, 0){\circle{10}}
\end{picture}

\begin{picture}(400, 40)(-185,5)
\put(-130, 0){$F(4)$}
\put(0, 0){\circle{10}}
\put(5, 2){\line(1, 0){20}}
\put(5, 0){\line(1, 0){20}}
\put(5, -2){\line(1, 0){20}}
\put(15, -3.25){$>$}
{\color{gray}\put(30, 0){\circle*{10}}}
\put(35, 1){\line(1, 0){20}}
\put(35, -1){\line(1, 0){20}}
\put(35, -3.25){$<$}
\put(60, 0){\circle{10}}
\put(65, 0){\line(1, 0){20}}
\put(90, 0){\circle{10}}
\end{picture}

\begin{picture}(400, 20)(-180,5)
{\color{gray}\put(0, 0){\circle*{10}}}
\put(5, 0){\line(1, 0){20}}
{\color{gray}\put(30, 0){\circle*{10}}}
\put(35, 1){\line(1, 0){20}}
\put(35, -1){\line(1, 0){20}}
\put(35, -3.25){$<$}
\put(60, 0){\circle{10}}
\put(65, 0){\line(1, 0){20}}
\put(90, 0){\circle{10}}
\end{picture}

\begin{picture}(400, 20)(-185,5)
\put(0, 0){\circle{10}}
\put(5, 2){\line(1, 0){20}}
\put(5, 0){\line(1, 0){20}}
\put(5, -2){\line(1, 0){20}}
\put(15, -3.25){$>$}
{\color{gray}\put(30, 0){\circle*{10}}}
\put(35, 0){\line(1, 0){20}}
\put(60, 0){\circle{10}}
\put(65, 1){\line(1, 0){20}}
\put(65, -1){\line(1, 0){20}}
\put(65, -3.25){$<$}
\put(90, 0){\circle{10}}
\end{picture}

\begin{picture}(400, 30)(-155,5)
\put(45, 0){\circle{10}}
\put(50, 1){\line(1, 0){20}}
\put(50, -1){\line(1, 0){20}}
\put(60, -3.5){$>$}
{\color{gray}\put(75, 0){\circle*{10}}}
\put(98, 19){\line(-1, -1){18}}
\put(98, 15){\line(-1, -1){16}}
\put(98, -19){\line(-1, 1){18}}
{\color{gray}\put(102, 15){\circle*{10}}}
\put(100,-10){\line(0, 1){19}}
\put(102,-10){\line(0, 1){19}}
\put(104,-10){\line(0, 1){19}}
{\color{gray}\put(102, -16){\circle*{10}}}
\end{picture}

\begin{picture}(400, 45)(-170,5)
\put(30, 0){\circle{10}}
\put(35, 1){\line(1, 0){20}}
\put(35, -1){\line(1, 0){20}}
\put(45, -3.5){$>$}
{\color{gray}\put(60, 0){\circle*{10}}}
\put(60, 5){\line(1, 1){10}}
\put(65, 1){\line(1, 0){20}}
\put(65, -1){\line(1, 0){20}}
{\color{gray}\put(90, 0){\circle*{10}}}
\put(90, 5){\line(-1, 1){10}}
\put(75, 15){\circle{10}}
\end{picture}

\begin{picture}(400, 50)(-170,5)
\put(-120, 0){$G(3)$}
{\color{gray}\put(30, 0){\circle*{10}}}
\put(35, 0){\line(1, 0){20}}
{\color{gray}\put(60, 0){\circle*{10}}}
\put(65, 2){\line(1, 0){20}}
\put(65, 0){\line(1, 0){20}}
\put(65, -2){\line(1, 0){20}}
\put(65, -3.25){$<$}
\put(90, 0){\circle{10}}
\end{picture}

\begin{picture}(400, 20)(-170,5)
\put(30, 0){\circle*{10}}
\put(35, 1){\line(1, 0){20}}
\put(35, -1){\line(1, 0){20}}
\put(35, -3.25){$<$}
{\color{gray}\put(60, 0){\circle*{10}}}
\put(65, 2){\line(1, 0){20}}
\put(65, 0){\line(1, 0){20}}
\put(65, -2){\line(1, 0){20}}
\put(65, -3.25){$<$}
\put(90, 0){\circle{10}}
\end{picture}

\begin{picture}(400, 40)(-130,5)
\put(75, 0){\circle{10}}
\put(80, 1){\line(1, -1){20}}
\put(80, -1){\line(1, -1){20}}
\put(82, -2){\line(2, -1){10}}
\put(82, -2){\line(1, -2){5}}
{\color{gray}\put(105, -20){\circle*{10}}}

\put(103, 15){\line(0, -1){30}}
\put(105, 15){\line(0, -1){30}}
\put(107, 15){\line(0, -1){30}}

{\color{gray}\put(105, 20){\circle*{10}}}

\put(100, 20){\line(-1, -1){19}}
\end{picture}

\begin{picture}(400, 60)(-130,25)
\put(-80, 0){$D(2, 1; \alpha)$}
{\color{gray}\put(75, 0){\circle*{10}}}
\put(98, 18){\line(-1, -1){18}}
\put(98, -18){\line(-1, 1){18}}
\put(80, 11){\tiny $-1$}
\put(73, -16){\tiny $1+\alpha$}
\put(102, 15){\circle{10}}
\put(102, -16){\circle{10}}
\end{picture}

\begin{picture}(400, 50)(-130,25)
{\color{gray}\put(75, 0){\circle*{10}}}
\put(98, 18){\line(-1, -1){18}}
\put(98, -18){\line(-1, 1){18}}
\put(80, 11){\tiny $-\alpha$}
\put(73, -16){\tiny $1+\alpha$}
\put(102, 15){\circle{10}}
\put(102, -16){\circle{10}}
\end{picture}

\begin{picture}(400, 50)(-130, 25)
{\color{gray}\put(75, 0){\circle*{10}}}
\put(98, 18){\line(-1, -1){18}}
\put(98, -18){\line(-1, 1){18}}
\put(82, 10){\tiny $1$}
\put(82 , -15){\tiny $\alpha$}
{\color{gray}\put(102, 15){\circle*{10}}}
\put(102,-10){\line(0, 1){19}}
\put(105,-5){\tiny $-(1+\alpha)$}
{\color{gray}\put(102, -16){\circle*{10}}}
\end{picture}

\begin{picture}(400, 80)(0, -15)
\put(0, 10){\line(1, 0){380}}
\end{picture}

In the diagrams in Table 2, a node marked with $\times$ can be white or grey.
However, the precise rule for assigning colours requires the knowledge
of the simple roots, which are described below.

\smallskip

\noindent  $A(m, n)$.
An ordering $({\mathcal E}_1,
{\mathcal E}_2, \dots, {\mathcal E}_{m+n+2})$ of
$\epsilon_i$ and $\delta_j$
is called admissible if $\epsilon_i$ appears before
$\epsilon_{i+1}$ for all $i$ and $\delta_j$ before $\delta_{j+1}$
for all $j$. Each admissible ordering
corresponds to one Weyl group conjugate class of Borel subalgebras, with the
associated simple roots given by ${\mathcal E}_a - {\mathcal
E}_{a+1}$ $(1\le a\le m+n+1)$. In particular, the distinguished
Borel corresponds to the admissible ordering such that all the $\epsilon_i$
appear before the $\delta_j$. Let us define $[{\mathcal E}_a]$
$(a=1, 2, \dots, m+n+2)$ by $[{\mathcal E}_a]=0$ (resp. $[{\mathcal
E}_a]=1$) if ${\mathcal E}_a$ is some $\epsilon_i$ (resp.
$\delta_j$). The $a$-th node from the left in the Dynkin diagram is
associated with the simple root ${\mathcal E}_a - {\mathcal
E}_{a+1}$, which is white if $[{\mathcal E}_a]=[{\mathcal E}_{a+1}]$
and grey otherwise.

\medskip

\noindent $B(m, n)$, $m>0$. Let $({\mathcal E}_1,
{\mathcal E}_2, \dots, {\mathcal E}_{m+n})$ be an admissible ordering of
$\epsilon_i$ $(i=1, \dots, m)$ and $\delta_j$ $(j=1, \dots, n)$.
Then the corresponding simple roots are
\[
{\mathcal E}_1 - {\mathcal E}_2, \dots, {\mathcal E}_{m+n-1} -
{\mathcal E}_{m+n}, {\mathcal E}_{m+n}.
\]
The first Dynkin diagram corresponds to the case ${\mathcal
E}_{m+n}=\epsilon_m$. The $a$-th node ($a<m+n$) from the left is
associated with the simple root ${\mathcal E}_a- {\mathcal
E}_{a+1}$, which is white if $[{\mathcal E}_a]=[{\mathcal E}_{a+1}]$
and grey otherwise. The second Dynkin diagram corresponds to the
case ${\mathcal E}_{m+n}= \delta_n$. The colours of the nodes marked
$\times$ are assigned in the same way as in type $A$.

\medskip

\noindent $C(n)$. We have  already specified the colours of the nodes in
the Dynkin diagrams, but it is still useful to have an explicit description of
the simple roots. Let $({\mathcal E}_1,
{\mathcal E}_2, \dots, {\mathcal E}_n)$ be an admissible ordering of
$\delta_j$ $(j=1, \dots, n-1)$ and $\epsilon_1$. The first Dynkin diagram
corresponds to the case with ${\mathcal E}_n=\delta_{n-1}$, where
simple roots are given by
\[
{\mathcal E}_1 - {\mathcal E}_2, \dots, {\mathcal E}_{n-1} -
{\mathcal E}_n, 2{\mathcal E}_n.
\]
The second Dynkin diagram corresponds to the case with ${\mathcal
E}_n=\epsilon_1$, where the simple roots are given by
\[
{\mathcal E}_1 - {\mathcal E}_2, \dots, {\mathcal E}_{n-1} -
{\mathcal E}_n, {\mathcal E}_{n-1}+{\mathcal E}_n.
\]
The colours of the nodes marked with $\times$'s are assigned in the
same way as in type $A$ and type $B$.

\medskip

\noindent  $D(m, n)$. Let $({\mathcal E}_1,
{\mathcal E}_2, \dots, {\mathcal E}_{m+n})$ be an admissible ordering of
$\epsilon_i$ $(i=1, \dots, m)$ and $\delta_j$ $(j=1, \dots, n)$.
If ${\mathcal E}_{m+n-1}=\epsilon_{m-1}$ and ${\mathcal
E}_{m+n}=\epsilon_{m}$,  or ${\mathcal E}_{m+n-1}=\delta_n$ and
${\mathcal E}_{m+n}=\epsilon_m$, the simple roots are given by
\[
{\mathcal E}_1 - {\mathcal E}_2, \dots, {\mathcal E}_{m+n-1} -
{\mathcal E}_{m+n},  {\mathcal E}_{m+n-1} + {\mathcal E}_{m+n}.
\]
The first Dynkin diagram corresponds to the former case, while the
second Dynkin diagram corresponds to the latter.
If ${\mathcal E}_{m+n-1}=\delta_{n-1}$ and ${\mathcal E}_{m+n}=\delta_n$,
the simple roots are given by
\[
{\mathcal E}_1 - {\mathcal E}_2, \dots, {\mathcal E}_{m+n-1} -
{\mathcal E}_{m+n},  2{\mathcal E}_{m+n}.
\]
The third Dynkin diagram corresponds to this case.

We assign colours to the nodes marked with $\times$ in the
same way as in the other cases.

\begin{remark} \label{rem:C-D-diagrams}
There are at least three grey nodes in the
Dynkin diagrams of type $D(m, n)$ in Table 2, but
in each of the Dynkin diagrams of type $C(n)$, there are only
two grey nodes which are always next to each other.
\end{remark}

\section{Presentations of irreducible modules}\label{modules}

In general it is hard to give an explicit description of a finite dimensional
irreducible module for a Lie superalgebra as the quotient of a (generalised) Verma module
in a form similar to \cite[Theorem 21.4]{H} in the context of ordinary semi-simple Lie algebras.
However, this is possible in some special cases, e.g., the natural module
for $\mathfrak{gl}_{m|n}$ in arbitrary root systems as discussed in Section
\ref{A-nondistingushed}. Here are two further cases, which are used in the proof of Lemma \ref{key}.

\subsection{An irreducible $\mathfrak{osp}_{2|4}$-module}\label{osp24-module}

Let $\fg$ be the Lie superalgebra $\mathfrak{osp}_{2|4}$ with the
choice of Borel subalgebra corresponding to the Dynkin diagram
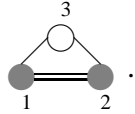
\begin{center}
\begin{picture}(60, 40)(60,-10)
{\color{gray}\put(60, 0){\circle*{10}}}
\put(60, 5){\line(1, 1){10}}
\put(65, 1){\line(1, 0){20}}
\put(65, -1){\line(1, 0){20}}
{\color{gray}\put(90, 0){\circle*{10}}}
\put(90, 5){\line(-1, 1){10}}
\put(75, 15){\circle{10}}

\put(60, -12){\tiny 1}
\put(90, -12){\tiny 2}
\put(75, 23){\tiny 3}

\put(100, 0){.}
\end{picture}
\end{center}
We present $\fg$ in the standard fashion using Chevalley generators
$e_i, f_i, h_i$ ($i=1, 2, 3$) and relations with the higher order Serre relations
being those associated with diagrma (\ref{HOS-g-6}) in Theorem \ref{Serre-g}. To be specific, we
denote by $\alpha_i$ the simple roots and take
\[
(\alpha_1, \alpha_3)=(\alpha_2, \alpha_3)=-1,
\quad (\alpha_1, \alpha_2)=2, \quad  (\alpha_3, \alpha_3)=2.
\]
Let $\overline{\fp}$ be the parabolic subalgebra generated by all
the generators but $e_1$. Then
$\overline{\fp}=\fl\oplus\overline{\fu}$ with
$\fl=\mathfrak{gl}_{2|1}$ and $\overline{\fu}$ spanned by
\[
\begin{array}{l l l}
\zeta_1:=e_1, &\quad \zeta_2:=[e_1, e_3], \\
X_1:=[e_1, e_2], &\quad X_2:=[[e_1, e_2], e_3],
&\quad X_3:=[[[e_1, e_2], e_3], e_3].
\end{array}
\]

Given the irreducible $\overline{\fp}$-module
$\overline{L}^0_\lambda= \C v_0$ with lowest weight $\lambda$ such
that
\[
(\lambda, \alpha_2)=0, \quad (\lambda, \alpha_3)=0, \quad (\lambda, \alpha_1)=-2,
\]
we construct the generalised Verma module
$\overline{V}_\lambda=U(\fg)\otimes_{U(\overline{\fp})}\overline{L}^0_\lambda$.
Then the maximal submodule $M_\lambda$ of $\overline{V}_\lambda$ is
given by
\begin{eqnarray}\label{generator-M-C}
M_\lambda=U(\fg)\zeta_1 X_1 v_0.
\end{eqnarray}
The irreducible quotient
$\overline{L}_\lambda=\overline{V}_\lambda/M_\lambda$ is
$10$-dimensional with a basis
\[
\begin{array}{l l l l l}
v_0, \quad & X_1 v_0, \quad & X_2 v_0, \quad & X_3 v_0, \quad  & X_1 X_3 v_0, \\
\zeta_1 v_0,  \quad &  \zeta_1 X_2 v_0, \quad & \zeta_1  X_3 v_0,
\quad  & \zeta_1 X_1 X_3 v_0, \quad &
\zeta_1 \zeta_2 v_0.
\end{array}
\]

\subsection{Graded symmetric tensor for $\mathfrak{gl}_{m|n}$ }\label{S2-gl}

Let $\fg=\mathfrak{gl}_{m|n}$ and set $r=m+n-1$. Choose an arbitrary
homogeneous basis for the natural module $\C^{m|n}$ with the last
element being odd. We regard $\fg$ as consisting of matrices relative
to this basis. Take the subalgebra consisting of the upper triangular matrices
as the Borel subalgebra, which corresponds to an admissible ordering
$({\mathcal E}_1, {\mathcal E}_2, \dots, {\mathcal E}_{m+n})$
of $\epsilon_i$ ($1\le i\le m$) and $\delta_j$ ($1\le j\le n$)
with ${\mathcal E}_{m+n}=\delta_n$. See Appendix \ref{append-diagrams}
for more details.

Let $\fl$, $\fu$ and $\overline\fu$ be
subalgebras respectively spanned by matrix units $e_{r+1, r+1}$ and
$e_{i j}$ with $1\le i, j\le r$,  by $e_{i, r+1}$ with $1\le r$,
and by $e_{r+1, i}$ with $1\le r$. Set
$\overline\fp=\fl\oplus\overline\fu$, which is a parabolic
subalgebra, and $\fg=\overline\fp\oplus\fu$.

For $\lambda=2\delta_n$, we consider the generalised Verma module
$\overline{V}_\lambda:=U(\fg)\otimes_{U(\overline\fp)}\C_\lambda$ of
lowest weight type, where $\C_\lambda$ denotes the irreducible
$\overline\fp$-module with lowest weight $\lambda$. Let $v_0$ denote
a generator of $\C_\lambda$, then
\begin{eqnarray}
\begin{aligned}
&f_r v_0=0, \\
&e_i v_0=0, \quad f_i v_0 =0, &\quad&  1\le i\le r-1,  \\
&e_{j j} v_0 =2\delta_{j, r+1} v_0, &\quad& 1\le j\le r+1,
\end{aligned}
\end{eqnarray}
where $e_i=e_{i, i+1}$ and $f_i=e_{i+1, i}$.

Now $\overline{V}_\lambda\cong U(\fu)\otimes\C_\lambda$ as
$\fl$-module, where $U(\fu)=S_s(\fu)$, the $\Z_2$-graded symmetric
algebra of $\fu$. This superalgebra has a $\Z$-grading with $\fu$
having degree $1$. It induces a natural $\Z$-grading on
$\overline{V}_\lambda$. The unique maximal submodule $M_\lambda$ of
$\overline{V}_\lambda$ is the direct sum of the homogeneous
subspaces of degrees greater than or equal to $3$, which is
generated by $U(\fu)_3\otimes\C_\lambda$, the homogeneous subspace
of degree $3$. The irreducible quotient $\overline{L}_\lambda$ of
$\overline{V}_\lambda$ is isomorphic to the $\Z_2$-graded symmetric
tensor of the natural $\fg$-module at rank $2$.

The natural $\fl$ action on $U(\fu)$ (obtained by generalising the
adjoint action) respects the $\Z$-grading. In the present case,
each homogeneous component is in fact an irreducible submodule. We
are interested in $U(\fu)_3$. If $u_3$ is a nonzero lowest weight
vector of $U(\fu)_3$, then $M_\lambda$ is generated over $\fg$ by
$u_3\otimes \C_\lambda$. The form of $u_3$ depends on the ordering
of the basis for $\C^{m|n}$. Denote by $E_{i j}\in U(\fg)$  the
image of $e_{i j}\in\fg$ under the natural embedding. The $u_3$ can
be expressed as follows:
\begin{eqnarray}
\begin{aligned}
&u_3=E_{r, r+1}^3, &\quad& \text{if $E_{r, r+1}$ is even}; \\
&u_3=E_{r-1, r+1}^2 E_{r, r+1},
&\quad& \text{if both $E_{r, r+1}$ and $E_{r-1, r}$ are odd};\\
&u_3=E_{r-2, r+1}E_{r-1, r+1} E_{r, r+1},
&\quad& \text{if  $E_{r, r+1}$ is odd but $E_{r-1, r}$ is even}.
\end{aligned}
\end{eqnarray}
\begin{remark}\label{D-small-r}
The third case becomes vacuous if $r= 2$;
and both the second and third cases are vacuous if $r= 1$.
\end{remark}

The irreducible quotient $\overline{L}_\lambda
=\overline{V}_\lambda/M_\lambda$  is isomorphic to the graded
skew symmetric rank two tensor $\wedge^2_s(\C^{m|n})$ of the natural
$\fg$-module.

\end{appendix}


\bigskip

\begin{thebibliography}{9999}
\bibitem{BGLL} S. Bouarroudj, P. Grozman, A. Lebedev, D. Leites,
Divided power (co)homology. Presentations of simple finite
dimensional modular Lie superalgebras with Cartan matrix. {\sl
Homology, Homotopy Appl. \bf 12} (2010), no. 1, 237 -- 278.

\bibitem{BGZ90} Bracken, A. J.; Gould, M. D.; Zhang, R. B.
Quantum supergroups and solutions of the Yang-Baxter equation. {\sl
Modern Phys. Lett. \bf A 5} (1990), no. 11, 831--840.

\bibitem{CCLL} D. Chapovalov, M. Chapovalov, A. Lebedev, D. Leites,
The classification of almost affine (hyperbolic) Lie superalgebras.
{\sl Journal of Nonlinear Mathematical Physics, \bf 1}, No. 1 (2009) 1 -- 55.

\bibitem{Dr} Drinfeld, V. G. Quantum groups. Proceedings of the
International Congress of Mathematicians, Vol. 1,
2 (Berkeley, Calif., 1986), 798--820, Amer. Math. Soc., Providence,
RI, 1987.

\bibitem{FLV} R. Floreanini, D. A. Leites, L. Vinet,
On the defining relations of quantum superalgebras.
{\sl Lett. Math. Phys. \bf 23} (1991), no. 2, 127 -- 131.

\bibitem{GL} P. Grozman, D. Leites,
Defining relations for Lie superalgebras with Cartan matrix.
{\sl Czechoslovak J. Phys.  \bf 51}  (2001),  no. 1, 1 -- 21.


\bibitem{GLP} Grozman, P.; Leites, D.; Poletaeva, E. Defining relations for
classical Lie superalgebras without Cartan matrices. The Roos
Festschrift volume, 2. {\sl Homology Homotopy Appl. \bf 4} (2002),
no. 2, part 2, 259 -- 275.

\bibitem{FSS} L. Frappat, A. Sciarrino, P. Sorba,
Structure of basic Lie superalgebras and of their affine extensions.
{\sl Comm. Math. Phys. \bf 121}  (1989),  no. 3, 457 -- 500.

\bibitem{GK} O. Gabber and V. G. Kac,
On defining relations of certain infinite-dimensional Lie algebras.
{\sl Bull. Amer. Math. Soc. (N.S.)  \bf 5}  (1981), no. 2, 185 -- 189.

\bibitem{G} N. Geer, Etingof-Kazhdan quantization of Lie superbialgebras.
{\sl Adv. Math. \bf 207} (2006), no. 1, 1 -- 38.

\bibitem{H} J. E. Humphreys, Introduction to Lie algebras and representation
theory. Graduate Texts in Mathematics, {\bf 9}. Springer-Verlag, New
York-Berlin, 1972.

\bibitem{Jim} M. Jimbo, A $q$-difference analogue of $\Uq(\fg)$ and the Yang-Baxter
equation. {\sl Lett. Math. Phys. \bf 10} (1985), no. 1, 63 -- 69.

\bibitem{Kac1} V. G. Kac, Lie superalgebra.
{\sl Advances in Math.  \bf 26}  (1977), no. 1, 8 -- 96.

\bibitem{Kac2} V. G. Kac, Representations of classical Lie superalgebras.
Differential geometrical methods in mathematical physics, II
(Proc. Conf., Univ. Bonn, Bonn, 1977),  pp. 597 -- 626,
{\sl Lecture Notes in Math., \bf 676}, Springer, Berlin, 1978.

\bibitem{Kac3} V. G. Kac, Infinite-dimensional Lie algebras.
Third edition. Cambridge University Press, Cambridge, 1990.

\bibitem{K} I. Kaplansky, Superalgebras. {\sl Pacific J. Math. \bf 86} (1980), no. 1, 93 -- 98.
\bibitem{K-collect} I. Kaplansky, Afterthought: superalgebras,
in Selected papers and other writings. With an introduction by Hyman Bass.
Springer-Verlag, New York, 1995. p225.

\bibitem{Lan} E. Lanzmann, The Zhang transformation and $\rm{U}_q({\rm osp}(1,2l))$-Verma modules annihilators.
{\sl Algebr. Represent. Theory \bf 5} (2002), no. 3, 235 -- 258.

\bibitem{LS} D. Leites and V. Serganova, Defining relations for classical
Lie superalgebras. I. Superalgebras with Cartan matrix or
Dynkin-type diagram.  Topological and geometrical methods in field
theory (Turku, 1991),  194 - 201, World Sci. Publ., River Edge, NJ,
1992.

\bibitem{LGZ93} Links, J. R.; Gould, M. D.; Zhang, R. B.
Quantum supergroups, link polynomials and representation of the
braid generator. {\sl Rev. Math. Phys. 5} (1993), no. 2, 345 -- 361.

\bibitem{L} G. Lusztig, Introduction to quantum groups. Progress in
Mathematics, {\bf 110}. Birkhäuser Boston, Inc., Boston, MA, 1993.

\bibitem{Man88} Manin, Yu. I. Quantum groups and noncommutative geometry.
Université de Montréal, Centre de Recherches Mathématiques, Montreal, QC, 1988.

\bibitem{Sch79} M. Scheunert, The theory of Lie superalgebras. An introduction.
{\sl Lecture Notes in Mathematics, \bf 716}. Springer, Berlin, 1979.

\bibitem{Sch93} Scheunert, M. The presentation and $q$
deformation of special linear Lie superalgebras. {\sl J. Math. Phys. \bf 34}
(1993), no. 8, 3780 -- 3808.

\bibitem{Se}Serganova V., Automorphisms of simple Lie superalgebras. {\sl Izv.
Akad. Nauk SSSR Ser. Mat. \bf 48} (1984), no. 3, 585–598; (Russian)
English translation: {\sl Math. USSR-Izv. \bf 24} (1985), no. 3,
539 -- 551

\bibitem{Y1} H. Yamane, Quantized enveloping algebras associated
with simple Lie superalgebras and their universal $R$-matrices.
{\sl Publ. Res. Inst. Math. Sci. \bf 30}(1994), no. 1, 15 -- 87.

\bibitem{Y2} H. Yamane, On defining relations of affine Lie superalgebras
and affine quantized universal enveloping superalgebras. {\sl Publ.
Res. Inst. Math. Sci. \bf 35} (1999), no. 3, 321 -- 390.

\bibitem{Z92}  Zhang, R. B. Finite-dimensional representations of $U_q({\rm osp}(1/2n))$
 and its connection with quantum ${\rm so}(2n+1)$. {\sl Lett. Math. Phys. \bf 25} (1992),
 no. 4, 317 -- 325.

\bibitem{Z95} Zhang, R. B. Quantum supergroups and topological invariants of
three-manifolds.
{\sl Rev. Math. Phys. \bf 7} (1995), no. 5, 809 -- 831.

\bibitem{Z98} Zhang, R. B. Structure and representations of the
quantum general linear supergroup.
{\sl Comm. Math. Phys. \bf 195} (1998), no. 3, 525 -- 547.

\bibitem{Z04} Zhang, R. B. Quantum superalgebra representations on
cohomology groups of non-commutative bundles.
{\sl J. Pure Appl. Algebra \bf 191} (2004), no. 3, 285 -- 314.

\end{thebibliography}
\end{document}